%% file: SSCN_arxiv.tex
\documentclass[10pt]{article}
\usepackage{geometry} 
\usepackage{natbib, algorithm, algorithmic,hyperref}

\usepackage{microtype}
\usepackage{graphicx}
\usepackage{subfigure}
\usepackage{booktabs} 
 \usepackage{enumitem}
 \usepackage{nicefrac}

\usepackage{authblk}

\graphicspath{{./Code/plots/}}
\input{everything.tex}

\begin{document}

\title{\bf Stochastic Subspace Cubic Newton Method}

\author[1]{Filip Hanzely}
\author[2]{Nikita Doikov}
\author[1]{Peter Richt\'{a}rik}
\author[2]{Yurii Nesterov}

\affil[1]{King Abdullah University of Science and Technology, Thuwal, Saudi Arabia}
\affil[2]{Catholic University of Louvain, Louvain-la-Neuve, Belgium}

\date{February 21, 2020}
 
\maketitle

\begin{abstract}
In this paper, we propose a new randomized second-order optimization algorithm---Stochastic Subspace Cubic Newton (SSCN)---for minimizing a high dimensional convex function $f$. Our method can be seen both as a {\em stochastic} extension of the cubically-regularized Newton method of Nesterov and Polyak (2006), and a {\em second-order} enhancement of stochastic subspace descent of Kozak et al. (2019). We prove that as we vary the minibatch size, the global convergence rate of SSCN interpolates between the rate of stochastic coordinate descent (CD) and the rate of cubic regularized Newton, thus giving new insights into the connection between first and second-order methods. Remarkably, the local convergence rate of SSCN matches the rate of stochastic subspace descent applied to the problem of minimizing the quadratic function $\frac12 (x-x^*)^\top \nabla^2f(x^*)(x-x^*)$, where $x^*$ is the minimizer of $f$, and hence depends on the properties of $f$ at the optimum only. Our numerical experiments show that SSCN outperforms non-accelerated first-order CD algorithms while being competitive to their accelerated variants.
\end{abstract}

\tableofcontents

\section{Introduction}

In this work we consider the optimization problem
\begin{equation}\label{eq:problem}
    \min \limits_{x \in \R^d}  \left\{ F(x) \eqdef f(x) + \psi(x) \right\}\, ,
\end{equation}
where $f:\R^d\to \R$ is convex and twice differentiable and $\psi:\R^d\to \R\cup \{+\infty\}$ is a simple convex function. We are interested in the regime where the dimension $d$ is very large, which arises in many contexts, such as  the training of modern over-parameterized machine learning models. In this regime, coordinate descent (CD) methods, or more generally subspace descent methods, are the methods of choice. 

\subsection{Subspace descent methods}

Subspace descent methods rely on  update rules of the form
\begin{equation} \label{eq:update_general}
x^+ = x + \mS h,\qquad \mS \in \R^{d\times \tau(\mS)}, \qquad h\in \R^{\tau(\mS)},
\end{equation}
where $\mS$ is a thin matrix, typically with a negligible number of columns compared to the dimension (i.e., $\tau(\mS)\ll d$). That is, they move from $x$ to $x^+$ along the subspace spanned by the columns of $\mS$.

In these methods,  the subspace matrix $\mS$ is typically chosen first, followed by the determination of the parameters $h$ which define the linear combination of the columns determining the update direction. Several different rules have been proposed in the literature for choosing the matrix $\mS$, including greedy,  cyclic and randomized rules. In this work we consider a \textit{randomized} rule. In particular, we assume that $\mS$ is sampled from an arbitrary but fixed distribution $\cD$ restricted to requiring that $\mS$ be of full column rank\footnote{It is rather simple to extend our results to matrices $\mS$ which are column-rank deficient. However, this would introduce a rather heavy notation burden which we decided to avoid for the sake of clarity and readability.} with probability one. 

Once $\mS\sim \cD$ is sampled, a rule for deciding  the stepsize $h$ varies from algorithm to algorithm, but is mostly determined by the underlying \emph{oracle model} for information access to function $f$. For instance, first-order methods require  access to the subspace gradient $$\nabla_{\mS} f(x) \eqdef  \mS^\top \nabla f(x),$$ and are relatively well studied~\citep{nesterov2012efficiency, stich2013optimization, richtarik2014iteration, wright2015coordinate, SSD:Kozak2019}. At the other extreme are  variants performing a full subspace minimization, i.e.,  $f$ is minimized over the affine subspace given by $$\{ x+\mS h \, | \, h\in \R^{\tau(\mS)}\};$$ see \citep{chang2008coordinate}.  In particular, in this paper we are interested in the \emph{second-order} oracle model; i.e., we claim access both to the subspace gradient $\nabla_{\mS} f(x)$ and the subspace Hessian $$\nabla^2_{\mS} f(x)\eqdef \mS^\top \nabla^2f(x)\mS.$$

\subsection{Contributions}

We now summarize our contributions:

\begin{itemize}
\item[(a)] {\bf New 2nd order subspace method.} We propose a new stochastic subspace method---Stochastic Subspace Cubic Newton (SSCN)---constructed by minimizing an oracle-consistent global upper bound on the objective $f$ in each iteration (Section~\ref{sec:Alg}). This bound is formed using both the subspace gradient and the subspace Hessian at the current iterate and relies on Lipschitzness of the subspace Hessian. 

\item[(b)]  {\bf  Interpolating global rate.} We prove (Section~\ref{sec:global}) that SSCN enjoys a global convergence rate that interpolates between the rate of stochastic CD and the rate of cubic regularized Newton as one varies the expected dimension of the subspace, $\E{\tau(\mS)}$. 

\item[(c)]  {\bf Fast local rate.} Remarkably, we establish a local convergence bound for SSCN  (Section~\ref{sec:local}) that matches the rate of stochastic subspace descent (SSD)~\citep{gower2015randomized} applied to solving the problem \begin{equation} \label{eq:quadratic_optimum}
 \min_{x \in \R^d} \frac12 (x-x^*)^\top \nabla^2f(x^*)(x-x^*),
\end{equation} where $x^*$ is the solution of~\eqref{eq:problem}. Thus, SSCN behaves as {\em if} it had access to a perfect second order model of $f$ at the optimum, and was given the (intuitively much simpler) task of minimizing this model instead. 
Furthermore, note that SSD~\citep{gower2015randomized} applied to minimize a convex quadratic can be interpreted as doing an exact subspace search in each iteration, i.e., it minimizes the objective exactly along the active subspace~\citep{richtarik2017stochastic}. Therefore, the local rate of SSCN matches the rate of the greediest strategy for choosing $h$ in the active subspace, and as such, this rate is the best one can hope for a method that does not incorporate some form of acceleration.

\item[(d)]  {\bf Special cases.} We discuss in Section~\ref{sec:special} how SSCN reduces to several existing stochastic second order methods in special cases, either recovering the best known rates, or improving upon them. This includes  SDSA~\citep{SDA}, CN~\citep{griewank1981modification, nesterov2006cubic} and RBCN~\citet{doikov2018randomized}. However, our method is more general and hence allows for more applications.

\end{itemize}

We discuss more remotely related literature in Section~\ref{sec:related_literature}. We now give a simple example of our setting.

\begin{example}[Coordinate subspace setup]
Let $\mI^d\in \R^{d\times d}$ be the identity and let $S$ be a random subset of $\{1, 2, \dots, d \}$. Given that $\mS = \mI^d_{(:,S)}$ with probability 1, the oracle model reveals $(\nabla f(x))_S$ and $(\nabla^2 f(x))_{(S,S)}$. Therefore, we have access to a random block of partial derivatives of $f$ and a block submatrix of its Hessian, both corresponding to the subset of indices $S$. Furthermore, the rule~\eqref{eq:update_general} updates a subset $S$ of coordinates only. In this setting, our method is a new \textit{second-order coordinate subspace descent} method.
\end{example}

\section{Preliminaries \label{sec:preliminaries}}

Throughout the paper, we assume that  $f$ is convex, twice differentiable, and sufficiently smooth and that $\psi$ is convex, albeit possibly non-differentiable.\footnote{We will also require separability of $\psi$; see Section~\ref{sec:setup}.}

\begin{assumption}\label{as:conv_lipschitz}
Function $f: \R^d \to \R$ is convex and twice differentiable with $M$-Lipschitz continuous Hessian. Function $\psi: \R^d \to \R \cup \{ +\infty \}$ is  proper closed and convex.
\end{assumption}

We always assume that a minimum of $F$  exists and by  $x^*$ denote any of its minimizers. We let $F^{*} \eqdef F(x^{*})$.

Since our method always takes steps along random subspaces spanned by the columns of $\mS \in \R^{d\times \tau(\mS)}$,  it is reasonable to define the Lipschitzness of the Hessian over the range of $\mS$:\footnote{By  $\|x\| \eqdef \la x, x \ra^{1/2}$we denote the standard Euclidean norm.}
\begin{equation}\label{eq:MS_def}
M_{\mS}  \eqdef \max_{x\in \R^d} \max_{\substack{  \;\; h \in \R^{\tau(\mS)}, \\ h \not= 0 } } \frac{ |\nabla^3 f(x)[\mS h]^3| }{\|\mS h\|^3}\, .
\end{equation}

As the next lemma shows,  the maximal value of $M_{\mS}$ for any $\mS$ of width $\tau$ can be up to $(\nicefrac{d}{\tau})^{\frac32}$ times smaller than $M$ and this will lead to  a tighter approximation of the objective. 
\begin{lemma} \label{lem:sharpness}
We have $$
M \geq \max_{\tau(\mS) = \tau }M_{\mS}.
$$
Moreover, there is a problem where $$ \max_{\tau(\mS) = \tau}M_{\mS} =\left(\frac{\tau}{d}\right)^{\frac32} M.$$ Lastly, if $\Range{\mS} = \Range{\mS'}$, then  $M_{\mS} = M_{\mS'}$.
\end{lemma}

The next lemma provides a direct motivation for our algorithm. It gives a global upper bound on the objective over a random subspace, given the first and second-order information at the current point. 

\begin{lemma}\label{lem:ub}
Let $x\in \R^d$, $\mS\in \R^{d\times \tau(\mS)}$, $h\in \R^{\tau(\mS)}$ and $x^+$ be as in \eqref{eq:update_general}. Then
\begin{align}
\left| f(x^+) -  f(x)  - \langle \nabla_{\mS} f(x), h \rangle - \frac12 \la \nabla^2_{\mS}f(x) h, h \ra \right|  \leq \frac{M_{\mS}}{6} \| \mS h\|^3. \label{eq:coordinate_ub}
\end{align}
As a consequence, we have
\begin{eqnarray} 
F(x^+) \leq  f(x) + T_{\mS}(x,h),
\label{eq:coordinate_ub_full}
\end{eqnarray}
where
$$T_{\mS}(x,h) \eqdef \langle \nabla_{\mS} f(x), h \rangle + \frac12 \la  \nabla^2_{\mS}  f(x) h, h \ra
+  \frac{M_{\mS}}{6} \| \mS h\|^3 + \psi(x+\mS h).$$
\end{lemma}

We shall also note that for function $\psi$
we require \textit{separability} with respect to the sampling distribution
(see Definition~\ref{def:separable} and the corresponding
Assumption~\ref{as:separable} in Section~\ref{sec:setup}).

For better orientation throughout the paper, we provide a table of frequently used notation in the Appendix.

\section{Algorithm} \label{sec:Alg}

For a given $\mS$ and current iterate $x^k$, it is a natural idea to choose $h$ as a minimizer of the upper bound~\eqref{eq:coordinate_ub_full} in $h$ for $x=x^k$, and subsequently set $x^{k+1} = x^+$ via~\eqref{eq:update_general}.  Note that we are choosing $\mS$ randomly according to a fixed distribution $\cD$ (with a possibly random number of columns). We have just described SSCN---Stochastic Subspace Cubic Newton---formally stated as  Algorithm~\ref{alg:crcd}.
 
\begin{algorithm}[!h]
\begin{algorithmic}[1]
\STATE \textbf{Initialization:} $x^0$, distribution $\cD$ of random matrices with $d$ rows and full column rank
\FOR{$k =  0, 1, \dots$}
\STATE Sample $\mS$ from distribution $\cD$
\STATE $h^k =  \argmin_{h \in \R^{\tau(\mS)}}  T_{\mS}(x^k,h)$
\STATE Set $x^{k+1}=x^k + \mS h^k$ \label{eq:x_update_CRCD}
\ENDFOR
\end{algorithmic}
\caption{SSCN: Stochastic Subspace Cubic Newton}
\label{alg:crcd}
\end{algorithm}

\begin{remark}\label{rem:monotonic}
Inequality~\eqref{eq:coordinate_ub_full} becomes an equality with $h=0$. As a consequence, we must have $F(x^{k+1}) \leq F(x^k)$, and thus the sequence $\{ F(x^k) \}_{k \geq 0}$ is non-increasing.
\end{remark}

\subsection{Solving the subproblem \label{sec:solving}}
Algorithm~\ref{alg:crcd} requires $T_{\mS}$ to be minimized in $h$ each iteration. As this operation does not have a closed-form solution in general, it requires an optimization subroutine itself of a possibly non-trivial complexity, which we discuss here. 

\paragraph{The subproblem without $\psi$.}
Let us now consider the case when  $\psi(x) \equiv 0$ in which our problem~\eqref{eq:problem} does not contain any nondifferentiable components. Various techniques for minimizing regularized quadratic functions were developed  during the development of
Trust-region methods (see~\citep{conn2000trust}), and applied
to Cubic regularization in~\citep{nesterov2006cubic}. The classical approach consists in performing some diagonalization of the matrix $\nabla^2_{\mS} f(x)$ first, by computing the \textit{eigenvalue} or \textit{tridiagonal} decomposition, which costs $\cO(\tau(\mS)^3)$ arithmetical operations.
Then, to find the minimizer, it merely remains to solve a one-dimensional nonlinear equation (this part can be done by $\tilde{\cO}(1)$ iterations of the one-dimensional Newton method, with a linear cost per step). More details and analysis of this procedure can be found
in~\citep{gould2010solving}.

The next example gives a setting in which  an explicit formula for the minimizer of $T_{\mS}$ can be deduced.
\begin{example}\label{ex:explicit} 
Let $e_i$ be the $i$th unit basis vector in $\R^d$. If $\mS\in \{e_1, \dots, e_d\}$ with probability 1 and $\psi(x) = 0$, the update rule can be written as
$x^{k+1} = x^k - \alpha_i^k e_i,$ with
\[
\alpha_i^k = \frac{2\nabla_i f(x^{k})}{\nabla^2_i f(x^{k}) + \sqrt{\left(\nabla^2_{ii} f(x^{k})  \right)^2 + 2M_{e_i}|\nabla_i f(x^{k}) |} },
\]
thus the cost of solving the subproblem is $\cO(1)$.
\end{example}

\paragraph{Subproblem with simple $\psi$.} In some scenarios, minimization of $T_{\mS}$  can be done using a simple algorithm if $\psi$ is simple enough. We now give an example of this.

\begin{example} 
	If $\mS\in \{e_1, \dots, e_d\}$ with probability 1, the subproblem can be solved using a binary search given that the evaluation of $\psi$ is cheap. In particular, if we can evaluate $\psi(x^k+ \mS h)- \psi(x^k)$ in $\tilde{\cO}(1)$, the cost of solving the subproblem will be $\tilde{\cO}(1)$.
\end{example}

\paragraph{The subproblem with general $\psi$.} In the case of general regularizers, recent line of work by \citet{carmon2019gradient} explores to the use
of \textit{first-order} optimization methods (Gradient Methods) for
computing an approximate minimizer of $T_{\mS}$.
We note that the backbone of such Gradient Methods is an implementation of the following operation 
(for a any given vector $b \in \R^{\tau(\mS)}$, and positive scalars $\alpha, \beta$):
\[
\arg \min \limits_{h \in \R^{\tau(\mS)}} \la b, h \ra + \frac{\alpha}{2}\|\mS h\|^2 + \frac{\beta}{3}\|\mS h\|^3 + \psi(x^k + \mS h).
\]

To the best of our knowledge, the most efficient gradient method is the Fast Gradient Method (FGM) of~\citet{nesterov2019inexact}, achieving an $\cO(1/k^6)$ convergence rate. However, FGM can deal with any $\psi$ as long as the above subproblem is cheap to solve. We shall also note that gradient methods do not require a storage of $\nabla^2_{\mS} f(x)$; but rather iteratively access partial Hessian-vector products $\nabla^2_{\mS} f(x) h$.

\paragraph{Line search.}

Note that in Algorithm~\ref{alg:crcd} we use the Lipschitz constants $M_{\mS}$ 
of the subspace Hessian (see Definition~\eqref{eq:MS_def}) as the regularization parameters. In many application, $M_{\mS}$ can be estimated cheaply (see Section~\ref{sec:applications}). In general, however, $M_\mS$ might be unknown or hard to estimate. 
In such a case, one might use a simple one-dimensional search on each iteration: multiply the estimate of $M_\mS$ by the factor of two until the bound~\eqref{eq:coordinate_ub_full} is satisfied, and divide it by two at the start of each iteration. Note that the average number of such line search steps per iteration can be bounded by two (see~\citep{grapiglia2017regularized} for the details).

\subsection{Special cases} \label{sec:special}
There are several scenarios where SSCN becomes an already known algorithm.  We list them below:

\begin{itemize}
\item 
{\bf Quadratic minimization.}  If $M=0$ and $\psi = 0$, SSCN reduces to the stochastic dual subspace ascent (SDSA) method~\citep{SDA}, first analyzed in an equivalent primal form as a \emph{sketch-and-project} method in~\citep{gower2015randomized}. In such a case, SSCN performs both first-order, second-order updates, and exact minimization over a subspace at the same time due to the quadratic structure of the objective~\citep{richtarik2017stochastic}. The convergence rate we provide in Section~\ref{sec:local} exactly matches the rate of sketch-and-project as well. As a consequence, we recover a subclass of matrix inversion algorithms~\citep{gower2017randomized} together with stochastic spectral (coordinate) descent~\citep{kovalev2018stochastic} along with their convergence theory. 

\item 
{\bf Full-space method.} If $\mS=\mI^d$ with probability 1, SSCN reduces to cubically regularized Newton (CN)~\citep{griewank1981modification, nesterov2006cubic}. In this case, we recover both existing global convergence rates and superlinear local convergence rates.  

\item 
{\bf Separable non-quadratic part of $f$.} The RBCN method of~\citet{doikov2018randomized} aims to minimize~\eqref{eq:problem} with $$f(x) = g(x) + \phi(x),$$ where $g,\phi$ are both convex, and $\phi$ is separable.\footnote{Separability is defined in Section~\ref{sec:setup}.}
They assume that $$\nabla^2 g(x) \preceq \mA \in \R^{d\times d}, \qquad \forall x \in \R^d,$$ while $\phi$ has Lipschitz continuous Hessian. In each iteration, RBCN constructs an upper bound on the objective using first order information from $g$ only. This is unlike SSCN, which uses second order information from $g$. In a special case when $\nabla^2 g(x) = \mA$ for all $x$, SSCN and RBCN are identical algorithms. However, RBCN is less general: it requires separable $\phi$, and thus does not cover some of our applications, and takes directions along coordinates only. Further, the rates we provide are better even in the setting where the two methods coincide ($\nabla^2 g(x)=\mA$). The simplest way to see that is by looking at local convergence -- RBCN does not achieve the local convergence rate of block CD to minimize~\eqref{eq:quadratic_optimum}, which is the best one might hope for. 
\end{itemize}

Besides these particular cases, for a general twice-differentiable $f$, SSCN is a new second-order method.
 
\section{Related Literature} \label{sec:related_literature}

Several methods in the literature are related to SSCN. We briefly review them below.

\begin{itemize}
\item \emph{Cubic regularization of Newton method} was proposed first by~\citet{griewank1981modification},
and received substantial attention after the work of~\citet{nesterov2006cubic}, where its global complexity guarantees were established.
During the last decade, there was a steady increase of research in second-order methods, 
discovering Accelerated~\citep{nesterov2008accelerating,monteiro2013accelerated},
Adaptive~\citep{cartis2011adaptive1,cartis2011adaptive2}, and
Universal~\citep{grapiglia2017regularized,grapiglia2019accelerated,doikov2019minimizing} 
schemes (the latter ones are adjusting automatically 
to the smoothness properties of the objective).

\item There is a vast literature on \emph{first-order coordinate descent (CD)} methods. While CD with $\tau=1$ is consistently the same method within the literature~\citep{nesterov2012efficiency, richtarik2014iteration, wright2015coordinate}, there are several ways to deal with $\tau>1$. The first approach constructs a separable upper bound on the objective (in expectation) in the direction of a random subset of coordinates~\citep{qu2016coordinate1, qu2016coordinate2}, which is minimized to obtain the next iterate. The second approach---SDNA~\citep{qu2016sdna}---works with a tighter non-separable upper bound. SDNA is, therefore, more costly to implement but requires a smaller number of iterations to converge. The literature on first-order subspace descent algorithms is slightly less rich, the notable examples are random pursuit~\citep{stich2013optimization} or stochastic subspace descent~\citep{SSD:Kozak2019}.

\item \emph{Randomized subspace Newton} (RSN)~\citep{gower2019rsn} is a method of the form $$x^{k+1} = x^k - \frac{1}{\hat{L}} \mS \left(\nabla^2_\mS f(x^k)\right)^{-1} \nabla_{\mS} f(x^k)$$ for some specific fixed $\hat{L}$. In particular, it can be seen as a method minimizing the following upper bound on the function, which follows from their assumption:
\begin{align*}
 & h^k = \arg \min_h\, \langle \nabla_{\mS} f(x^k), h  \rangle + \frac{\hat{L}}{2}  \la \nabla^2_{\mS}  f(x^k) h, h \ra. 
\end{align*}
This is followed by an update over the subspace: $ x^{k+1} = x^k + \mS h^k$. 
Since both RSN and SSCN are analyzed under different assumptions, the global linear rates are not directly comparable. However, the local rate of SSCN is superior to RSN. We shall also note that RSN is a stochastic subspace version of a method from~\citep{karimireddy2018global}.

\item \emph{Subsampled Newton} (SN) methods~\citep{byrd2011use, erdogdu2015convergence, xu2016sub, roosta2019sub} and \emph{subsampled cubic regularized Newton methods}~\citep{kohler2017sub, xu2017newton, wang2018stochastic} and \emph{stochastic (cubic regularized) Newton methods}~\citep{tripuraneni2018stochastic,cartis2018global, kovalev2019stochastic} are  stochastic second-order algorithms to tackle finite sum minimization. Their major disadvantage is a requirement of an immense sample size, which makes them often impractical if used as theory prescribes. A notable exception that does not require a large sample size was recently proposed by~\citet{kovalev2019stochastic}. However, none of these methods are directly comparable to SSCN as they are not subspace descent methods, but rather randomize over data points (or sketch the Hessian from ``inside''~\citep{pilanci2017newton}). 

\end{itemize}

\section{Global Complexity Bounds}
\label{sec:global}

We first start presenting the global complexity results of SSCN.

\subsection{Setup~\label{sec:setup}}

Throughout this section, we require some kind of uniformity of the distribution $\cD$ over subspaces given by $\mS$. In particular, we require $$\mP^\mS\eqdef \mS \left(\mS^\top \mS \right)^{-1} \mS^\top,$$ the projection matrix onto the range of $\mS$, to be a scalar multiple of identity matrix in expectation. 
\begin{assumption}\label{as:uniform}
$\exists \tau>0$ such that distribution  $\cD$ satisfies
\begin{equation} \label{eq:uniform_sampling}
\E{\mP^{\mS}} = \frac{\tau}{d} \mI^d.
\end{equation}
\end{assumption}
A direct consequence of Assumption~\ref{as:uniform} is that $\tau$ is an expected width of $\mS$, as the next lemma states.
\begin{lemma}\label{lem:exp_size}
If Assumption~\ref{as:uniform} holds, then $\E{\tau(\mS)} = \tau$.
\end{lemma}

As mentioned before, the global complexity results are interpolating between convergence rate of (first-order) CD and (global) convergence rate of Cubic Newton. However, first-order CD requires Lipschitzness of gradients, and thus we will require it as well.
\begin{assumption}\label{as:smooth}
Function $f$ has $L$-Lipschitz continuous gradients, i.e., $\nabla^2 f(x) \preceq L \mI^d$ for all $x\in \R^d$.
\end{assumption}

We will also need an extra assumption on $\psi$. It is well known that proximal (first-order) CD with fixed step size does not converge if $\psi$ is not separable -- in such case, even if $f(x^k) = f(x^*)$ we might have $f(x^{k+1})> f(x^*)$. Therefore, we might not hope that SSCN will converge without additional assumptions on $\psi$. Informally speaking, separability of $\psi$ with respect to directions given by columns of $\mS$ is required. 
To define it formally, let us introduce first the notion of a separable set.
\begin{definition}\label{def:separable_set}
Set $Q \subseteq \R^d$ is called $D$-separable, if $\forall x, y \in Q, \mS \in D$:
$
\ba{rcl}
\mP^\mS x + (\mI^d- \mP^\mS)y  \in  Q.
\ea
$
\end{definition}
 Using the set separability, we next define a separability of a function. 
\begin{definition}\label{def:separable}
Function $\phi: \R^d \rightarrow \R \cup \{ +\infty \}$ is $D$-separable if
$\dom \phi$ is $D$-separable, and there is map $\phi': \dom \phi \rightarrow \R^d$ such that
\begin{enumerate}
\item  $\forall x \in \dom\phi:\; \phi(x) = \langle \phi'(x), e \rangle$,\footnote{By $e\in \R^d$ we mean the vector of all ones.}
\item $\forall x,y \in \dom\phi, \mS\in D: \; \phi'(\mP^\mS x + (\mI^d- \mP^\mS)y) = \mP^\mS\phi'(x) + (\mI^d-\mP^\mS) \phi'(y)$.		
\end{enumerate}
\end{definition}

\begin{example}
If $D$ is a set of matrices whose columns are standard basis vectors, $D$-separability reduces to classical (coordinate-wise) separability.
\end{example}

\begin{example}
If $D$ is set of matrices which are column-wise submatrices of orthogonal $\mU$, $D$-separability of $\phi$ reduces to classical coordinate-wise separability of $\phi(\mU^\top x)$.
\end{example}

\begin{example}
$\phi(x) = \frac12\| x\|^2$ is $D$-separable for any~$D$.
\end{example}

\begin{assumption}\label{as:separable}
Function $\psi$ is $\Range{\cD}$-separable.
\end{assumption}

We are now ready to present the convergence rate of SSCN. 

\subsection{Theory}

First, let us introduce the critical lemma from which the main global complexity results are derived.  The next lemma states, what is the expected progress we have for one step of SSCN.

\begin{lemma}\label{lem:keylemma}
Let Assumptions~\ref{as:conv_lipschitz},~\ref{as:uniform},~\ref{as:smooth} and~\ref{as:separable} hold.
Then, for every $k \geq 0$ and $y \in \R^d$ we have
\begin{equation} \label{GlobalUpper}
  \E{F(x^{k + 1}) \, | \, x^k } 
 \leq 
\left(1 - \frac{\tau}{d} \right)F(x^k) + \frac{\tau}{d} F(y) 
   +  \frac{\tau}{d} \left( 
\frac{d - \tau}{d } \frac{L}{2}\|y - x^k\|^2 + \frac{M}{3}\|y - x^k\|^3
\right).
\end{equation}
\end{lemma}

Now we are ready to present global complexity results for the general class of convex functions. The convergence rate is obtained by summing ~\eqref{GlobalUpper} over the different iterations $k$, and with a specific choice of $y$. 

\begin{theorem}\label{thm:global_weakly}
	Let Assumptions~\ref{as:conv_lipschitz},~\ref{as:uniform},~\ref{as:smooth} and~\ref{as:separable} hold. Denote
\begin{equation} \label{Rdef}
	R  \Def  \sup\limits_{x \in \R^d} \left \{  \|x - x^{*} \| \; : \; F(x) \leq F(x^{0})  \right\},
\end{equation}
	and suppose that $R < +\infty$.
	Then, for every $k \geq 1$ we have
\begin{equation}\label{GlobalConv}
  \E{ F(x^k) } - F^{*} \leq 
	\frac{d - \tau}{\tau} \cdot \frac{4.5 L R^2}{ k } + \left(\frac{d}{\tau}\right)^2 \cdot \frac{9 M R^3}{k^2}
	+ \frac{F(x^0) - F^{*}}{1 + \frac{1}{4}\left( \frac{\tau}{d} k \right)^{3}}.
\end{equation}
\end{theorem}

Note that convergence rate of the minibatch version\footnote{Sampling $\tau$ coordinates at a time for objectives with $L$-Lipschitz gradients.} of first-order CD is $\cO\bigl( \frac{d}{\tau} \frac{LR^2}{k}\bigr)$. At the same time, (global) convergence rate of cubically regularized Newton method is $\cO\bigl( \frac{MR^3}{k^2}\bigr)$. Therefore, Theorem~\ref{thm:global_weakly} shows that the global rate of SSCN well interpolates between the two extremes,
depending on the sample size $\tau$ we choose.

\begin{remark}
According to estimate~\eqref{GlobalConv}, in order to have $\E{F(x^k)} - F^{*} \leq \varepsilon$,
it is enough to perform
$$
k = \cO\left(  \frac{d - \tau}{\tau} \frac{LR^2}{\varepsilon} + \frac{d}{\tau} \sqrt{\frac{MR^3}{\varepsilon}}
+ \frac{d}{\tau}\left( \frac{F(x^0) - F^{*}}{\varepsilon} \right)^{1/3}  \right)
$$
iterations of SSCN.
\end{remark}

Next, we move to the strongly convex case.  

\begin{assumption}\label{as:sc}
Function $f$ is $\mu$-strongly convex, i.e., $\nabla^2 f(x) \succeq \mu \mI^d$ for all $x\in \R^d$.
\end{assumption}

\begin{remark}
Strong convexity of the objective (assumed for Theorem~\ref{thm:global_strongly} later) implies:
$ R < +\infty$. Furthermore, due to monotonicity of the sequence $\{ F(x_k) \}_{k \geq 0}$ (see Remark~\ref{rem:monotonic}), we have
$\|x^k - x^{*}\| \leq R$ for all $k$. Therefore, it is sufficient to require Lipschitzness of gradients over the sublevel set, which holds with $L = \lambda_{\max}(\nabla^2 f(x^{*})) + MR$.
\end{remark}

As both extremes cubic regularized Newton (where $\mS = \mI^d$ always) and (first-order) CD ($\mS = e_i$ for randomly chosen $i$) enjoy (global) linear rate under strong convexity, linear convergence of SSCN is expected as well. At the same time, the leading complexity term should be in between the two extremes. Such a result is established as Theorem~\ref{thm:global_strongly}.

\begin{theorem}\label{thm:global_strongly}
Let Assumptions~\ref{as:conv_lipschitz},~\ref{as:uniform},~\ref{as:separable}
and~\ref{as:sc} hold. 
Then, 
$ \E{F(x^k)} - F^{*} \leq  \varepsilon$,
as long as the number of iterations of SSCN is
$$
 k =   \cO\left( \left(  \frac{d - \tau}{\tau} \frac{L}{\mu} + \frac{d}{\tau}\sqrt{\frac{M R}{\mu}} + \frac{d}{\tau}     
\right) 
 \log \left(\frac{F(x^0) - F^{*}}{\varepsilon}\right)  \right)
\,.
$$
\end{theorem}

Indeed, if $\mS =\mI^d$ with probability 1 and $MR\geq \mu$, the leading complexity term becomes $\sqrt{\frac{MR}{\mu}} \log\frac{1}{\varepsilon}$ which corresponds to the global complexity of cubically regularized Newton for
minimizing strongly convex functions~\citep{nesterov2006cubic}.
On the other side of the spectrum if $\mS = e_i$ with probability $\frac1d$, the leading complexity term becomes $\frac{dL}{\mu} \log\frac{1}{\varepsilon}$, which again corresponds to convergence rate of CD~\citep{nesterov2012efficiency}. 
Lastly, if $1<\tau<d$, the global linear rate interpolates the rates mentioned above.

\begin{remark}
Proof of Theorem~\ref{thm:global_strongly} only uses the following consequence of strong convexity:
\beq \label{eq:SConvex}
\frac{\mu}{2}\|x - x^{*}\|^2  \leq F(x) - F^{*}, \qquad x \in \R^d
\eeq
and thus the conditions of Theorem~\ref{thm:global_strongly} might be slightly relaxed.\footnote{However, this relaxation is not sufficient to obtain the local convergence results.} For detailed comparison of various relaxations of strong convexity, see~\citep{karimi2016linear}.
\end{remark}

\section{Local Convergence \label{sec:local}}

Throughout this section, assume that $\psi = 0$. We first present the key descent lemma, which will be used to obtain local rates. Let $$\mH_{\mS}(x) \eqdef  \nabla^2_{\mS} f(x) +\sqrt{ \frac{M_{\mS}}{2}} \| \nabla_{\mS} f(x)\|^{\frac12} \mI^{\tau(\mS)} .$$

\begin{lemma} \label{lem:decrease} We have 
\begin{equation}\label{eq:decrease}
f(x^k)-f(x^{k+1})  \geq   \frac12 \| \nabla_{\mS} f(x^k)\|^2_{  \mH^{-1}(x^k)  }.
\end{equation}
\end{lemma}

Before stating the convergence theorem, it will be suitable to define the stochastic condition number of  $\mH_*\eqdef \nabla^2 f(x^*)$:
\begin{equation}\label{eq:sc_generalized}
\zeta  \eqdef \lambda_{\min} \left( \mH_*^{\frac12}  \E{{\mS} \left(\mS^\top \mH_* \mS \right)^{-1}{\mS}^\top}  \mH_*^{\frac12} \right),
\end{equation}
as it will drive the local convergence rate of SSCN.

\begin{theorem}[Local Convergence]\label{thm:local}
Let Assumptions~\ref{as:conv_lipschitz},~\ref{as:sc} hold, and suppose that $\psi = 0$. 
 For any $\varepsilon>0$ there exists $\delta>0$ such that if $F(x^0) - F^* \leq \delta$, we have
\begin{equation}\label{eq:local_rate}
\E{F(x^k) - F^*} \leq \left( 1- \left(1-\varepsilon \right)\zeta  \right)^k\left( F(x^0) - F^*\right)
\end{equation}
and therefore the local complexity of SSCN is $$\cO\left( \zeta^{-1} \log\frac1\varepsilon\right).$$ If further $M=0$ (i.e., $f$ is quadratic), then $\varepsilon =0$ and $\delta = \infty $, and thus the rate is global.

\end{theorem}

The proof of Theorem~\ref{thm:local} along with the exact formulas for $\varepsilon,\delta$ can be found in Section~\ref{sec:local_proofs} of the Appendix. Theorem~\ref{thm:local} provides a local linear convergence rate of SSCN. While one might expect a superlinear rate to be achievable, this is not the case, and we argue that the rate from Theorem~\ref{thm:local} is the best one can hope for. 

In particular, if $M = 0$, Algorithm~\ref{alg:crcd} becomes subspace descent for minimizing positive definite quadratic which is a specific instance of sketch-and-project~\citep{gower2015randomized}. However, sketch-and-project only converges linearly -- the iteration complexity of sketch-and-project to minimize $(x-x^*)^\top \mA (x-x^*)$ with $\mA\succ 0$ is 
\[ \cO\left(  \left[\lambda_{\min} \left( \mA^{\frac12}  \E{{\mS}\left( \mS^\top \mA \mS \right)^{-1}{\mS}^\top}  \mA^{\frac12} \right) \right]^{-1} \log\frac1\varepsilon\right). \]
 Notice that this rate is matched by Theorem~\ref{thm:local} in this case.

Next, we compare the local rate of SSCN to the rate of SDNA~\citep{qu2016sdna}. To best of our knowledge, SDNA requires the least oracle calls to minimize $f$ among all first-order non-accelerated methods.

\begin{remark}
SDNA is a first-order analogue to Algorithm~\ref{alg:crcd} with $\mS=\mI^d_{(:,S)}$. In particular, given matrix $\mL$ such that $\mL\succeq \nabla^2 f(x) \succ 0 $ for all $x$, the update rule of SDNA is $$x^+ = x - \mS \left( \mS^\top \mL \mS \right)^{-1} \nabla_{\mS} f(x),$$ where $\mS = \mI^d_{(:,S)}$ for a random subset of columns $S$.  SDNA enjoys linear convergence rate with leading complexity term $\left(\mu \lambda_{\min} \left(\E{\mS(\mS^\top \mL \mS)^{-1}\mS^\top} \right)\right)^{-1}$. The leading complexity term of SSCN is $\zeta^{-1}$, and we can bound
\begin{eqnarray*}
\zeta &\geq & \lambda_{\min}\left( \mH_* \right) \lambda_{\min} \left(\E{{\mS} \left(\mS^\top \mH_* \mS\right)^{-1}{\mS}^\top}   \right)
\geq  \mu \lambda_{\min} \left(\E{\mS \left( \mS^\top \mL \mS \right)^{-1}\mS^\top}\right).
\end{eqnarray*}
Hence, the local rate of SSCN is no worse than the  rate of SDNA. Furthermore, both of the above inequalities might be very loose in some cases (i.e., there are examples where $\frac{\zeta}{ \mu \lambda_{\min}\E{\mS(\mL_\mS)^{-1}\mS^\top} }$ can be arbitrarily high). Therefore, local convergence rate of SSCN might be arbitrarily better than the convergence rate of SDNA. As a consequence, the local convergence of SSCN is better than convergence rate of any non-accelerated first order method.\footnote{The rate of SSCN and rate of accelerated subspace descent methods are not directly comparable -- while the (local) rate of SSCN might be better than rate of ACD, the reverse might happen as well. However, both ACD and SSCN are faster than non-accelerated subspace descent.}.
\end{remark}

Lastly, the local convergence rate provided by Theorem~\ref{thm:local} recovers the superlinear rate of cubic regularized Newton's method, as the next remark states.

 \begin{remark}
 If $\mS = \mI^d$ with probability 1, Algorithm~\ref{alg:crcd} becomes cubic regularized Newton method~\citep{griewank1981modification, nesterov2006cubic}. For $\mH_*\eqdef \nabla^2 f(x^*)$ we have 
 \begin{align*}
 \zeta = \lambda_{\min} \left( \mH_*^{\frac12} \mH_*^{-1} \mH_*^{\frac12} \right) =  \lambda_{\min} (\mI^d) = 1.
 \end{align*}
As a consequence of Theorem~\ref{thm:local}, for any $\varepsilon>0$ there exists $\delta>0$ such that if $F(x)-F(x^*)\leq \delta$, we have
 \[
 F(x^+)- F(x^*)\leq \varepsilon(F(x)-F(x^*)).
 \]
 Therefore, we obtain a superlinear convergence rate. 
 \end{remark}

\section{Applications}
\label{sec:applications}

\subsection{Linear Models\label{sec:linear}}
Consider only $\mS = \mI^d_{(:,S)}$ for simplicity.
 Let 
 \begin{equation} \label{eq:linear_model}
 F(x) \eqdef \frac1n \sum \limits_{i=1}^n \phi_i( \la a_i, x \ra )+ \psi(x),
 \end{equation}
 and $f(x) \eqdef \frac1n \sum_{i=1}^n \phi_i( \la a_i, x \ra )$ 
  and suppose that $|\nabla^3 \phi_i (y)| \leq c$. Then clearly,  $$\nabla^3 f(x)[h]^3 =\frac1n \sum_{i=1}^n \nabla^3\phi_i( \la a_i, x\ra ) \la a_i, h \ra^3 $$ for any $h \in \R^d$. While evaluating $E\eqdef \max_{\|h\|=1, x} \nabla^3 f(x)[h]^3 $ is infeasible, we might bound it instead via
\begin{align}
 E &\leq  \max \limits_{\|h\|=1}  \frac{c}{n}\sum \limits_{i=1}^n| \la a_i, h \ra |^3  \leq    \frac{c}{n}\sum \limits_{i=1}^n\|a_i\|^3, \label{eq:ndanjdajn}
\end{align}
which means that $M=\frac{c}{n} \sum_{i=1}^n \| a_i\|^3$ is a feasible choice. On the other hand, for $S = \{ j \}$ we have
 \[
 \max \limits_{\|h_j\|=1, x} \nabla^3 f(x)[h_j]^3 = \max \limits_{x} \nabla^3 f(x)[e_j]^3 \leq \frac{c}{n} \sum \limits_{i=1}^n|a_{ij}|^3 
 \]
 and thus we might set $M_j =\frac{c}{n} \sum_{i=1}^n | a_{ij}|^3$. The next lemma compares the above choices of $M$ and $M_j$.

\begin{lemma}
We have $M\geq \max_j M_j$. At the same time, there exist vectors $\{a_i\}$ that $$\max_j M_j  = \frac{M}{d^{\frac32}}.$$
\end{lemma}
\begin{proof}
The first part is trivial. For the second part, consider $a_{i,j}\in \{-1,1\}$. 
\end{proof}

\begin{remark}
One might avoid the last inequality from~\eqref{eq:ndanjdajn} using polynomial optimization; however, this might be more expensive than solving the original optimization problem and thus is not preferable. Another strategy would be to use a line search, see Section~\ref{sec:solving}.
\end{remark}

Both the formula for $M$ and the formula for $M_j$ require the prior knowledge of $c\geq 0$ such that $|\nabla^3 \phi_i (y)| \leq c$ for all $i$. The next lemma shows how to compute such $c$ for the logistic regression (binary classification model).

\begin{lemma}
Let $\phi_i(y) = \log(1+e^{-b_iy})$, where $ b_i \in \{-1,1\}$. Then $c = \frac{1}{6\sqrt{3}}$.
\end{lemma}
\begin{proof}
$\nabla^3 \phi_i(y) = -\frac{e^x(e^x-1)}{(1+e^x)^3}$ $\Rightarrow$ $\left| \nabla^3 \phi_i(y) \right|\leq \frac{1}{6\sqrt{3}}$.
\end{proof}

\paragraph{Cost of performing a single iteration}
For the sake of simplicity, let $\tau(\mS)=1$, $\psi=0$. Any CD method (i.e,. method with update rule~\eqref{eq:update_general} with $\mS\in \{e_1, \dots, e_d\}$) can be efficiently implemented by memorizing the residuals $\la a_i, x^k \ra$, which is cheap to track since $x^{k+1} - x^k$ is a sparse vector. The overall cost of updating the residuals is $\cO(n)$ while the cost of computing $\nabla_i f(x)$ and $\nabla^2_{i,i} f(x)$ (given the residuals are stored) is $\cO(n)$. Therefore the overall cost of performing a single iteration is $\cO(n)$. Generalizing to $\tau(\mS)=\tau\geq 1$, the overall cost of single iteration of SSCN can be estimated as $\cO(n\tau^2 + \tau^3)$,
where $\cO(n\tau^2)$ comes from evaluating subspace gradient and Hessian, while $\cO(\tau^3)$ comes from solving the cubic subproblem.

\subsection{Dual of linear models \label{sec:dual}}
So far, all results and applications for CRDS we mentioned were problems with large model size $d$. In this section we describe how SSCN can be efficient to tackle big data problems in some settings. Let $\mA\in \R^{n\times d}$ is data matrix and consider a specific instance of~\eqref{eq:linear_model} where
\begin{equation}\label{eq:to_be_dualized}
 \min \limits_{x\in \R^d} F_P(x) \eqdef  \frac1d \sum \limits_{i=1}^n \rho_i(\mA_{(:,i)}x) + \frac{\lambda}{2} \|x \|^2.
\end{equation}
where $\rho_i$ is convex for all $i$. One can now formulate a dual problem of~\eqref{eq:to_be_dualized} as follows:
\begin{equation}\label{eq:dual_linear}
 \max \limits_{y\in \R^n} F_D(y) \eqdef - \frac{1}{2\lambda n^2} \left\| \mA^\top y \right\|^2 -\frac{1}{n} \sum \limits_{i=1}^n\rho_i^*(e_i^\top x).
\end{equation}
Note that~\eqref{eq:dual_linear} is of form~\eqref{eq:linear_model}, and therefore if $\rho^*_i$ has Lipschitz Hessian, we can apply SSCN to efficiently solve it (same as Section~\ref{sec:linear}). Given the solution of~\eqref{eq:dual_linear}, we can recover the solution of~\eqref{eq:to_be_dualized} (duality theory). Thus, SSCN can be used as a data-stochastic method to solve finite-sum optimization problems.

The trick described in this section is rather well known. It was  first used in~\citep{shalev2013stochastic}, where  CD applied to the problem~\eqref{eq:dual_linear} (SDCA) was shown to be competitive with the variance reduced methods like SAG~\citep{roux2012stochastic}, SVRG~\citep{johnson2013accelerating} or SAGA~\citep{defazio2014saga}.


\section{Experiments}
We now numerically verify our theoretical claims.

\subsection{Logistic Regression}

In this section, we consider binary classification with  LIBSVM~\citep{chang2011libsvm} data modelled by regularized logistic regression. Regularized logistic regression is a machine learning model for binary classification. Given data matrix $\mA\in \R^{n\times d}$, labels $b\in \{-1, 1\}^{n}$ and regularization parameter $\lambda >0$, the training corresponds to solving the following optimization problem
\[
f(x)= \frac1n \sum_{i=1}^n \log \left(1+\exp\left(\mA_{i,:}x\cdot  b\right) \right)+\frac{\lambda}{2} \| x\|^2, \qquad x\in \R^d.
\]

We compare SSCN against three different instances of (first-order) randomized coordinate descent: CD with uniform sampling, CD with importance sampling~\citep{nesterov2012efficiency}, and accelerated CD with importance sampling~\citep{allen2016even, nesterov2017efficiency}. 

\subsubsection{Coordinate sketching setup}

In the first experiment, we compare SSCN to first-order coordinate descent (CD) on LIBSVM~\citep{chang2011libsvm}. We consider three different instances of CD: CD with uniform sampling, CD with importance sampling~\citep{nesterov2012efficiency}, and accelerated CD with importance sampling~\citep{allen2016even, nesterov2017efficiency}.

In order to be comparable with the mentioned first-order methods, we consider $\mS\in \{e_1, \dots, e_d\}$ with probability 1  -- the complexity of performing each iteration is about the same for each algorithm now. At the same time, computing $M_{e_i}$ for all $1\leq i \leq d$ is of cost $\cO(nd)$ -- the same cost as computing coordinate-wise smoothness constants for (accelerated) coordinate descent (see Section~\ref{sec:linear} for the details). Figure~\ref{fig:libsvm} shows the result for non-normalized data, while Figure~\ref{fig:libsvm_normalzied} shows the results for normalized data (thus importance sampling is identical to uniform). 

In all examples, SSCN outperformed CD with uniform sampling. Moreover, the performance of SSCN was always either about the same or significantly better to CD with importance sampling. Furthermore, SSCN was also competitive to accelerated CD with importance sampling (in about half of the cases, SSCN was faster, while in the other half, accelerated CD was faster).

\begin{figure}[!h]
\centering
\begin{minipage}{0.3\textwidth}
  \centering
\includegraphics[width =  \textwidth ]{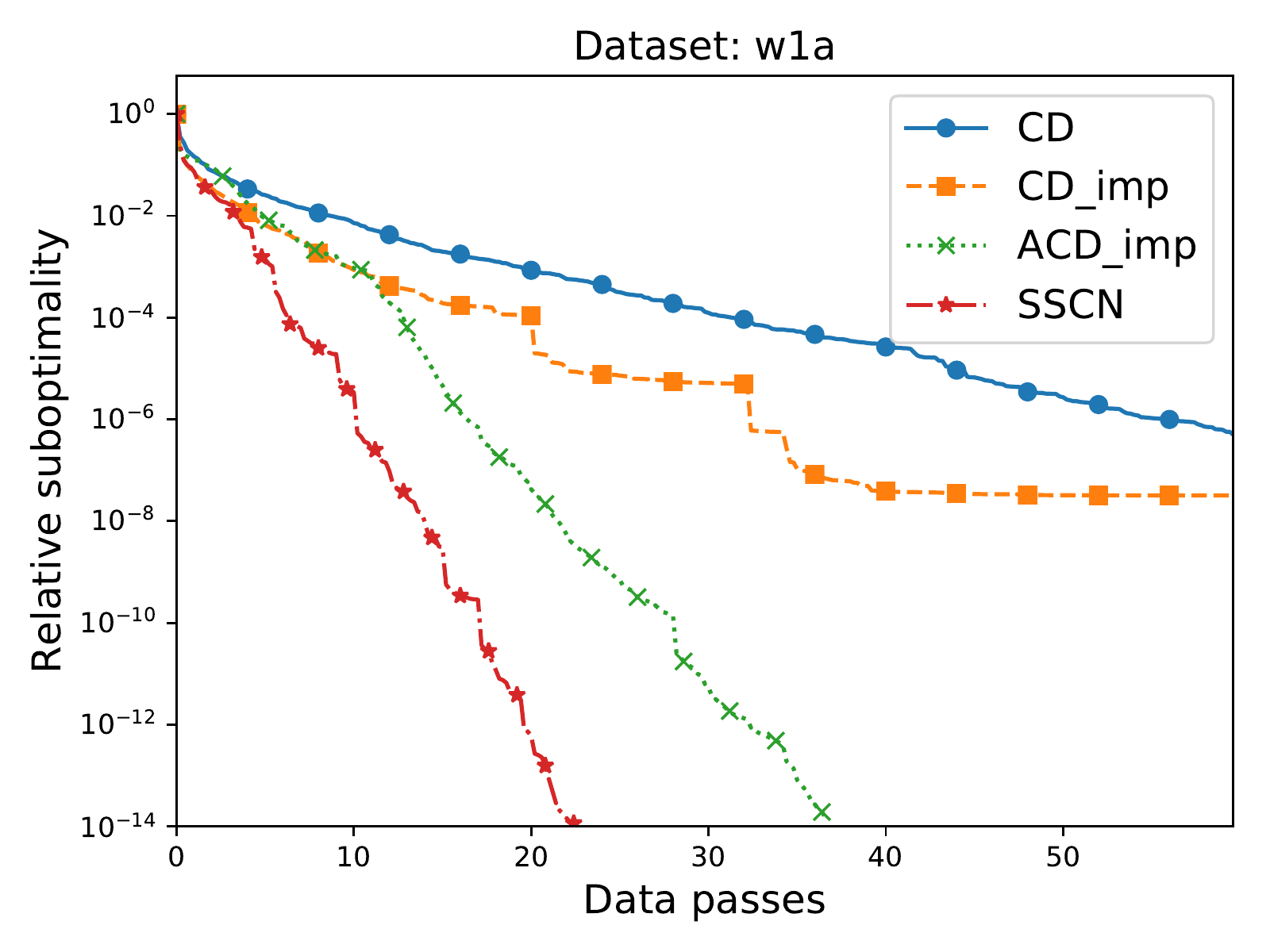}
\end{minipage}%
\begin{minipage}{0.3\textwidth}
  \centering
\includegraphics[width =  \textwidth ]{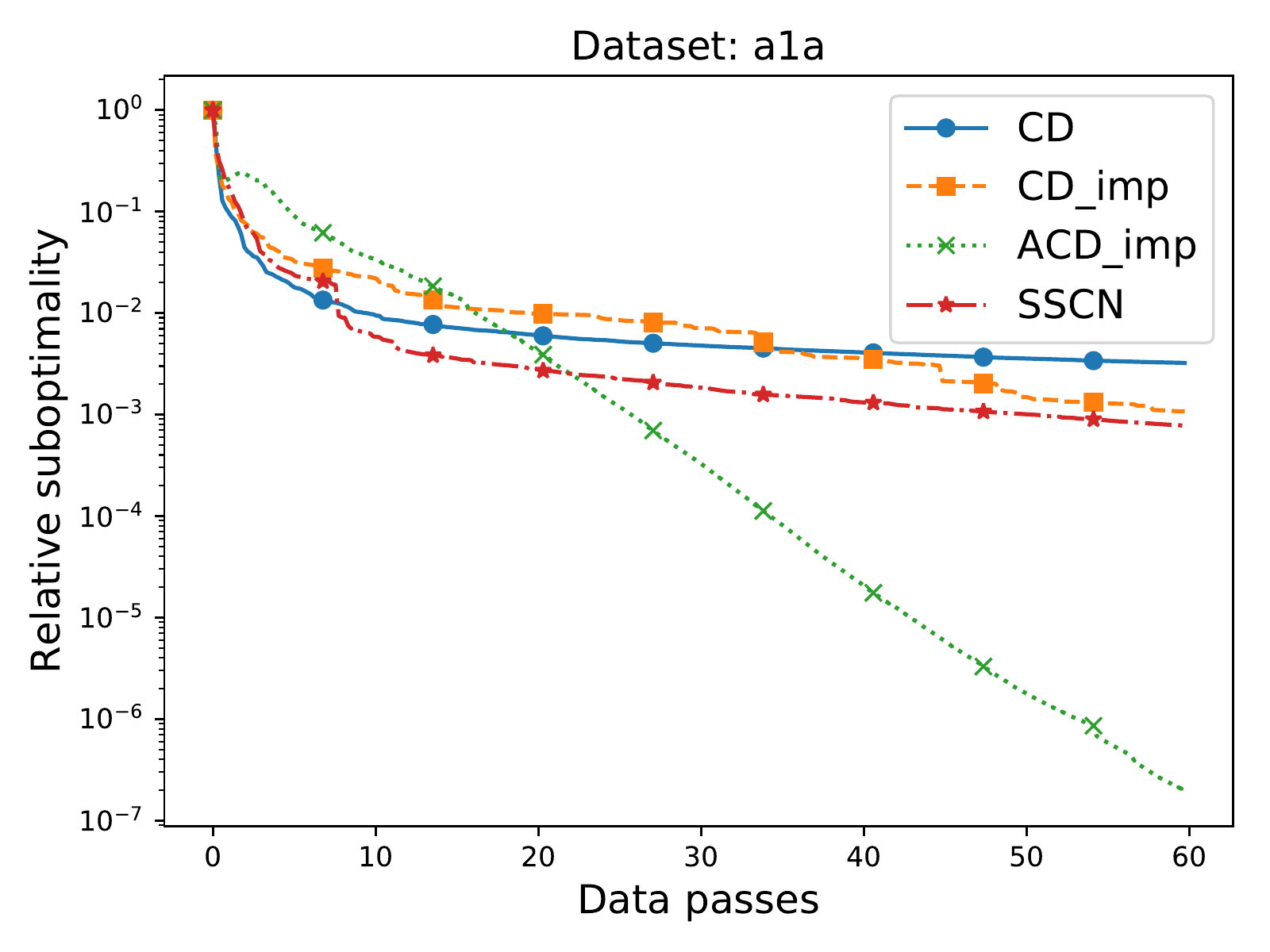}
\end{minipage}%
\begin{minipage}{0.3\textwidth}
  \centering
\includegraphics[width =  \textwidth ]{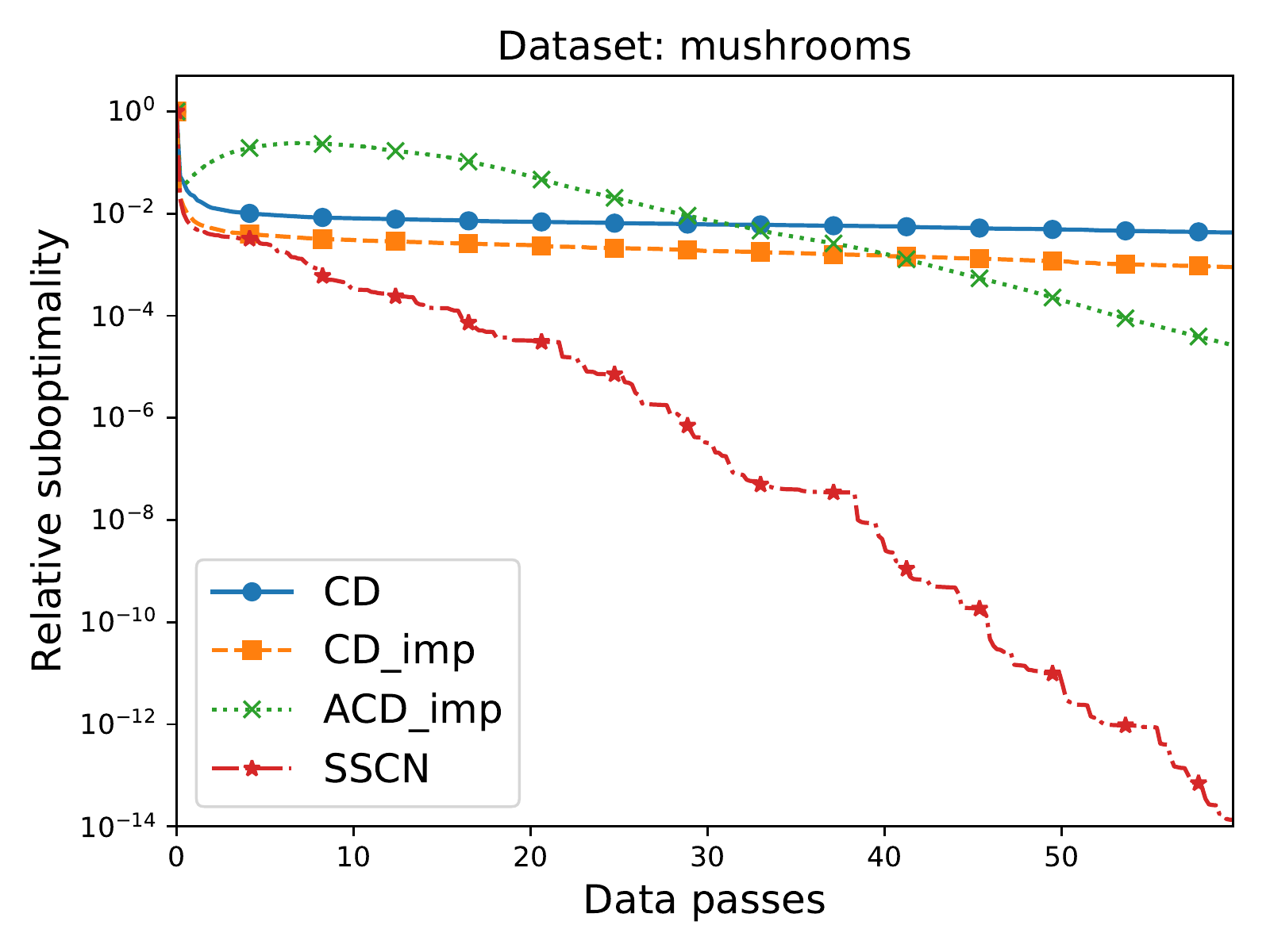}
\end{minipage}%
\\
\begin{minipage}{0.3\textwidth}
  \centering
\includegraphics[width =  \textwidth ]{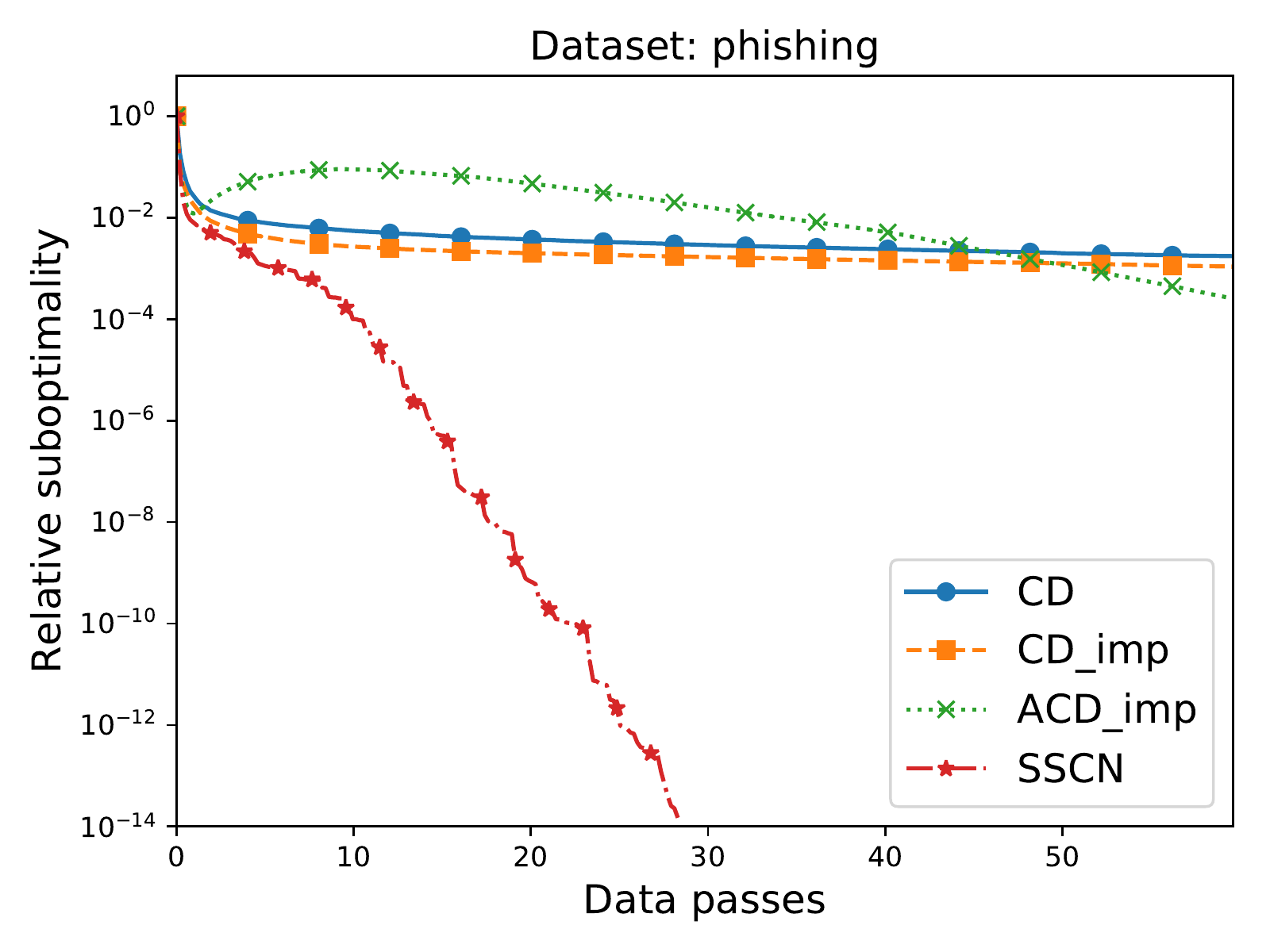}
\end{minipage}%
\begin{minipage}{0.3\textwidth}
  \centering
\includegraphics[width =  \textwidth ]{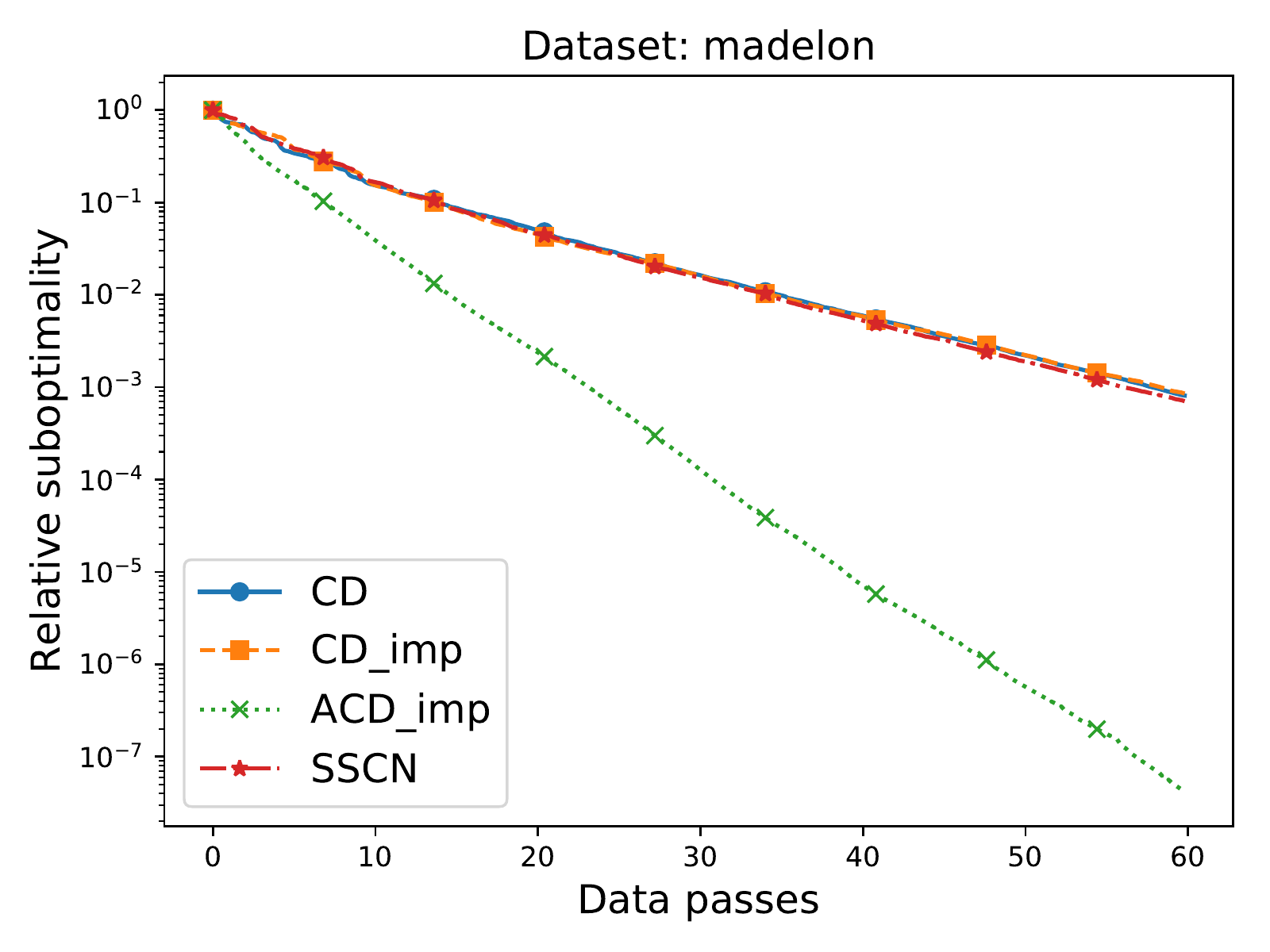}
\end{minipage}%
\begin{minipage}{0.3\textwidth}
  \centering
\includegraphics[width =  \textwidth ]{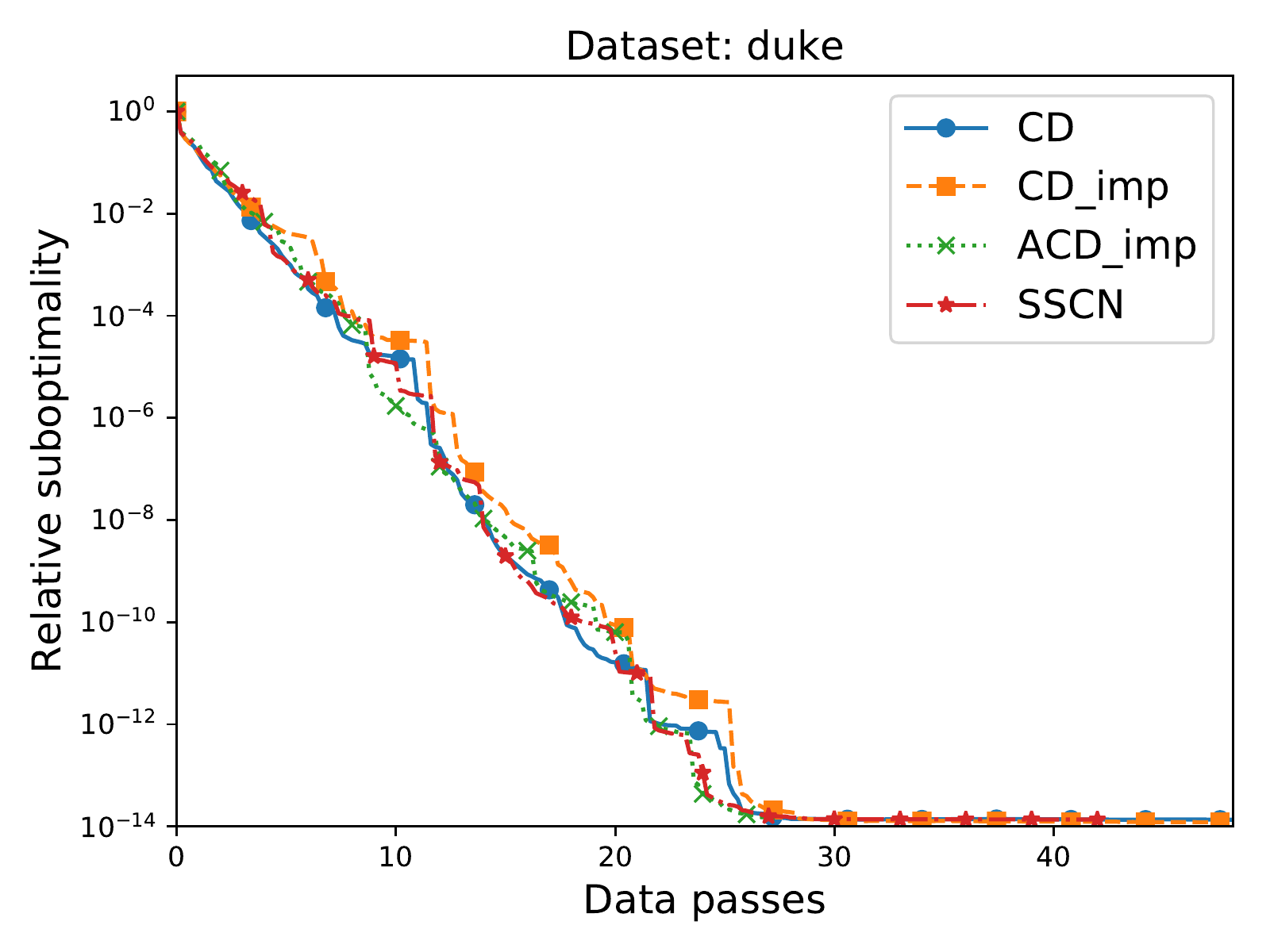}
\end{minipage}%
\\
\begin{minipage}{0.3\textwidth}
  \centering
\includegraphics[width =  \textwidth ]{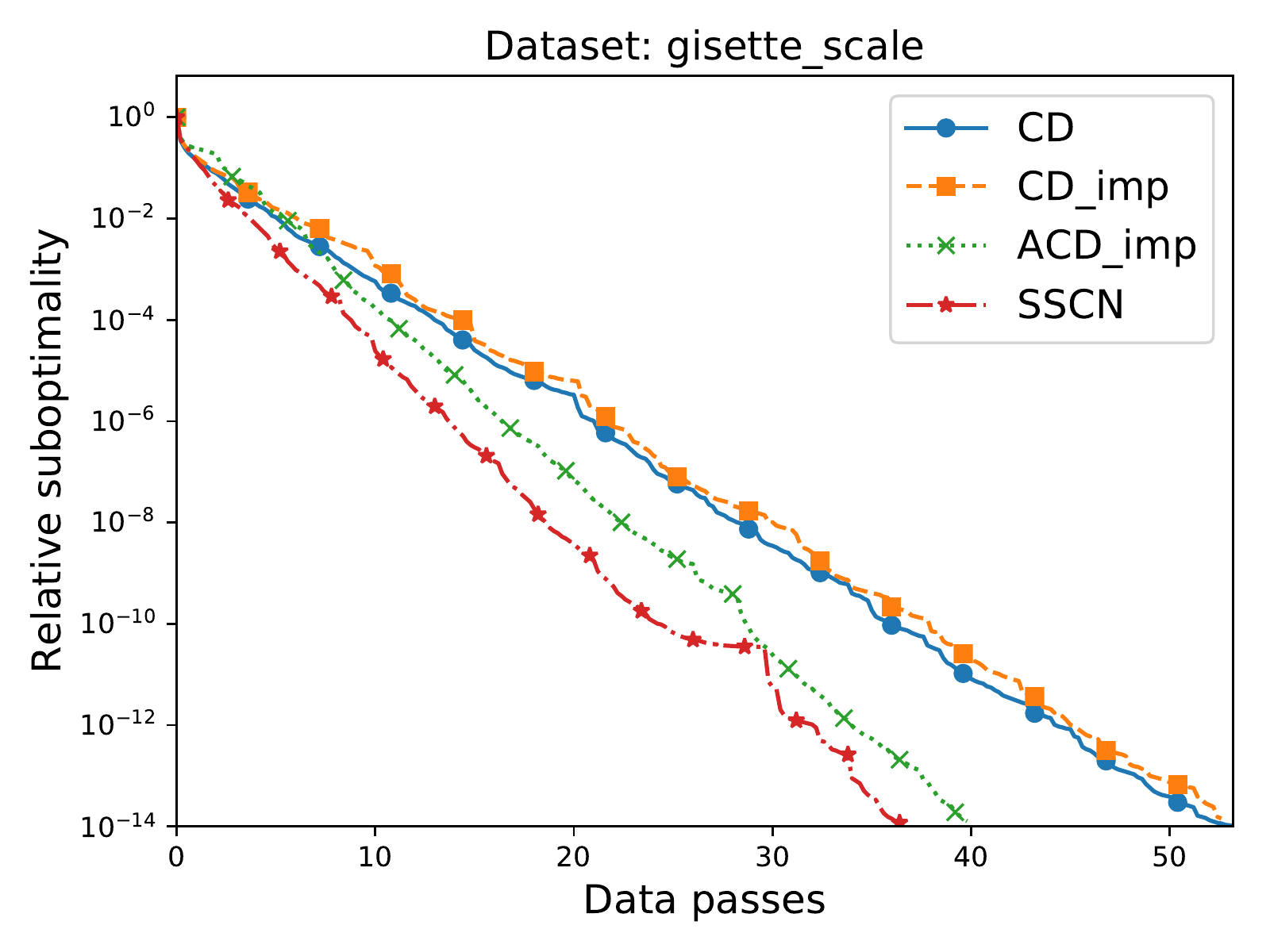}
\end{minipage}%
\begin{minipage}{0.3\textwidth}
  \centering
\includegraphics[width =  \textwidth ]{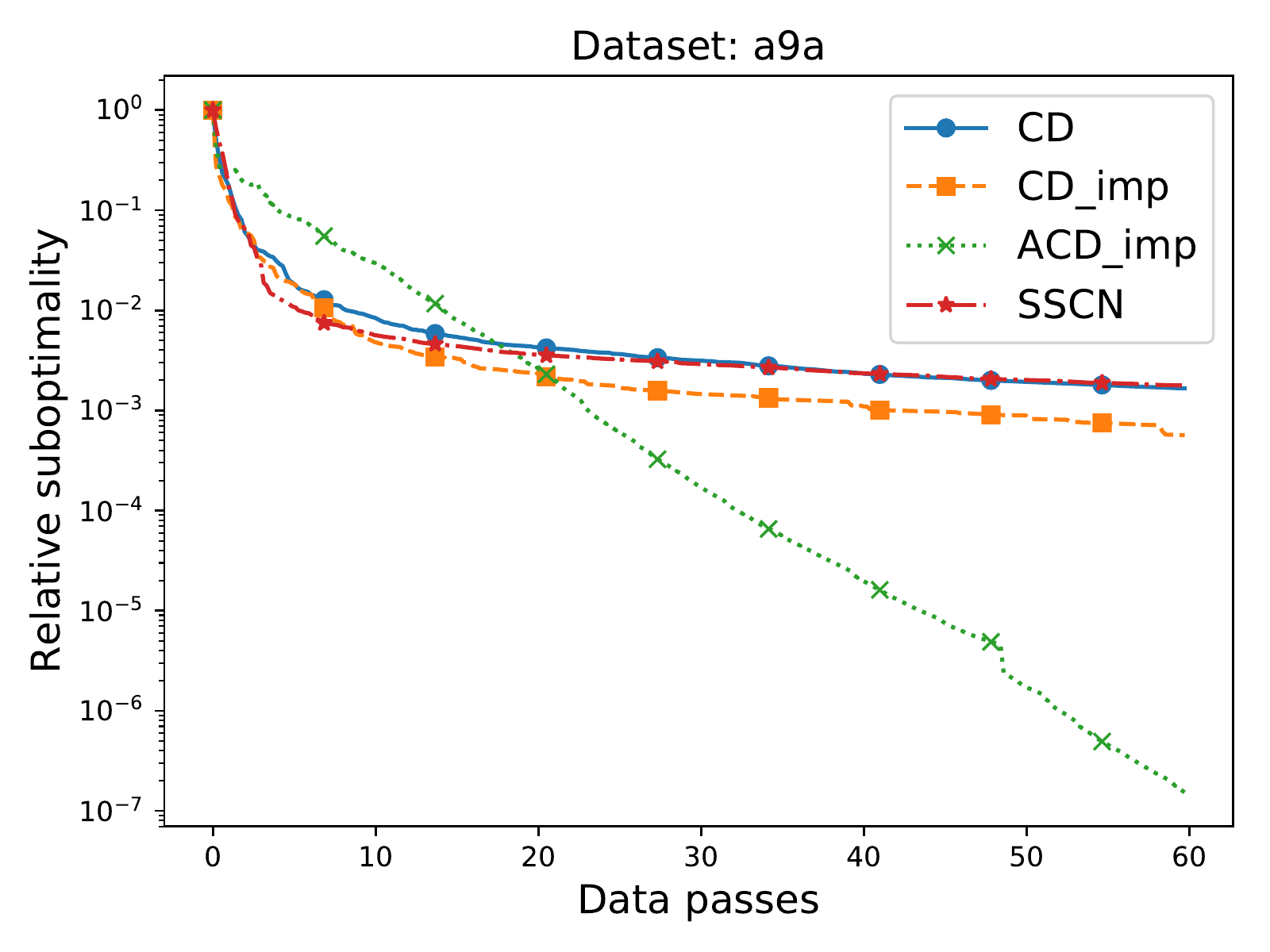}
\end{minipage}%
\begin{minipage}{0.3\textwidth}
  \centering
\includegraphics[width =  \textwidth ]{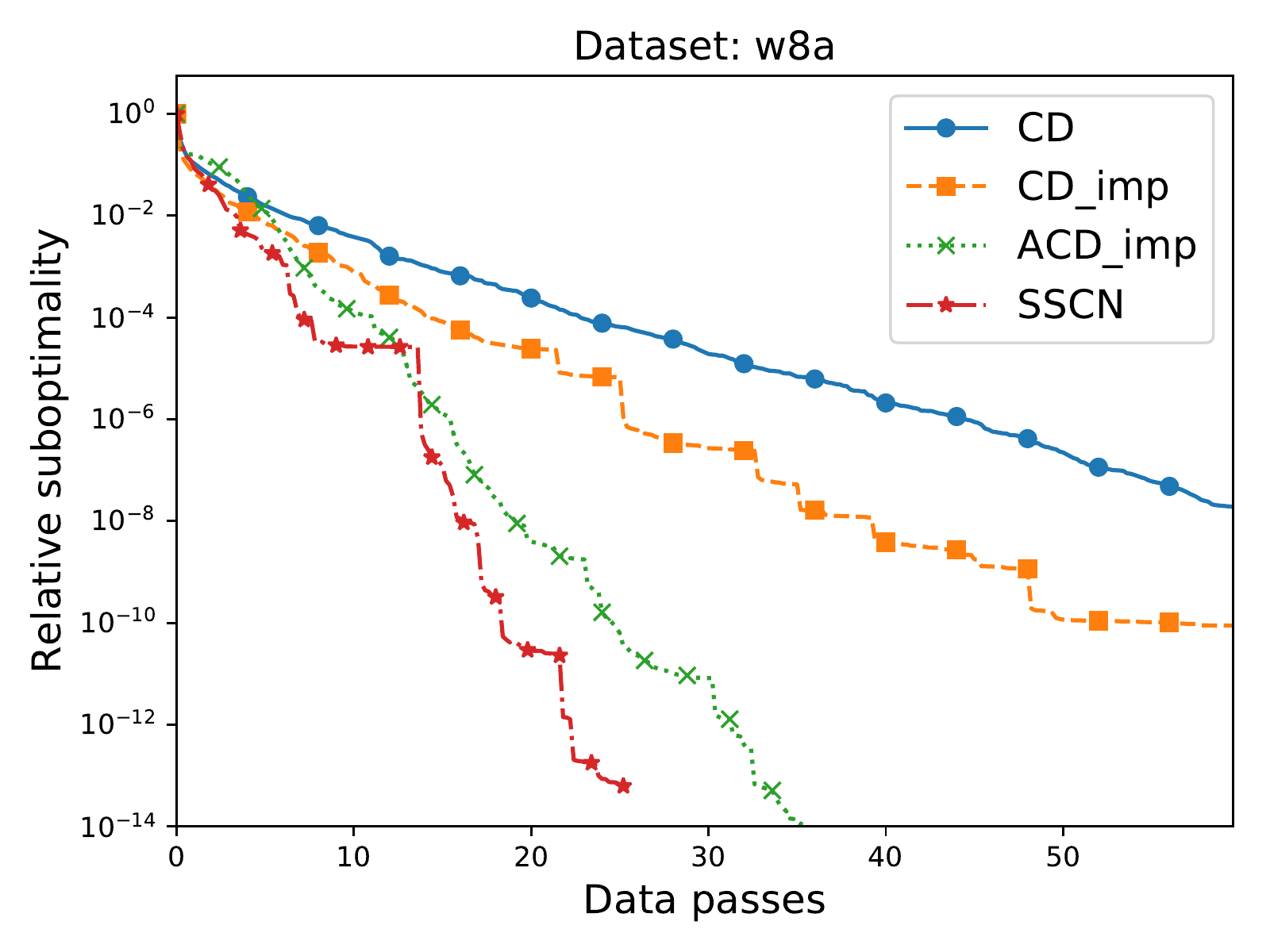}
\end{minipage}%
\caption{Comparison of CD with uniform sampling, CD with importance sampling, accelerated CD with importance sampling and SSCN (Algorithm~\ref{alg:crcd}) with uniform sampling on LibSVM datasets.} 
\label{fig:libsvm}
\end{figure}

\begin{figure}[!h]
\centering
\begin{minipage}{0.3\textwidth}
  \centering
\includegraphics[width =  \textwidth ]{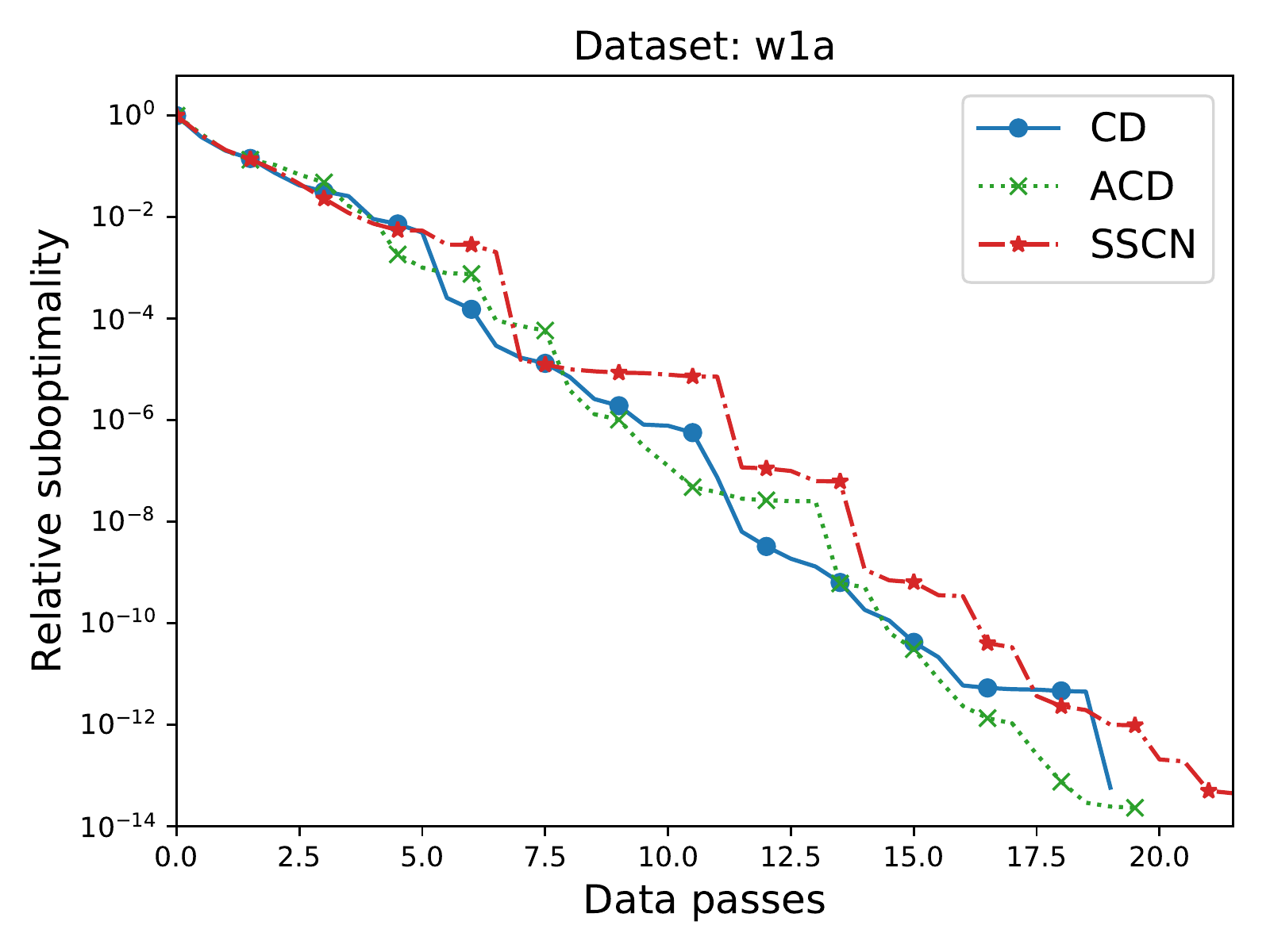}
\end{minipage}%
\begin{minipage}{0.3\textwidth}
  \centering
\includegraphics[width =  \textwidth ]{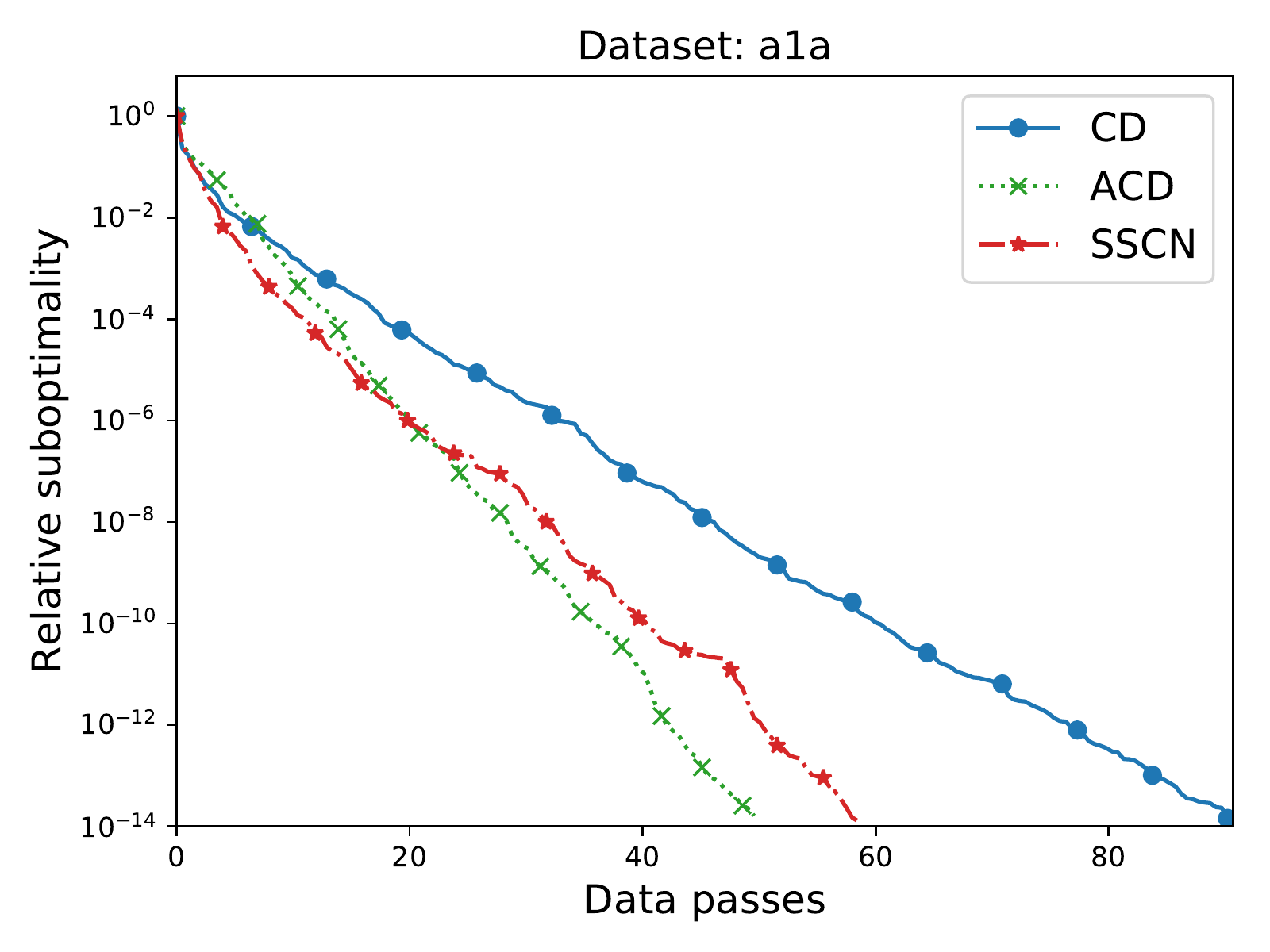}
\end{minipage}%
\begin{minipage}{0.3\textwidth}
  \centering
\includegraphics[width =  \textwidth ]{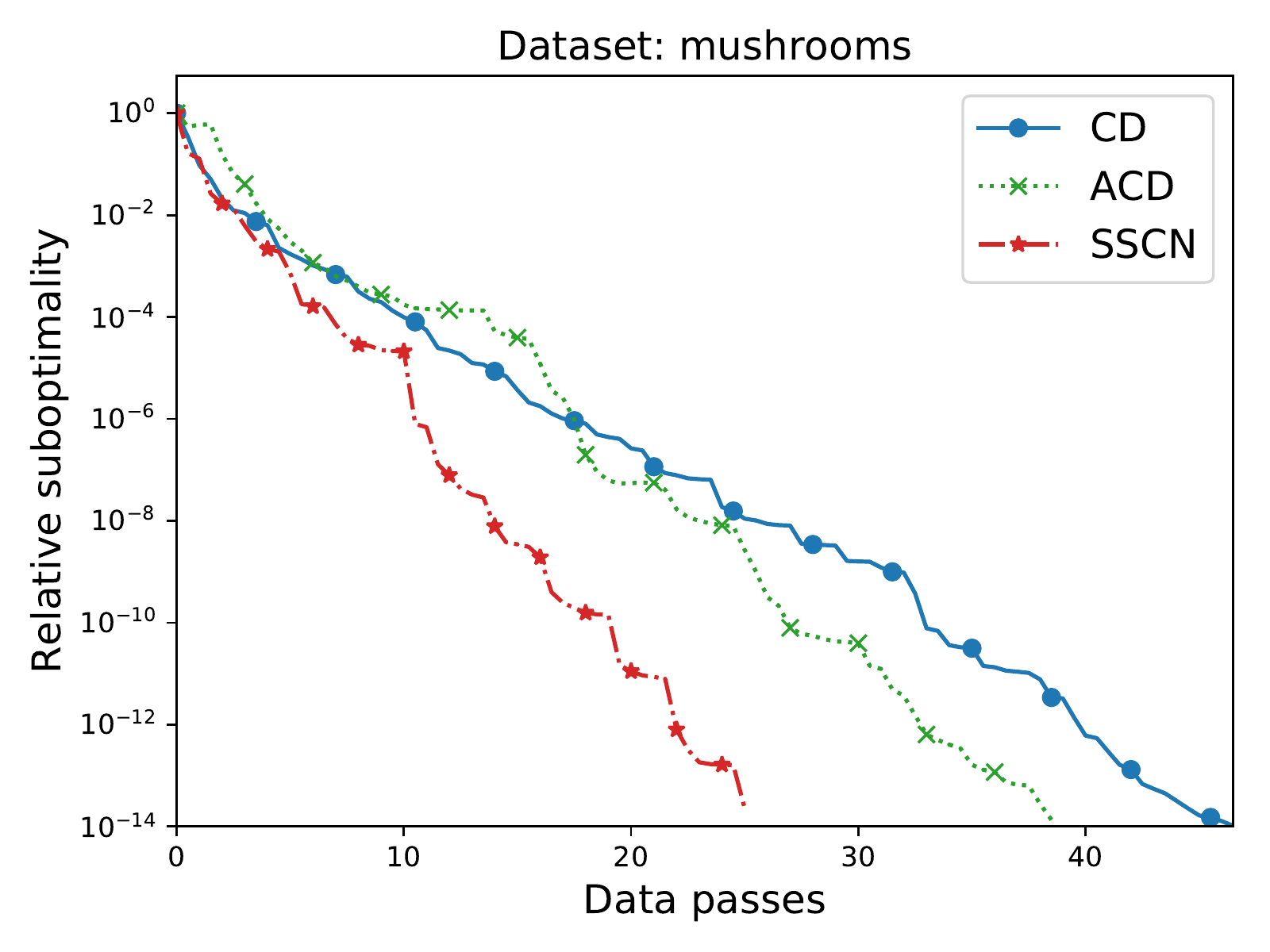}
\end{minipage}%
\\
\begin{minipage}{0.3\textwidth}
  \centering
\includegraphics[width =  \textwidth ]{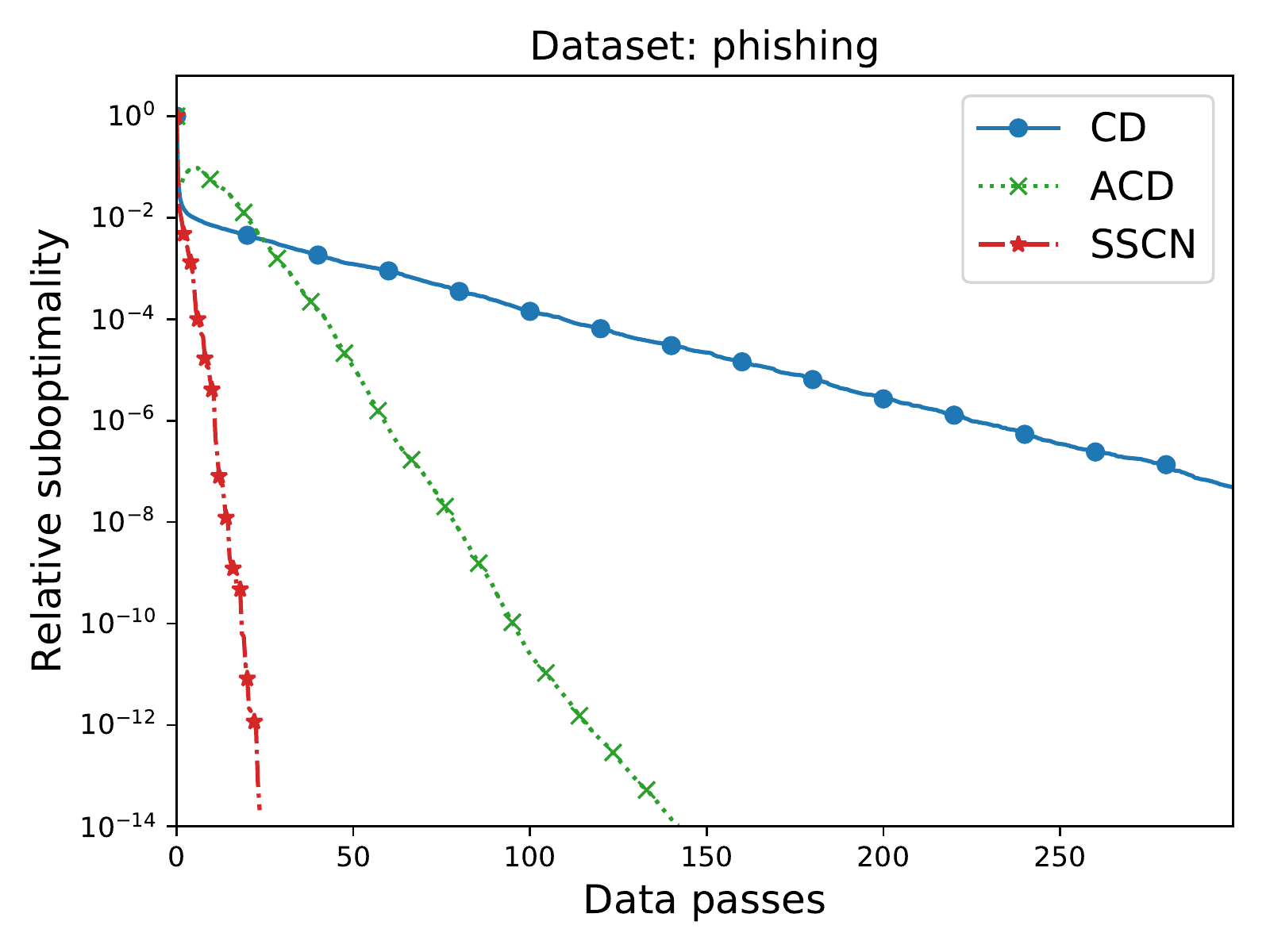}
\end{minipage}%
\begin{minipage}{0.3\textwidth}
  \centering
\includegraphics[width =  \textwidth ]{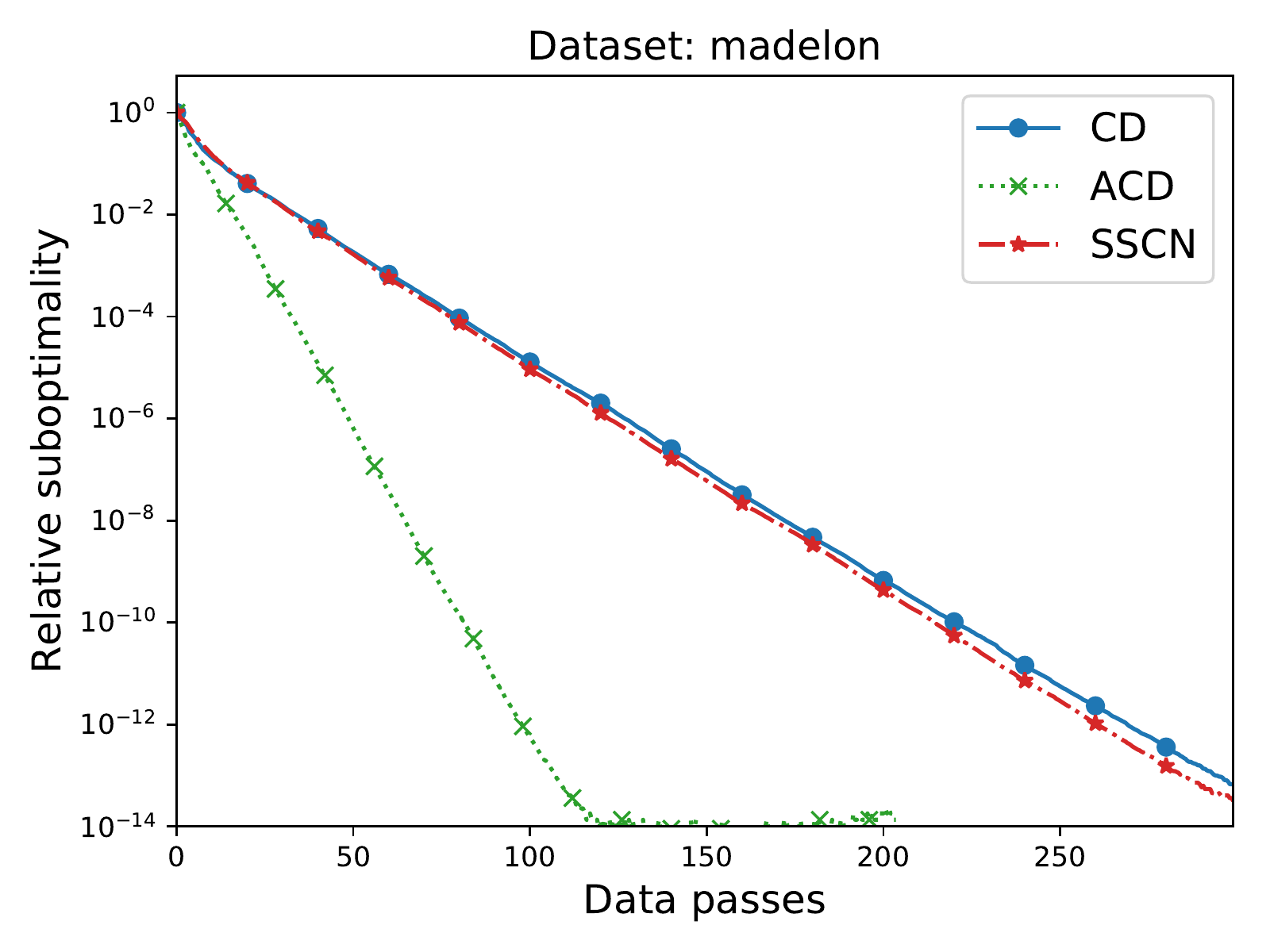}
\end{minipage}%
\begin{minipage}{0.3\textwidth}
  \centering
\includegraphics[width =  \textwidth ]{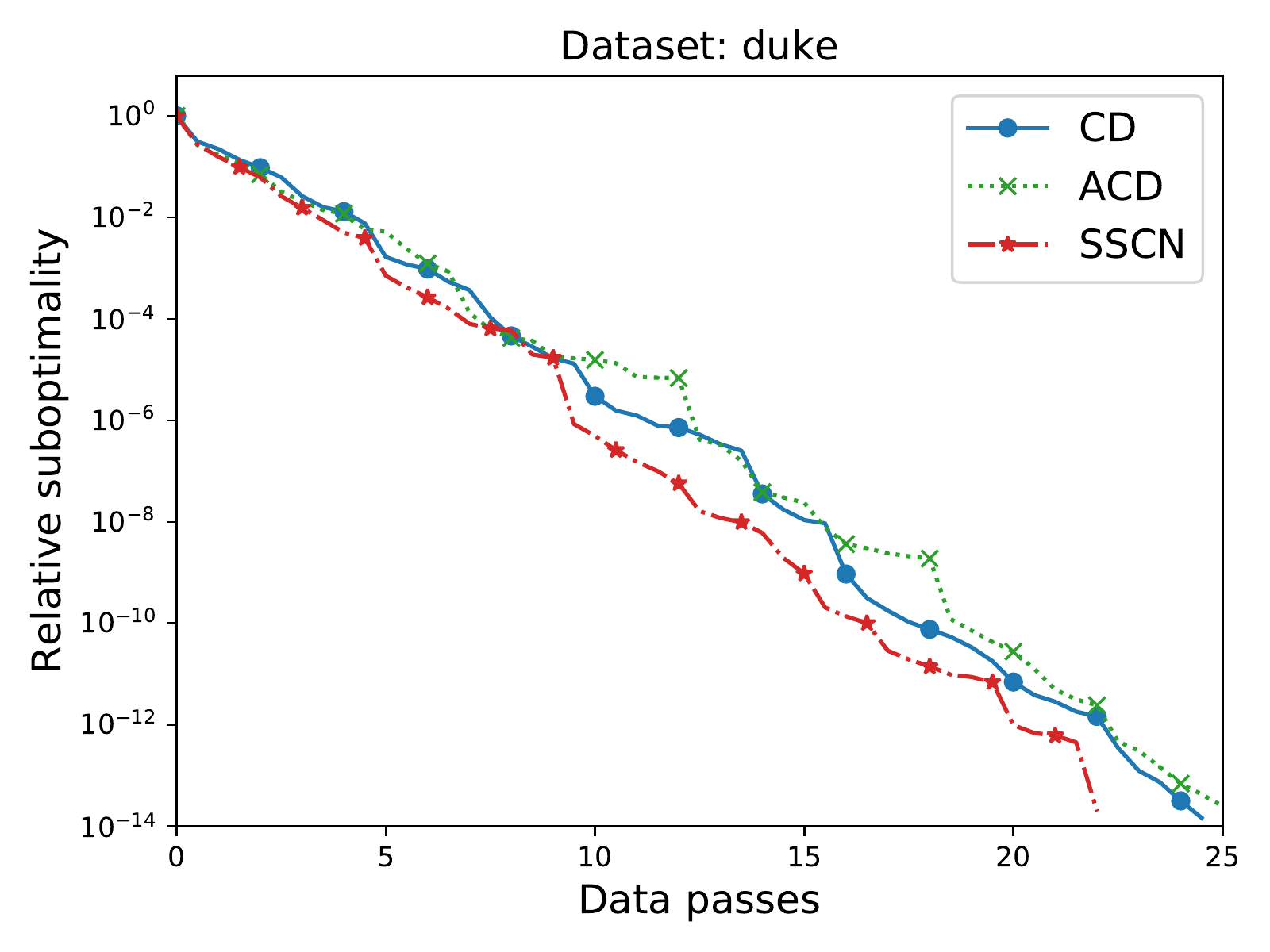}
\end{minipage}%
\\
\begin{minipage}{0.3\textwidth}
  \centering
\includegraphics[width =  \textwidth ]{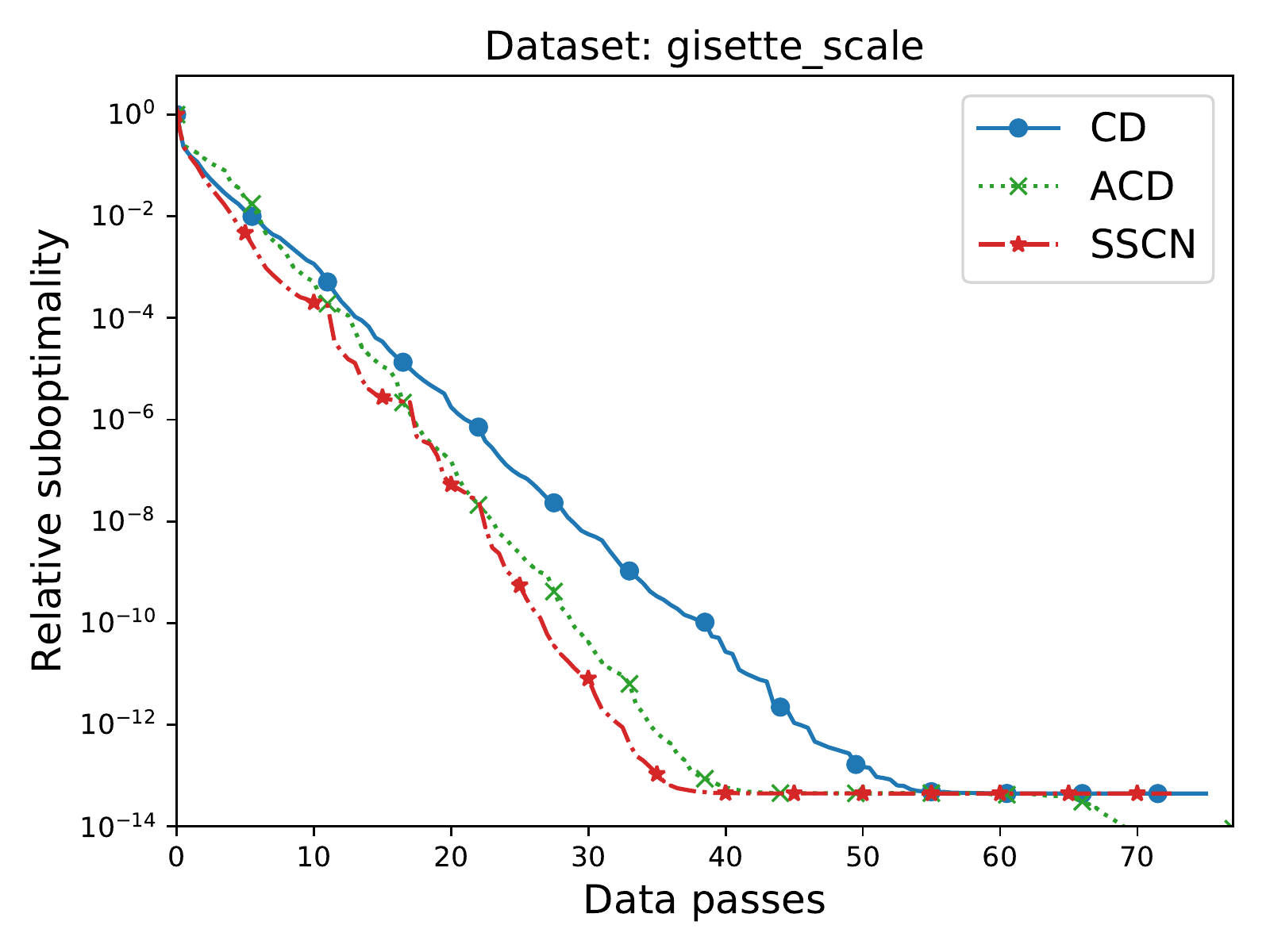}
\end{minipage}%
\begin{minipage}{0.3\textwidth}
  \centering
\includegraphics[width =  \textwidth ]{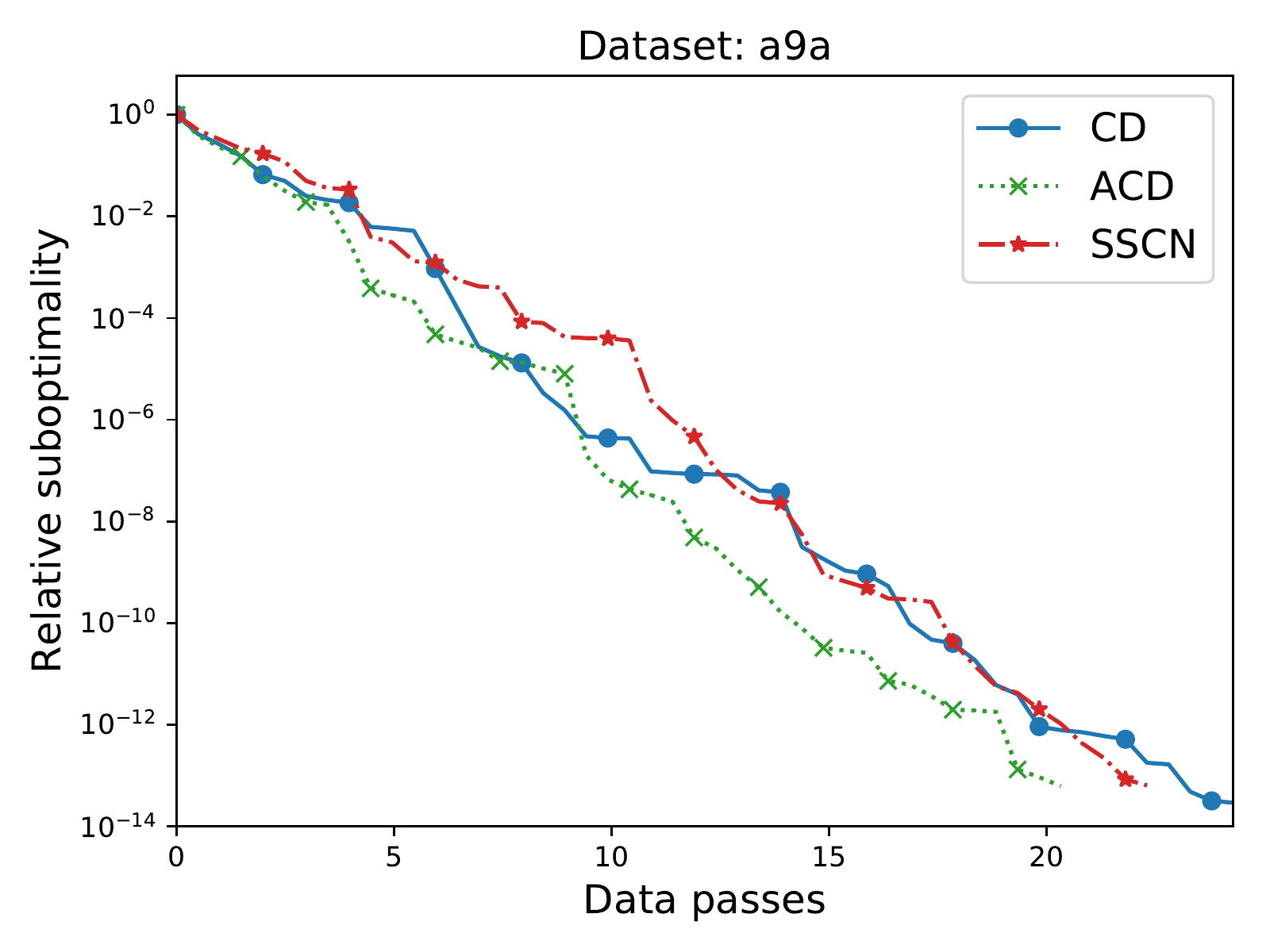}
\end{minipage}%
\begin{minipage}{0.3\textwidth}
  \centering
\includegraphics[width =  \textwidth ]{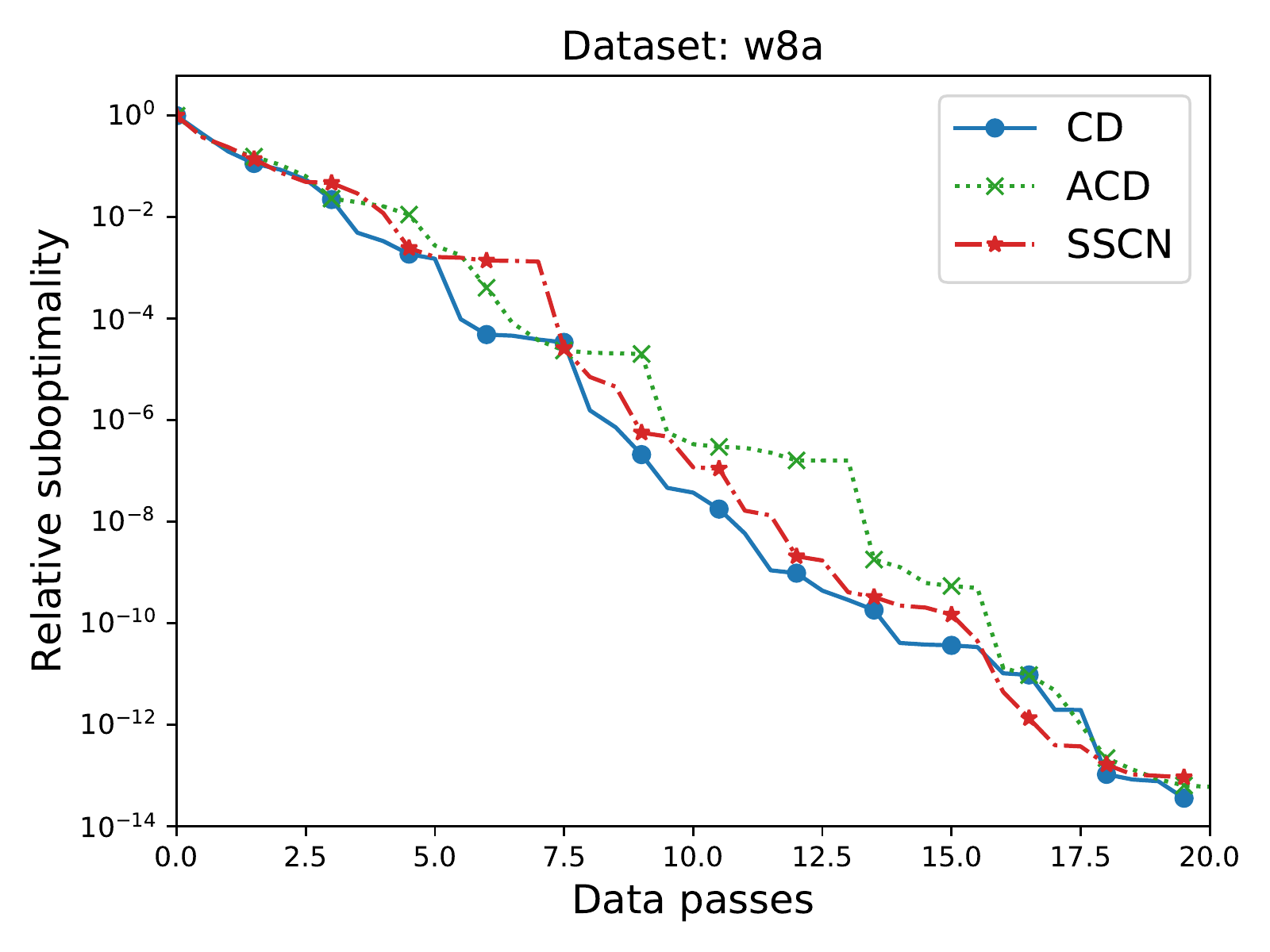}
\end{minipage}%
\caption{Comparison of coordinate descent, accelerated coordinate descent and SSCN (all with uniform sampling) on LibSVM datasets. In each case we have normalized the data matrix to have identical norms of all columns.} 
\label{fig:libsvm_normalzied}
\end{figure}

\subsubsection{Effect of sketch size}

The next experiment studies the effect of $\tau(\mS)$ on the convergence. We compare SSCN against the fastest non-accelerated first-order method -- SDNA, both with varying $\tau(\mS)$. 
We consider $\tau \in \{1,5,25\}$. In all cases, we sample uniformly --  every subset of size $\tau$ have equal chance to be chosen at every iteration (independent of the past). 

There is, however, one tricky part in terms of implementation. While we can evaluate and store $M_{e_i}$ ($i\leq d$) cheaply for linear models, this is not the case for evaluating/storing $M_S$ (at least we do not know how to do it efficiently). Therefore, we use $M_S = M$ for $|S|>1$ for SSCN. Figure~\ref{fig:libsvm_sdna} shows the result.

\begin{figure}[!h]
\centering
\begin{minipage}{0.3\textwidth}
  \centering
\includegraphics[width =  \textwidth ]{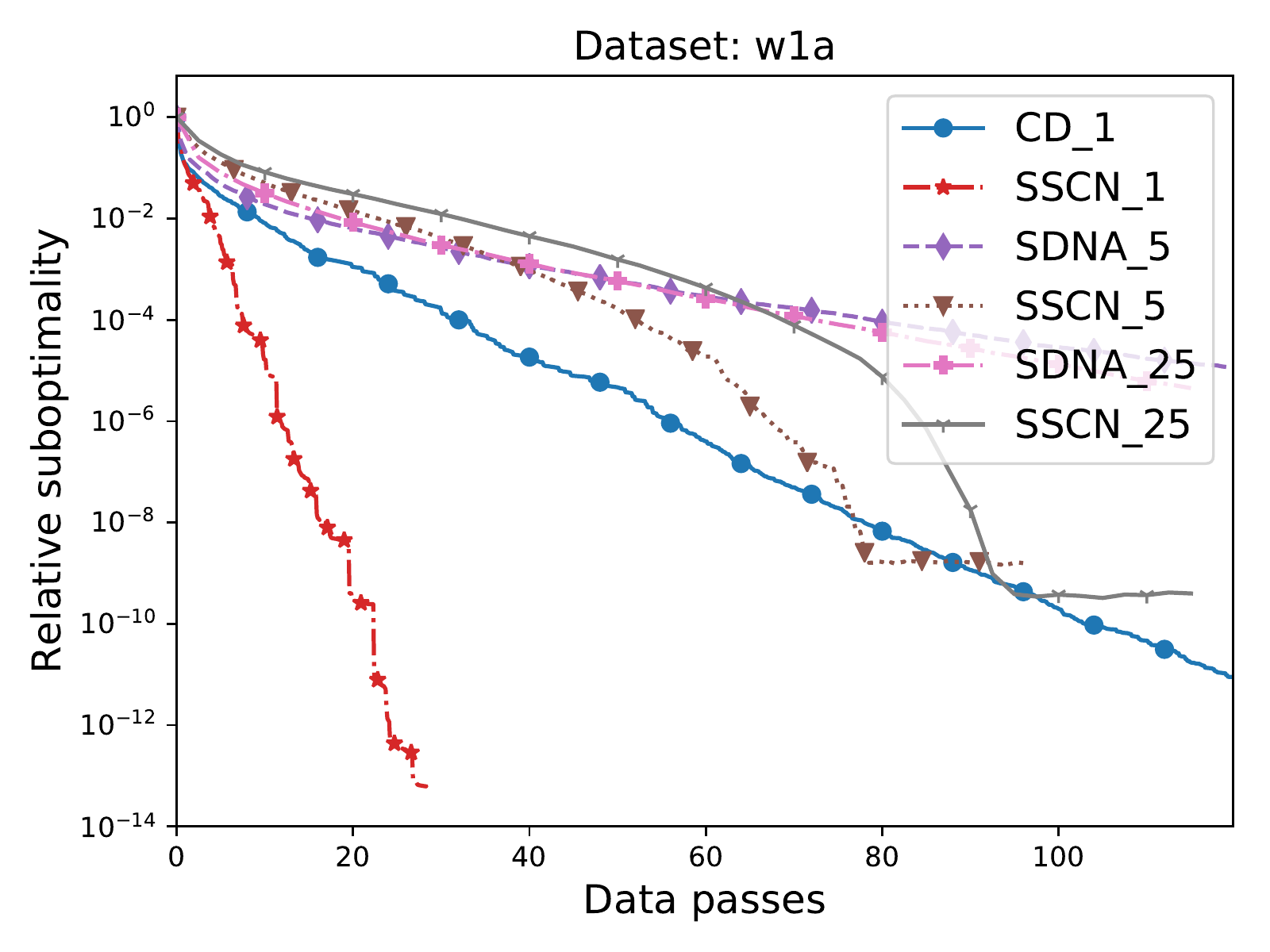}
\end{minipage}%
\begin{minipage}{0.3\textwidth}
  \centering
\includegraphics[width =  \textwidth ]{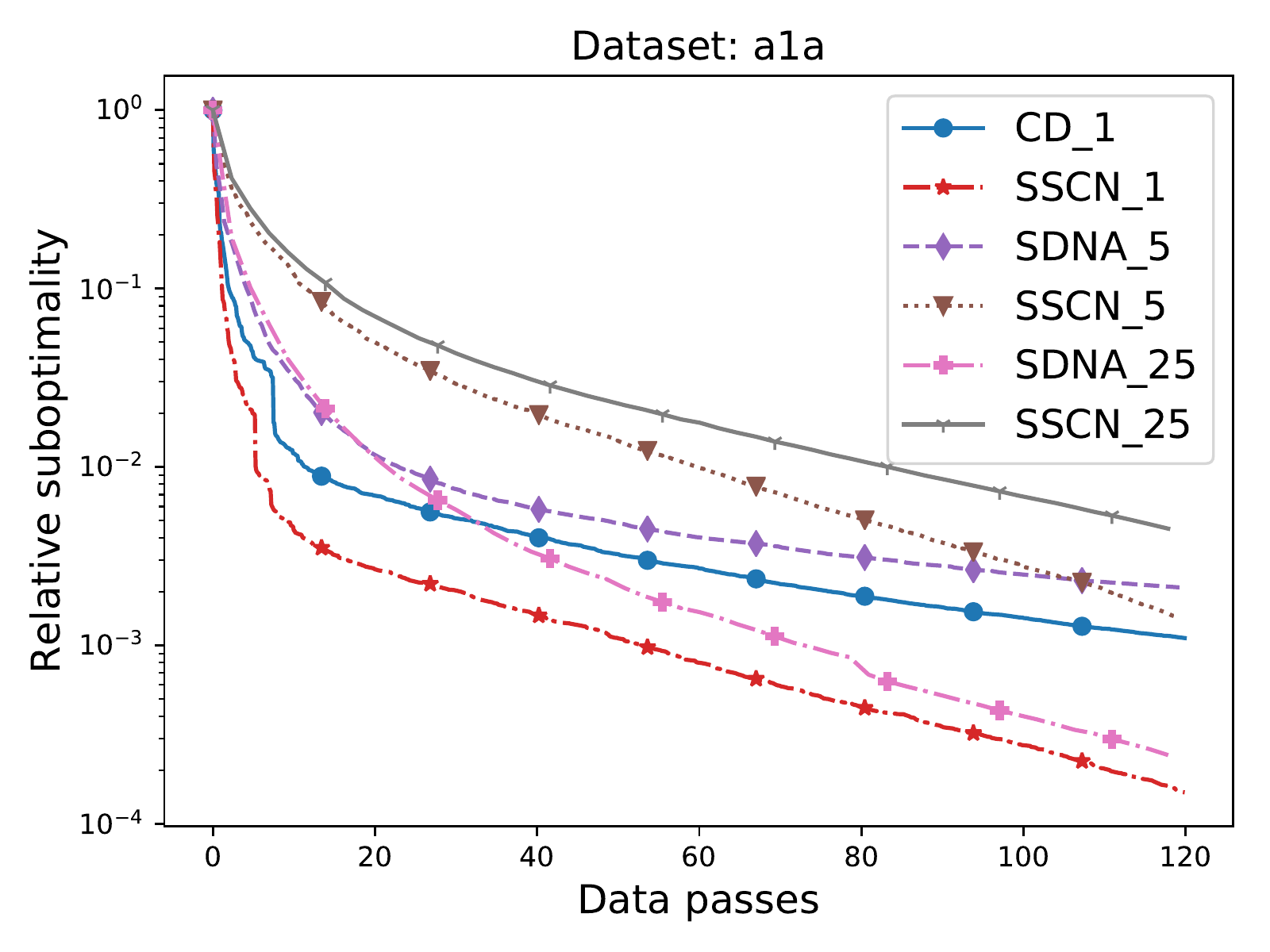}
\end{minipage}%
\begin{minipage}{0.3\textwidth}
  \centering
\includegraphics[width =  \textwidth ]{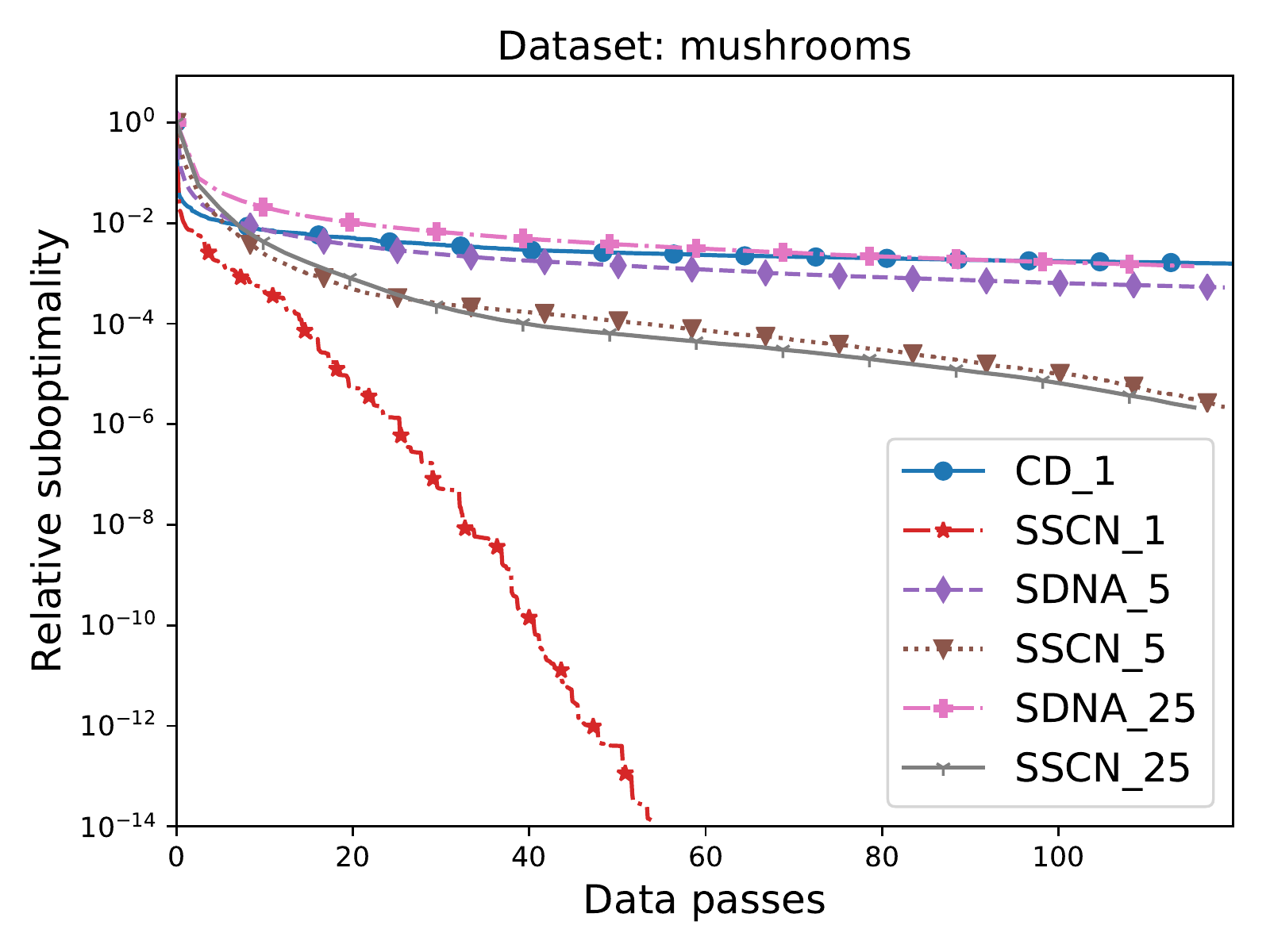}
\end{minipage}%
\\
\begin{minipage}{0.3\textwidth}
  \centering
\includegraphics[width =  \textwidth ]{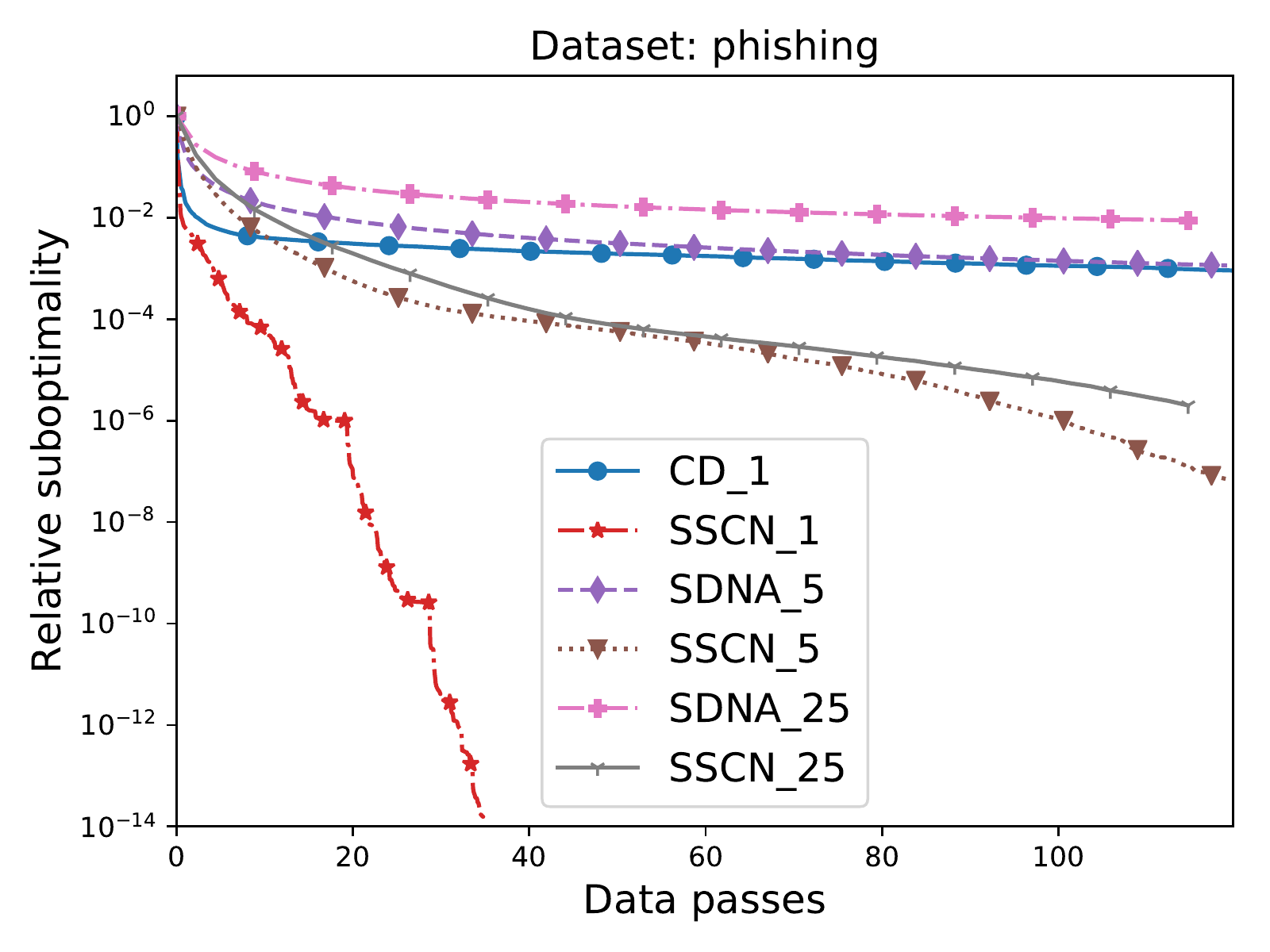}
\end{minipage}%
\begin{minipage}{0.3\textwidth}
  \centering
\includegraphics[width =  \textwidth ]{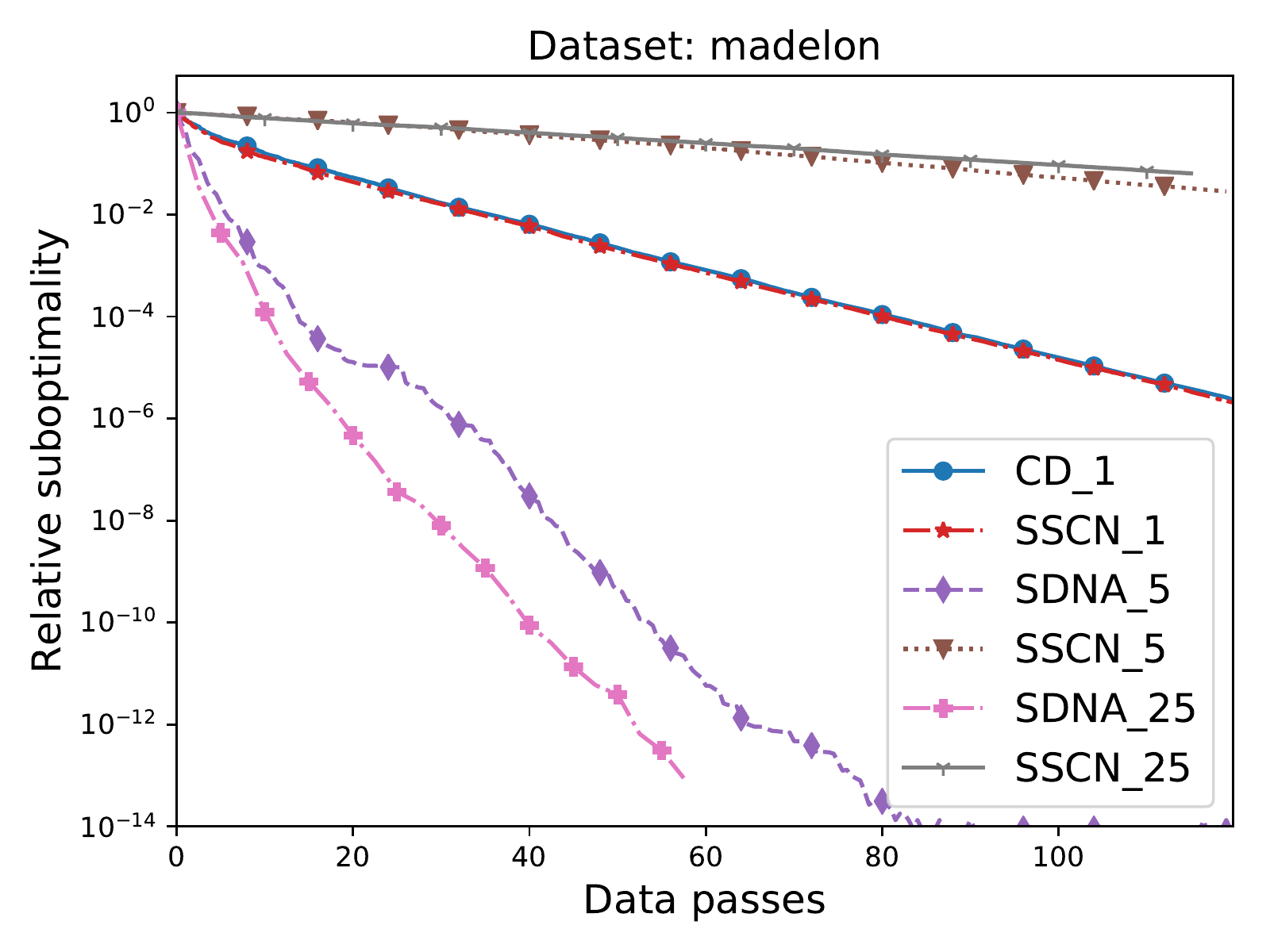}
\end{minipage}%
\begin{minipage}{0.3\textwidth}
  \centering
\includegraphics[width =  \textwidth ]{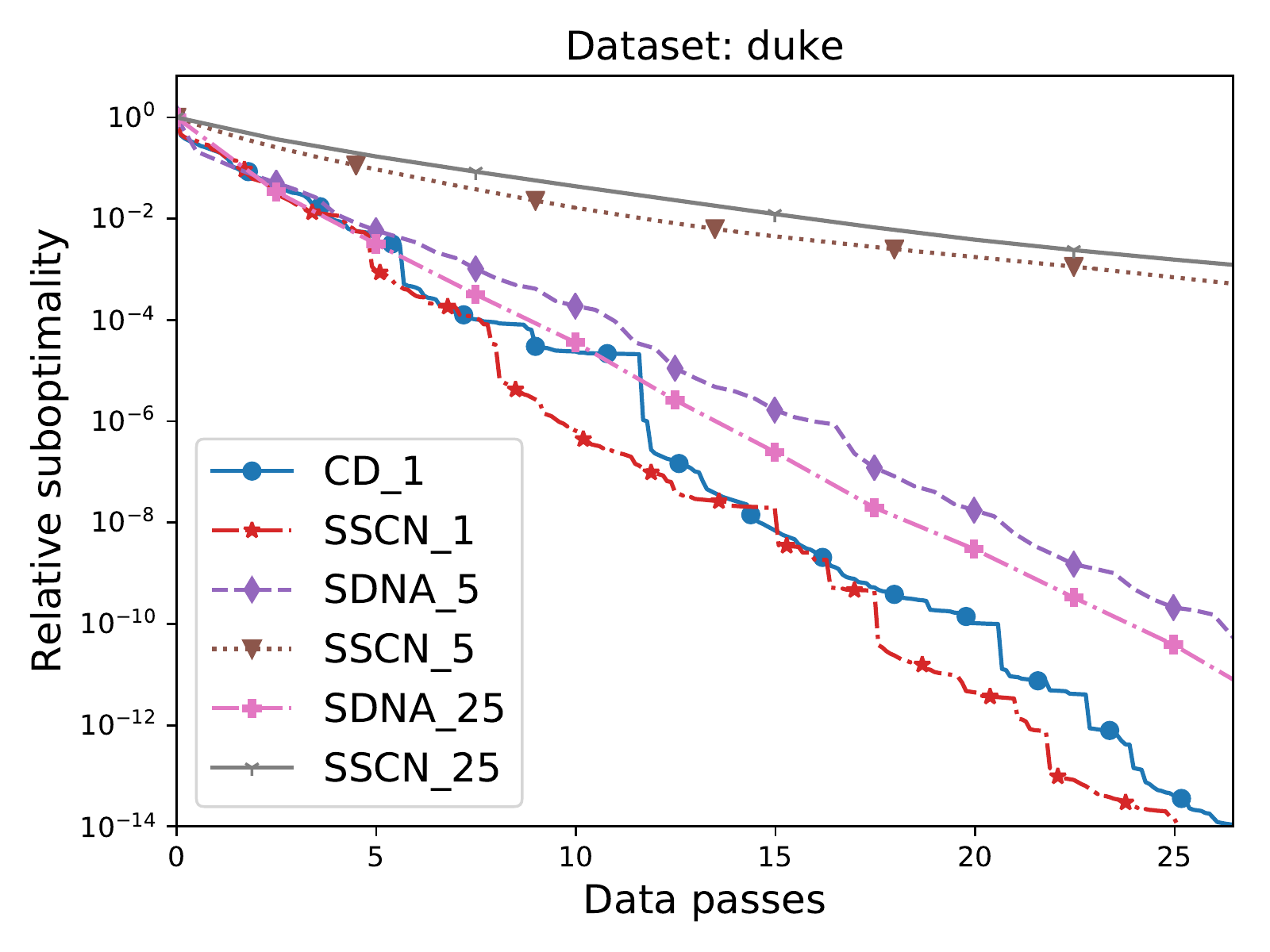}
\end{minipage}%
\\
\begin{minipage}{0.3\textwidth}
  \centering
\includegraphics[width =  \textwidth ]{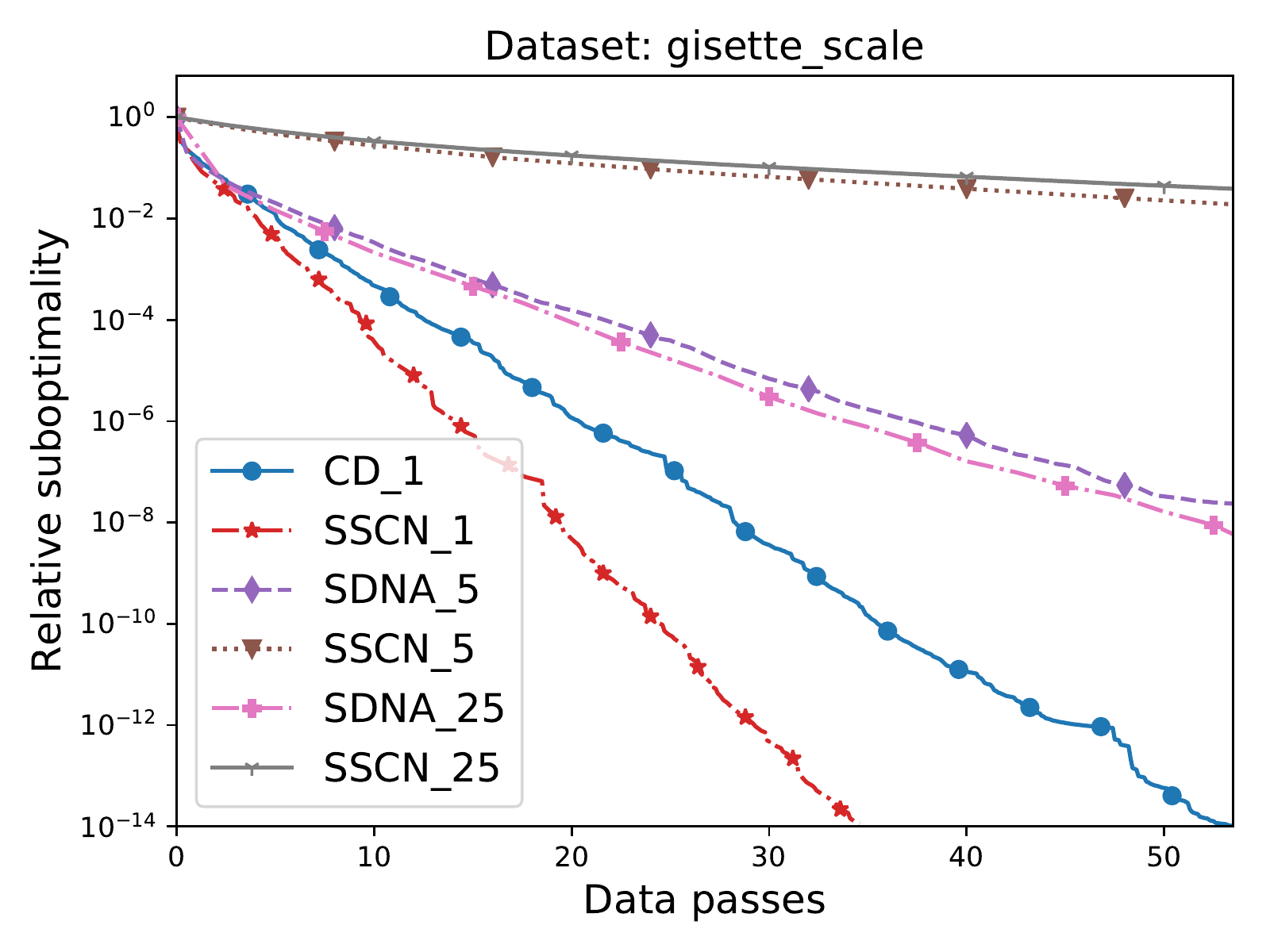}
\end{minipage}%
\begin{minipage}{0.3\textwidth}
  \centering
\includegraphics[width =  \textwidth ]{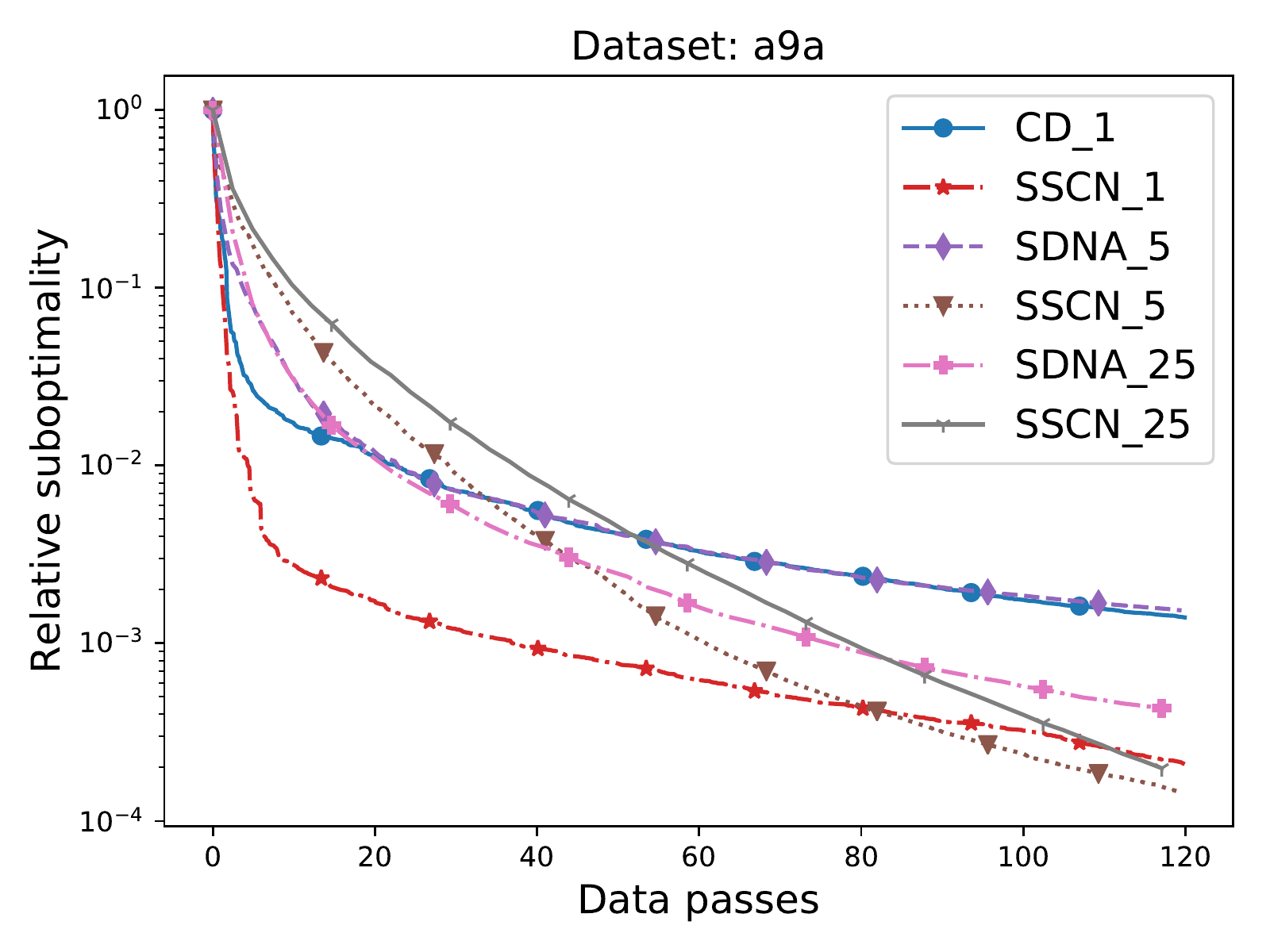}
\end{minipage}%
\begin{minipage}{0.3\textwidth}
  \centering
\includegraphics[width =  \textwidth ]{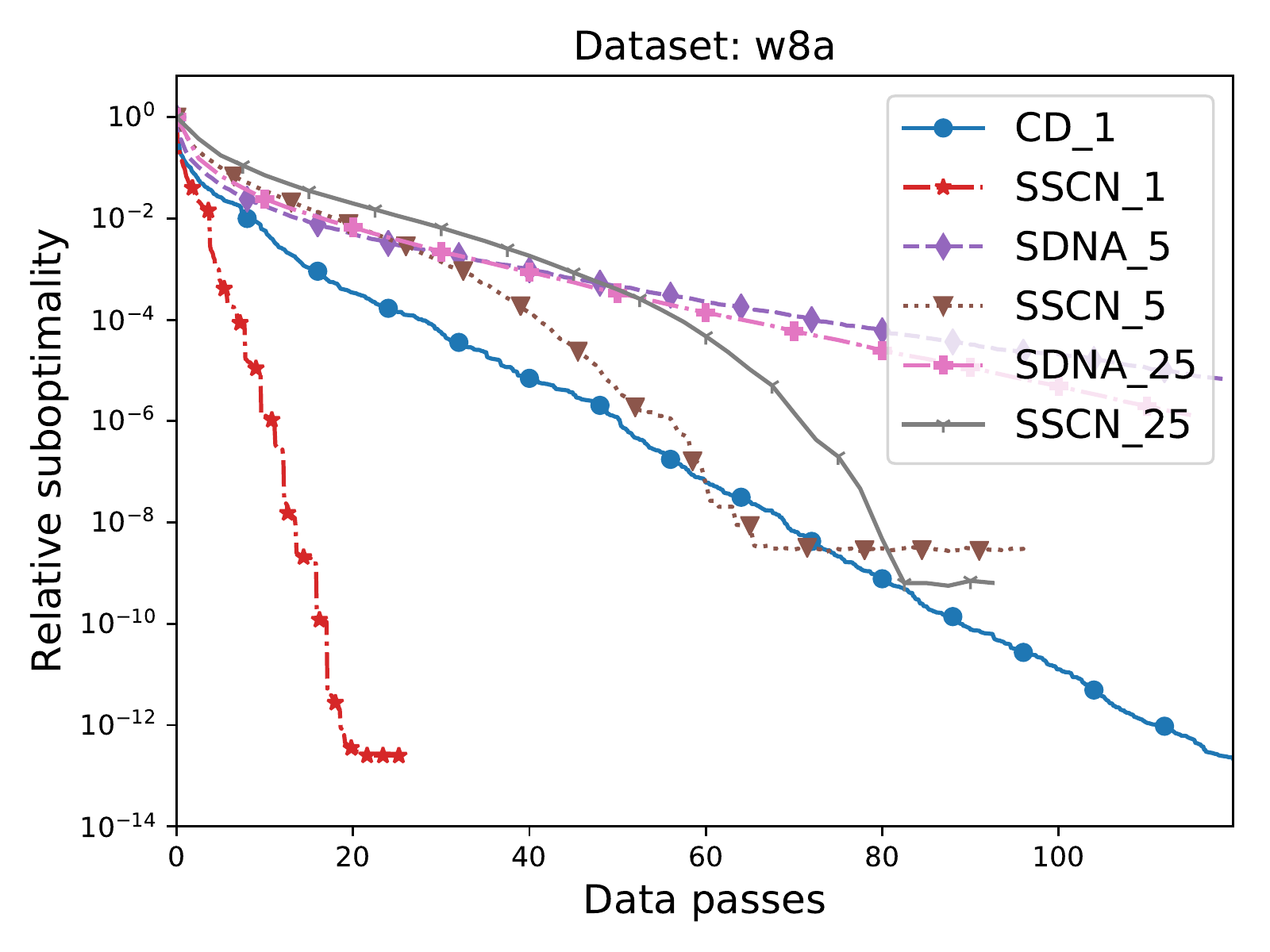}
\end{minipage}%
\caption{SSCN vs. SDNA on LibSVM datasets. All algorithms with uniform sampling.} 
\label{fig:libsvm_sdna}
\end{figure}

\subsection{Soft Maximum (Log-Sum-Exp)}

In this section, let us consider unconstrained minimization of the following Log-Sum-Exp function
$$
\ba{rcl}
f(x) & = & \displaystyle \sigma \log\left( \sum\limits_{i = 1}^n \exp \left( \frac{\la a_i, x \ra - b_i}{\sigma} \right) \right),
\qquad x \in \R^d,
\ea
$$
where $\sigma > 0$ is a \textit{smoothing} parameter, while
$a_i \in \R^d$, $1 \leq i \leq n$ and $b \in \R^n$ are given data.
This function has both Lipschitz continuous gradient and Lipschitz continuous Hessian
(see Example~1 in~\citep{doikov2019minimizing}). 

In our experiments, we first generate randomly elements of $\{ \tilde{a}_i \}_{i = 1}^n$ and $b$ from
uniform distribution on $[-1, 1]$. Then, we form an auxiliary function 
$\tilde{f}(x) := \sigma \log\Bigl( \sum\limits_{i = 1}^n \exp \bigl(  \frac{\la \tilde{a}_i, x \ra - b_i}{\sigma} \bigr) \Bigr)$,
using these parameters, and set
$$
\ba{rcl}
a_i & := & \tilde{a}_i - \nabla \tilde{f}(0), \quad 1 \leq i \leq n.
\ea
$$
Thus, we essentially obtain the optimum $x^{*}$ of $f$ in the origin, since $\nabla f(0) = 0$.

We use $x_0 \eqdef e$ (vector of all ones) as a starting point,
and always set $n \eqdef 6d$.

For this problem, we compare the performance of SSCN with the first-order Coordinate Descent (CD), using uniform samples of coordinates $S \subseteq [d]$ of a fixed size $\tau = |S|$. 

Note, that keeping scalar products $\{ \la a_i, x_k \ra \}_{i = 1}^n$
precomputed for a current point $x_k$, we are able to compute the partial gradient
$\nabla_{\mS} f(x^k)$ in time $O(\tau n)$ and the partial Hessian $\nabla^2_{\mS} f(x^k)$ in time $O(\tau^2 n)$.
To find the next direction $h^k$ of SSCN (solving the Cubic subproblem), we call Nonlinear Conjugate Gradient method,
and use the following condition as a stopping criterion:
$$
\ba{rcl}
\| \nabla_h T_{\mS}(x^k; h^k) \| & \leq & 10^{-4},
\ea
$$
where $T_{\mS}(x^k; h)  := \la \nabla_{\mS} f(x^k), h \ra + \frac{1}{2}\la \nabla_{\mS}^2 f(x^k)h, h \ra + \frac{M_k}{6}\|\mS h\|^3$ is the Cubic model, and $M_k \geq 0$ is a regularization constant.

For both methods, we use one-dimensional search at every iteration, to fit the corresponding parameter:
\begin{enumerate}
	\item For the Coordinate Descent, we find $L_k$ such that $f(x^{k}) - f(x^{k + 1}) \geq \frac{1}{2L_k}\| \nabla_{\mS} f(x^k) \|^2$, 
	where $x^{k + 1}$ is the next point of the method: $x^{k + 1} = x^k + \frac{1}{L_k}\mS \nabla_{\mS} f(x^k)$.
	
	\item For SSCN, we find $M_k$ such that~\eqref{eq:coordinate_ub_full}
	is satisfied, i.e., $f(x^{k}) - f(x^{k + 1}) \geq -T_{\mS}(x^k, h^k)$.
\end{enumerate}
Therefore,  we need to evaluate the function value inside the procedure, which is not very expensive.

The results are shown on Figures~\ref{fig:log_sum_exp_500},\ref{fig:log_sum_exp_1000},
for $d = 500$ and $1000$ respectively\footnote{Clock time was evaluated using the machine with Intel Xeon Gold 6146 CPU, 3.20GHz; 251 GB RAM.}.
We see, that SSCN outperforms CD significantly in terms of the iteration rate.
For SSCN with a medium batchsize $\tau$, we may obtain the best performance 
in terms of the total computational time.

\begin{figure}[h!]
	\centering
	\begin{minipage}{0.23\textwidth}
		\centering
		\includegraphics[width =  \textwidth ]{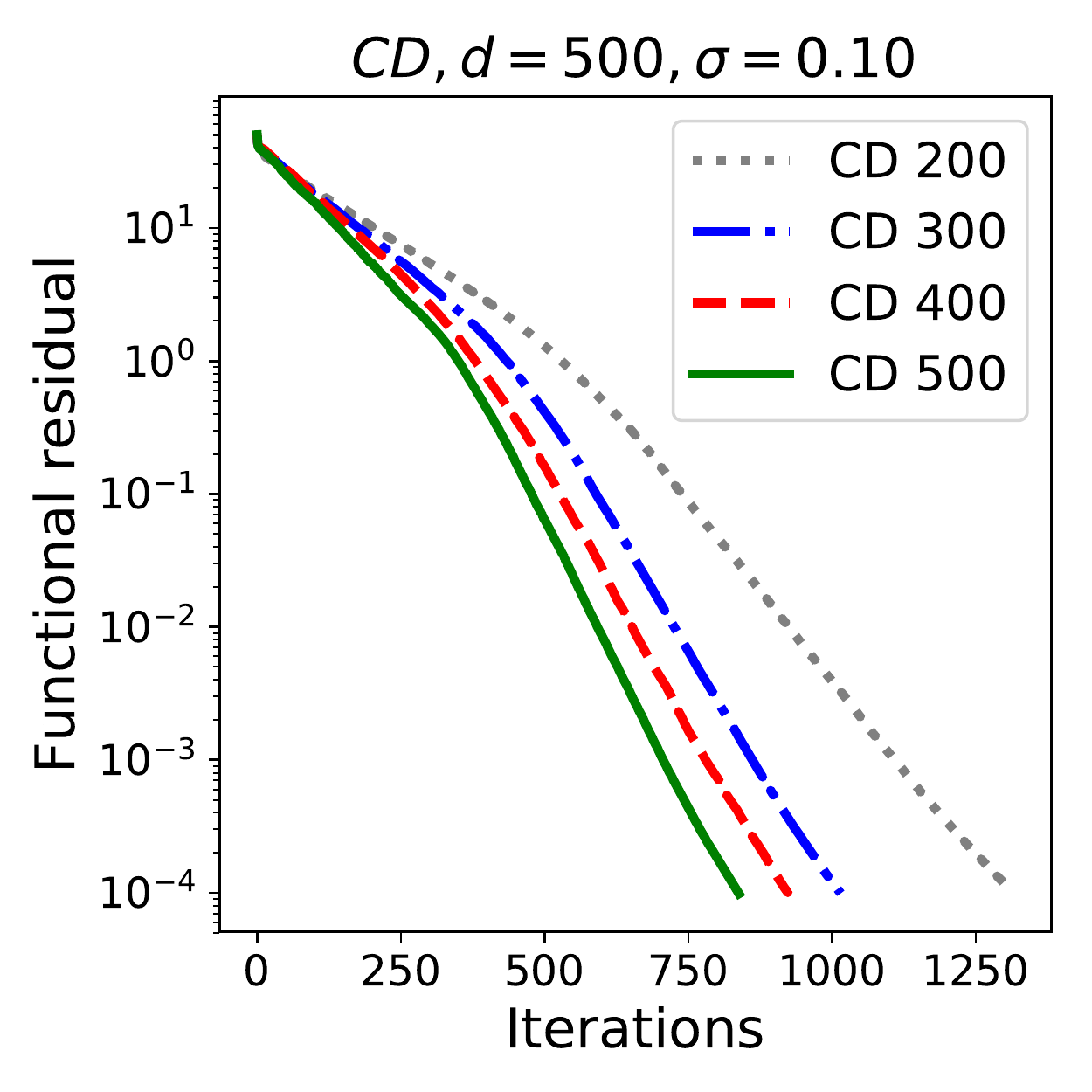}
	\end{minipage}
	\begin{minipage}{0.23\textwidth}
		\centering
		\includegraphics[width =  \textwidth ]{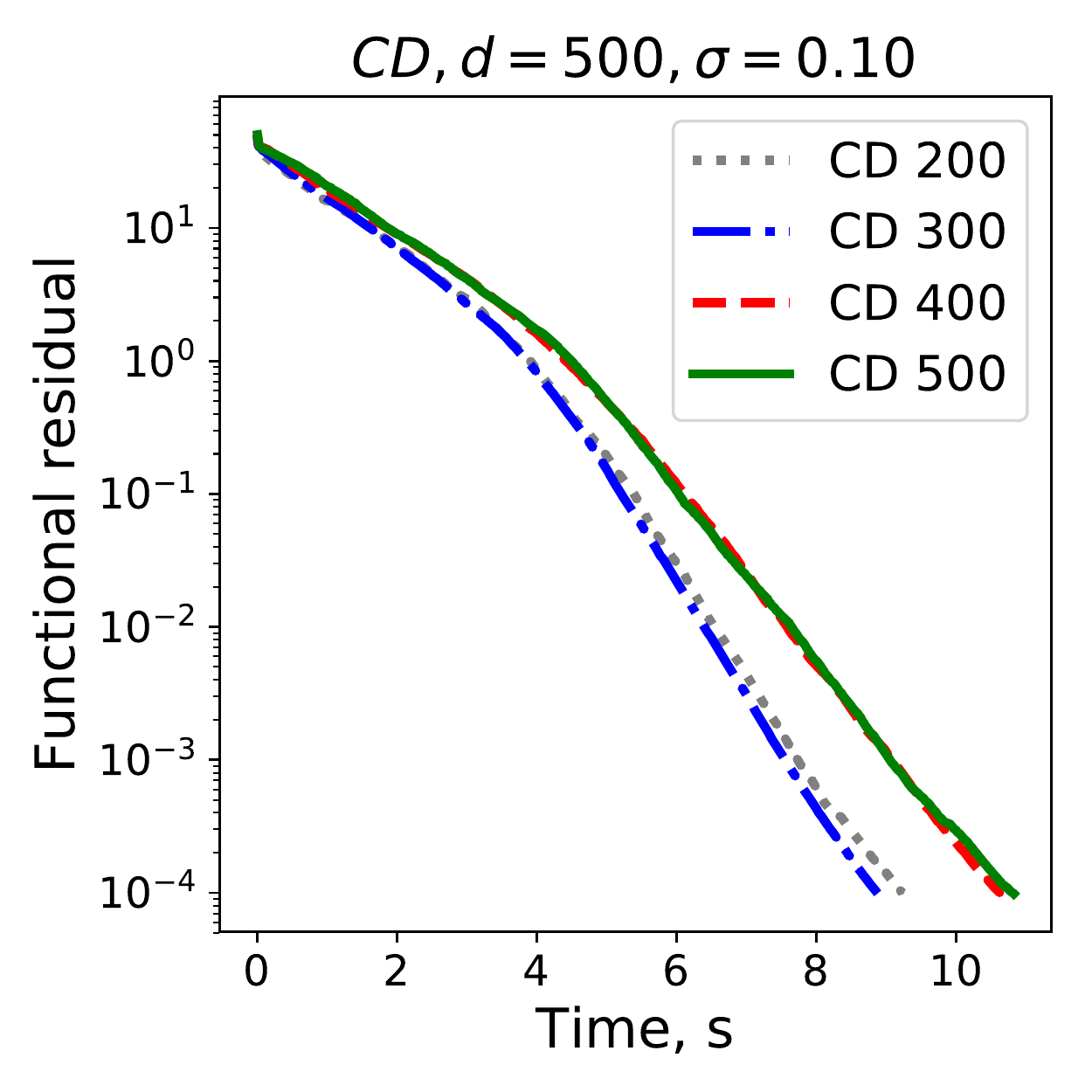}
	\end{minipage}
	\begin{minipage}{0.23\textwidth}
		\centering
		\includegraphics[width =  \textwidth ]{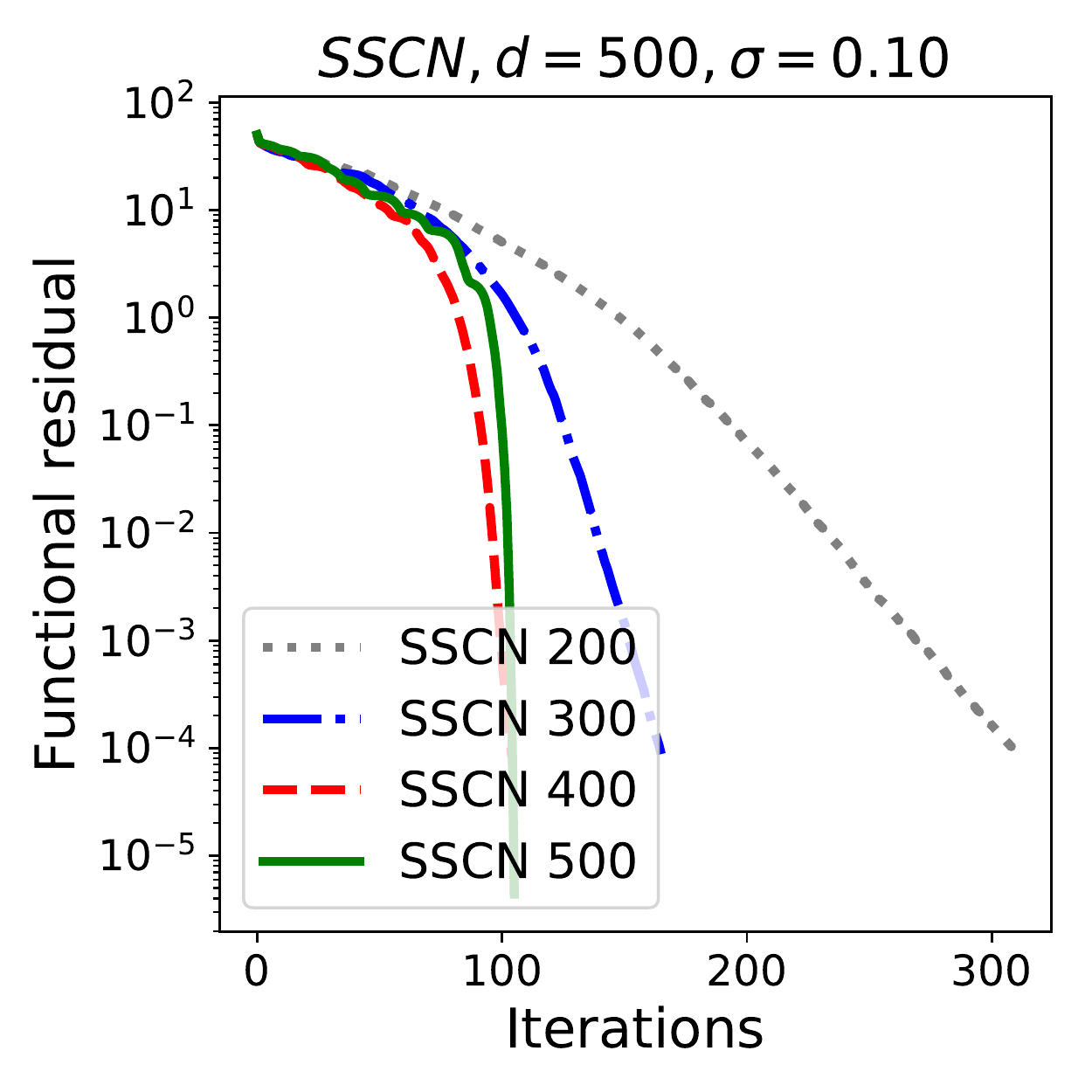}
	\end{minipage}
	\begin{minipage}{0.23\textwidth}
		\centering
		\includegraphics[width =  \textwidth ]{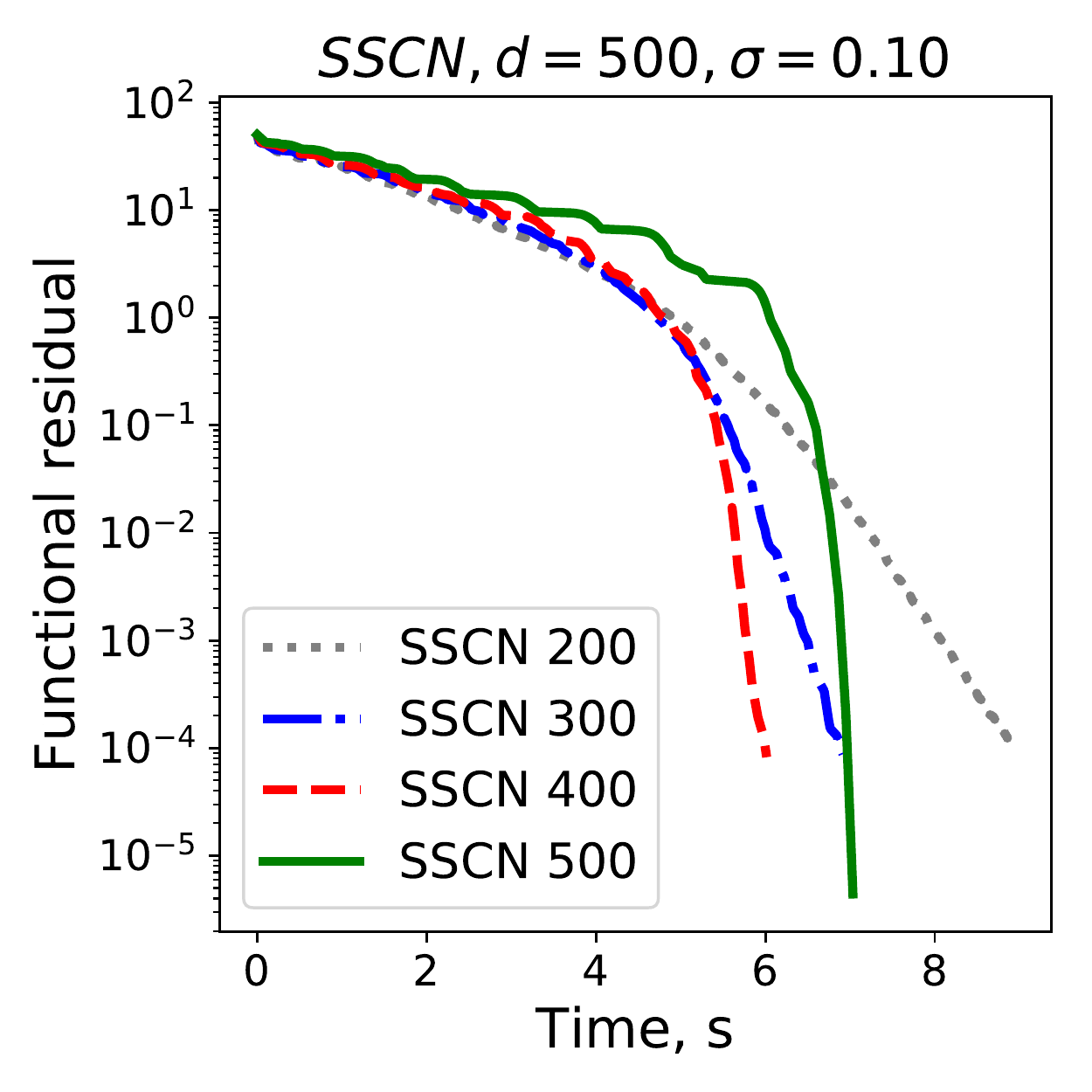}
	\end{minipage}

		\begin{minipage}{0.23\textwidth}
		\centering
		\includegraphics[width =  \textwidth ]{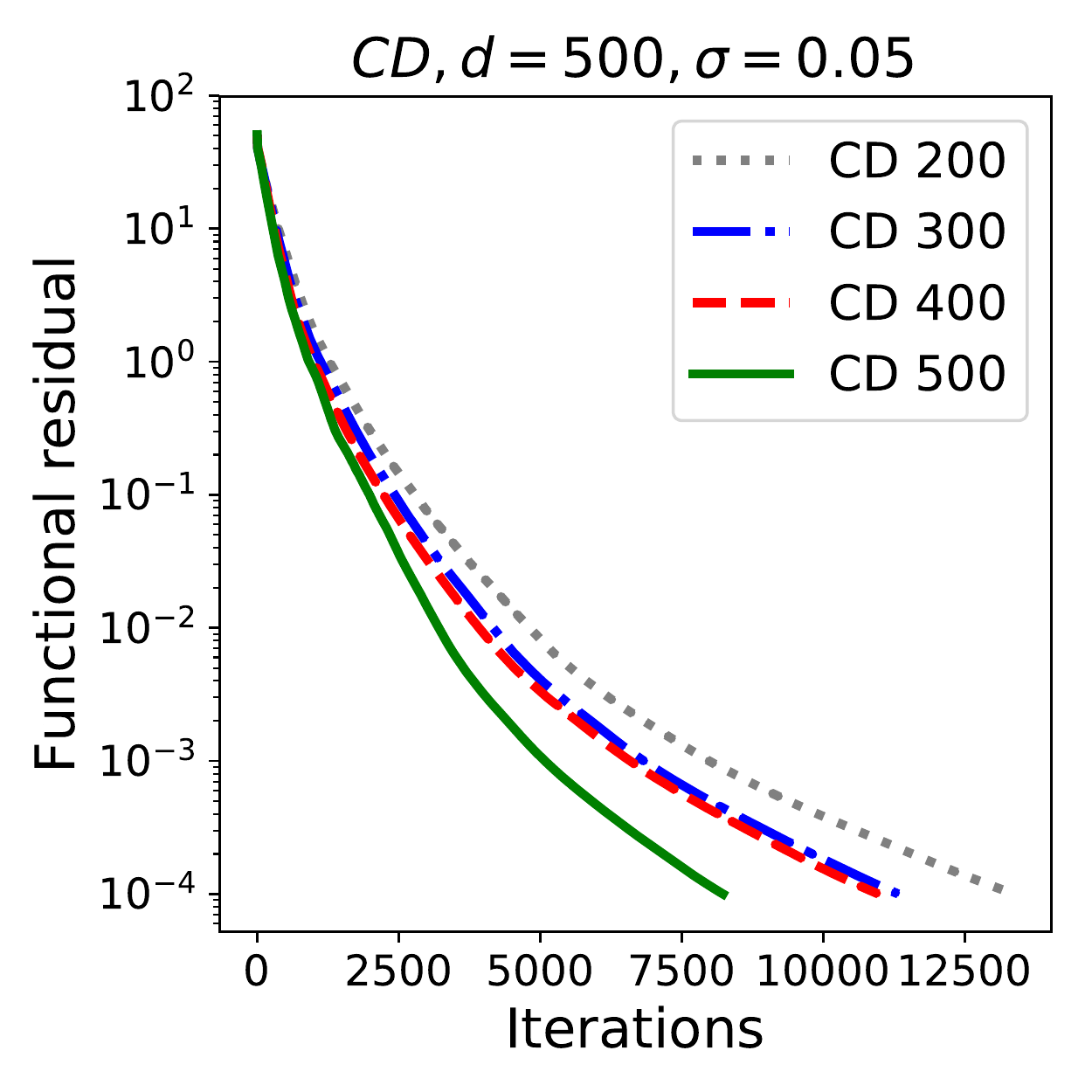}
	\end{minipage}
	\begin{minipage}{0.23\textwidth}
		\centering
		\includegraphics[width =  \textwidth ]{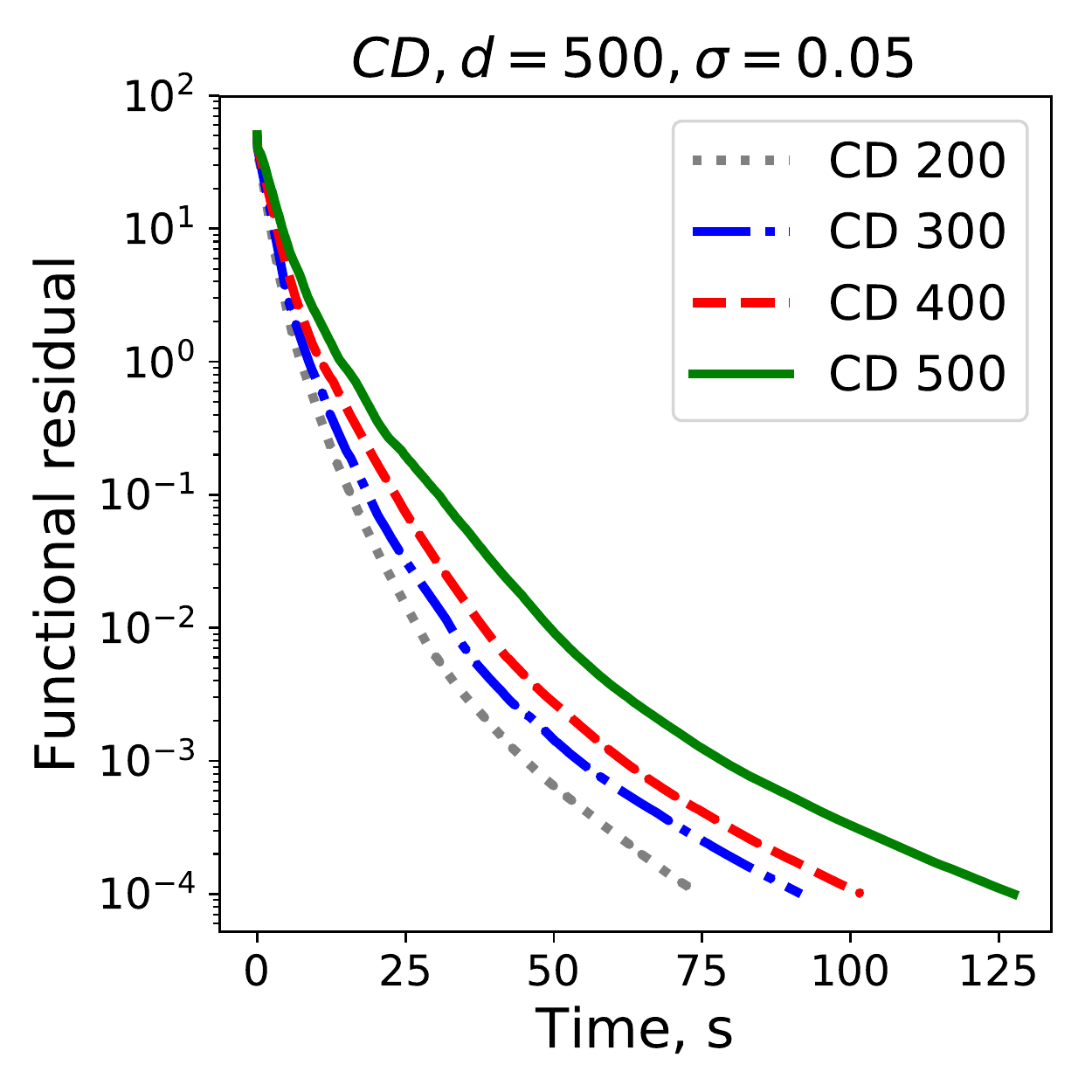}
	\end{minipage}
	\begin{minipage}{0.23\textwidth}
		\centering
		\includegraphics[width =  \textwidth ]{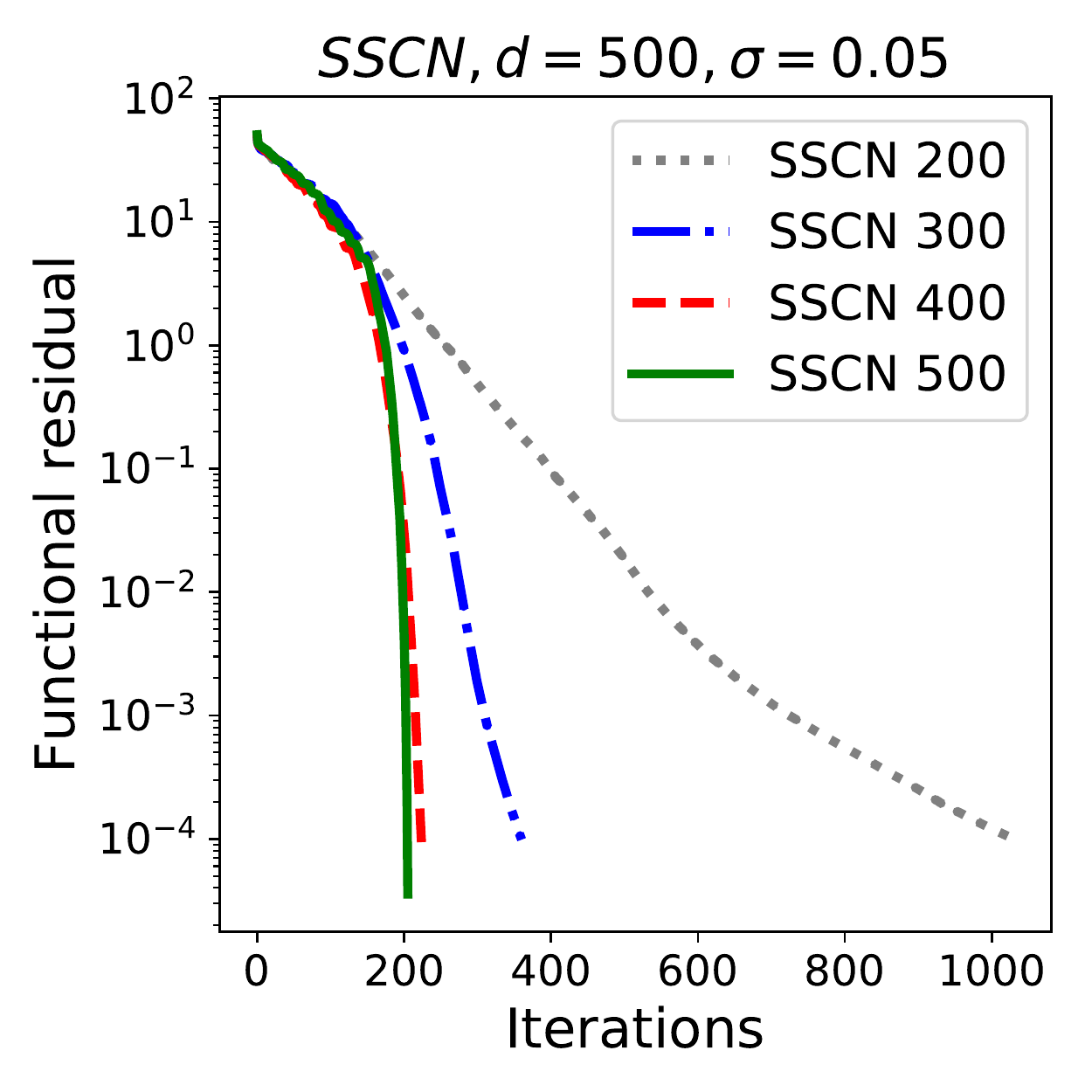}
	\end{minipage}
	\begin{minipage}{0.23\textwidth}
		\centering
		\includegraphics[width =  \textwidth ]{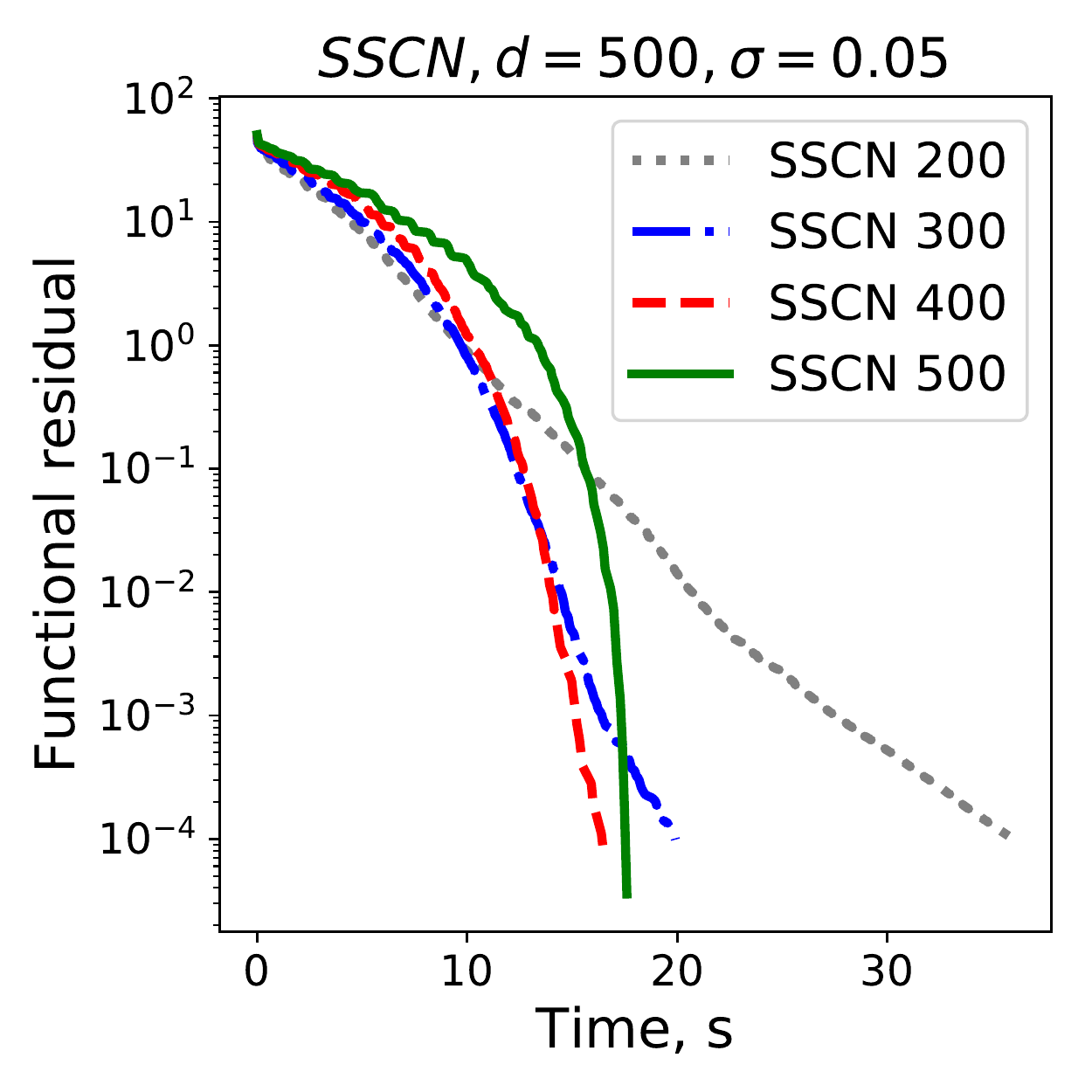}
	\end{minipage}

	\caption{SSCN and Coordinate Descent (CD) methods, minimizing Log-Sum-Exp function, $d = 500$.} 
	\label{fig:log_sum_exp_500}
\end{figure}

\begin{figure}[h!]
	\centering
	\begin{minipage}{0.23\textwidth}
		\centering
		\includegraphics[width =  \textwidth ]{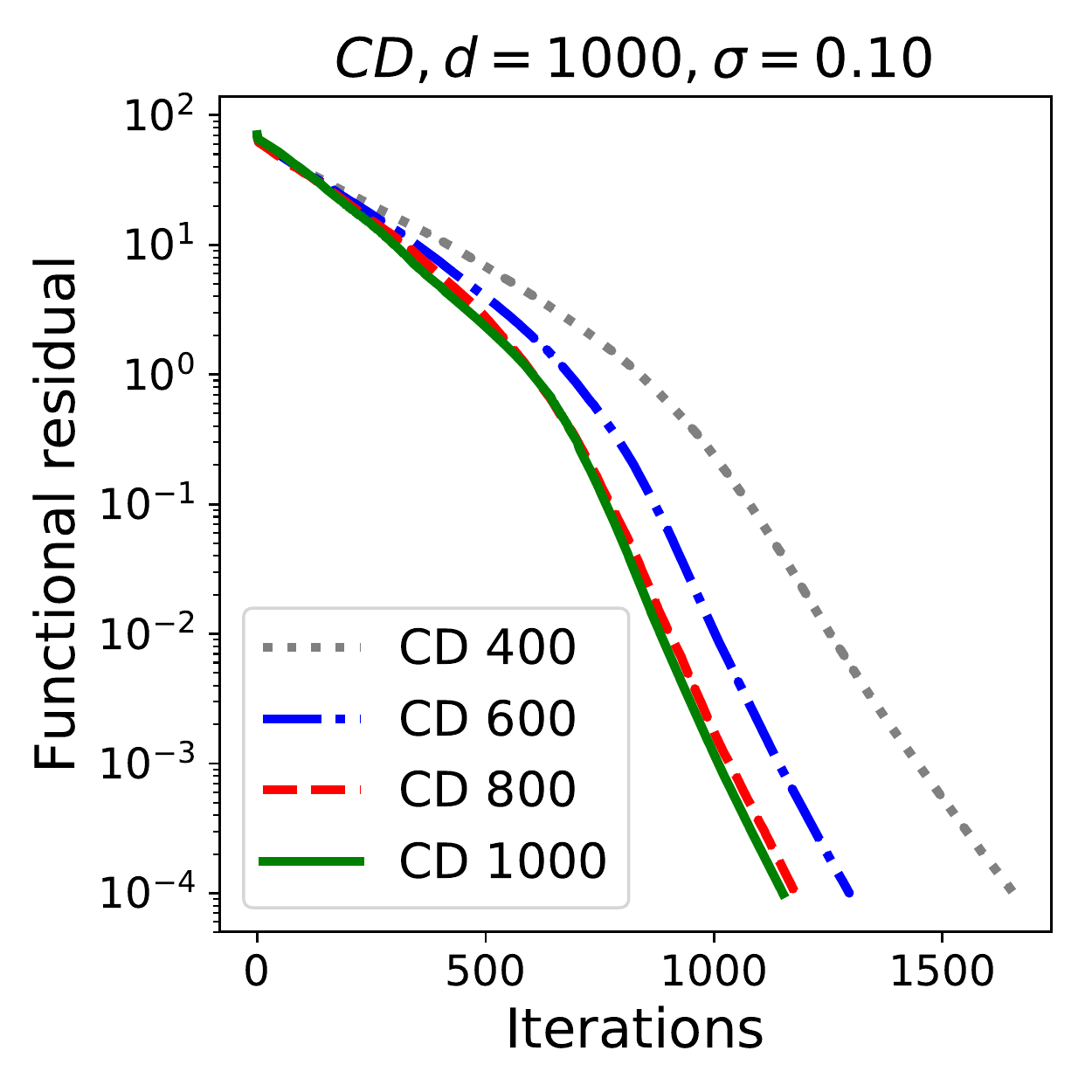}
	\end{minipage}
	\begin{minipage}{0.23\textwidth}
		\centering
		\includegraphics[width =  \textwidth ]{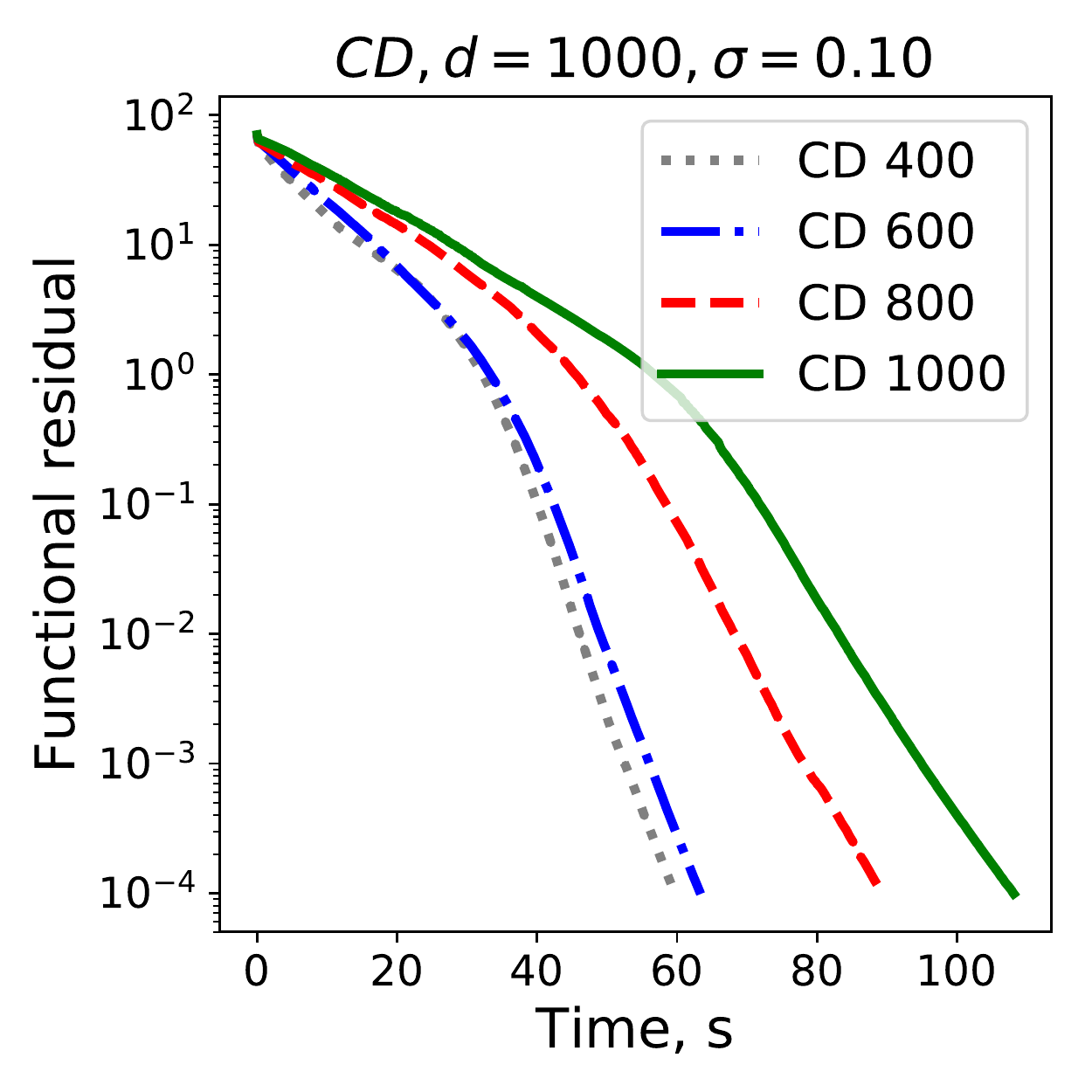}
	\end{minipage}
	\begin{minipage}{0.23\textwidth}
		\centering
		\includegraphics[width =  \textwidth ]{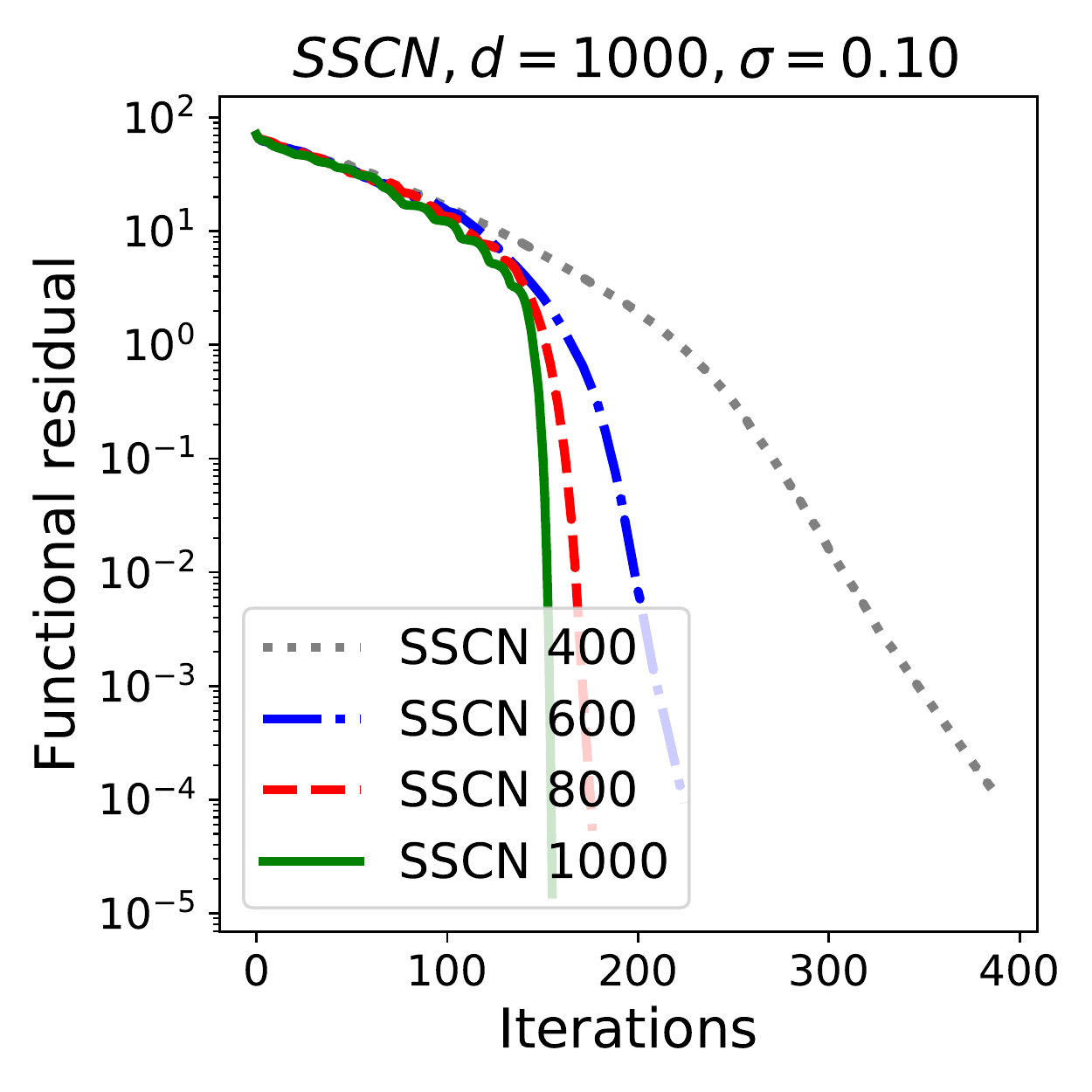}
	\end{minipage}
	\begin{minipage}{0.23\textwidth}
		\centering
		\includegraphics[width =  \textwidth ]{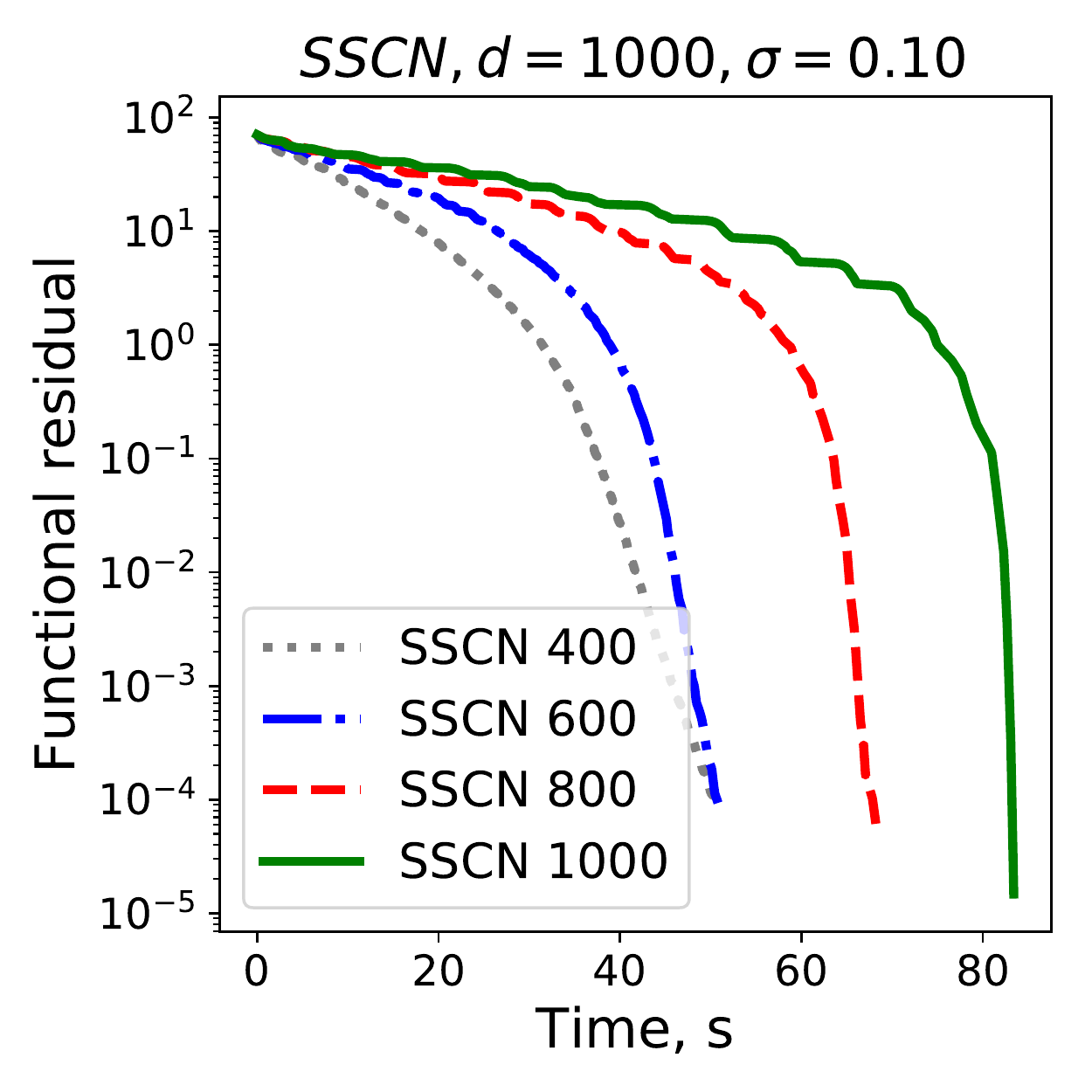}
	\end{minipage}
	
	\begin{minipage}{0.23\textwidth}
		\centering
		\includegraphics[width =  \textwidth ]{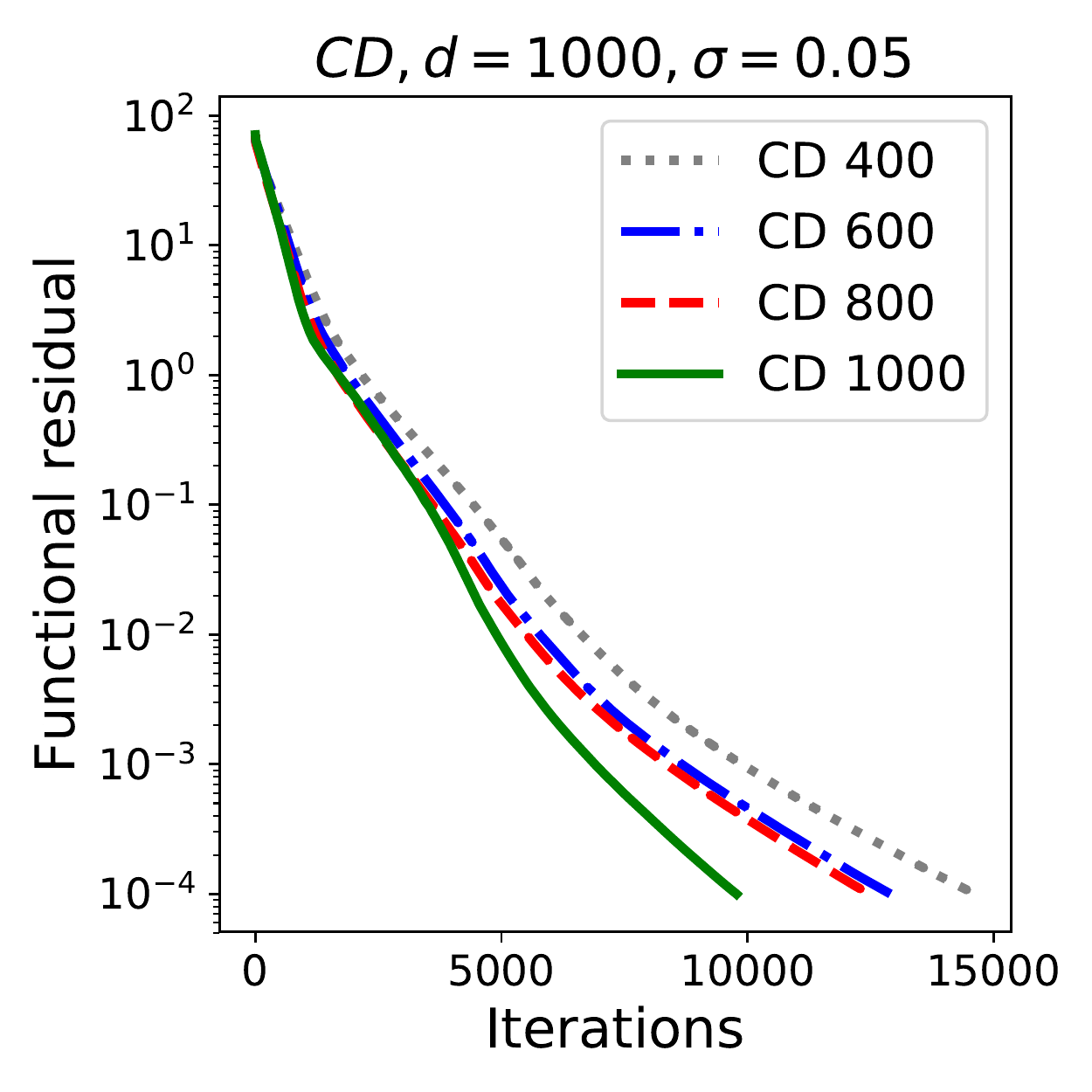}
	\end{minipage}
	\begin{minipage}{0.23\textwidth}
		\centering
		\includegraphics[width =  \textwidth ]{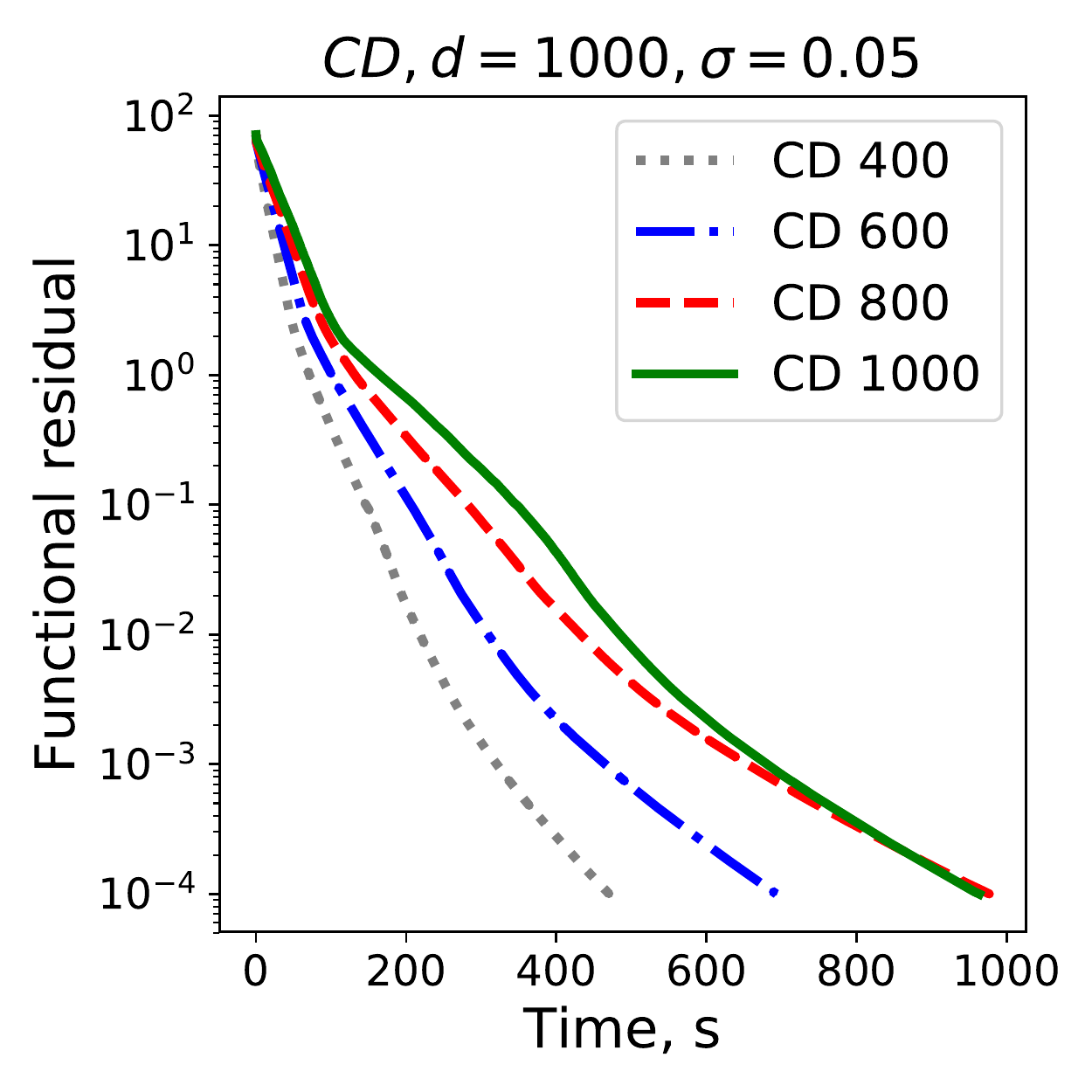}
	\end{minipage}
	\begin{minipage}{0.23\textwidth}
		\centering
		\includegraphics[width =  \textwidth ]{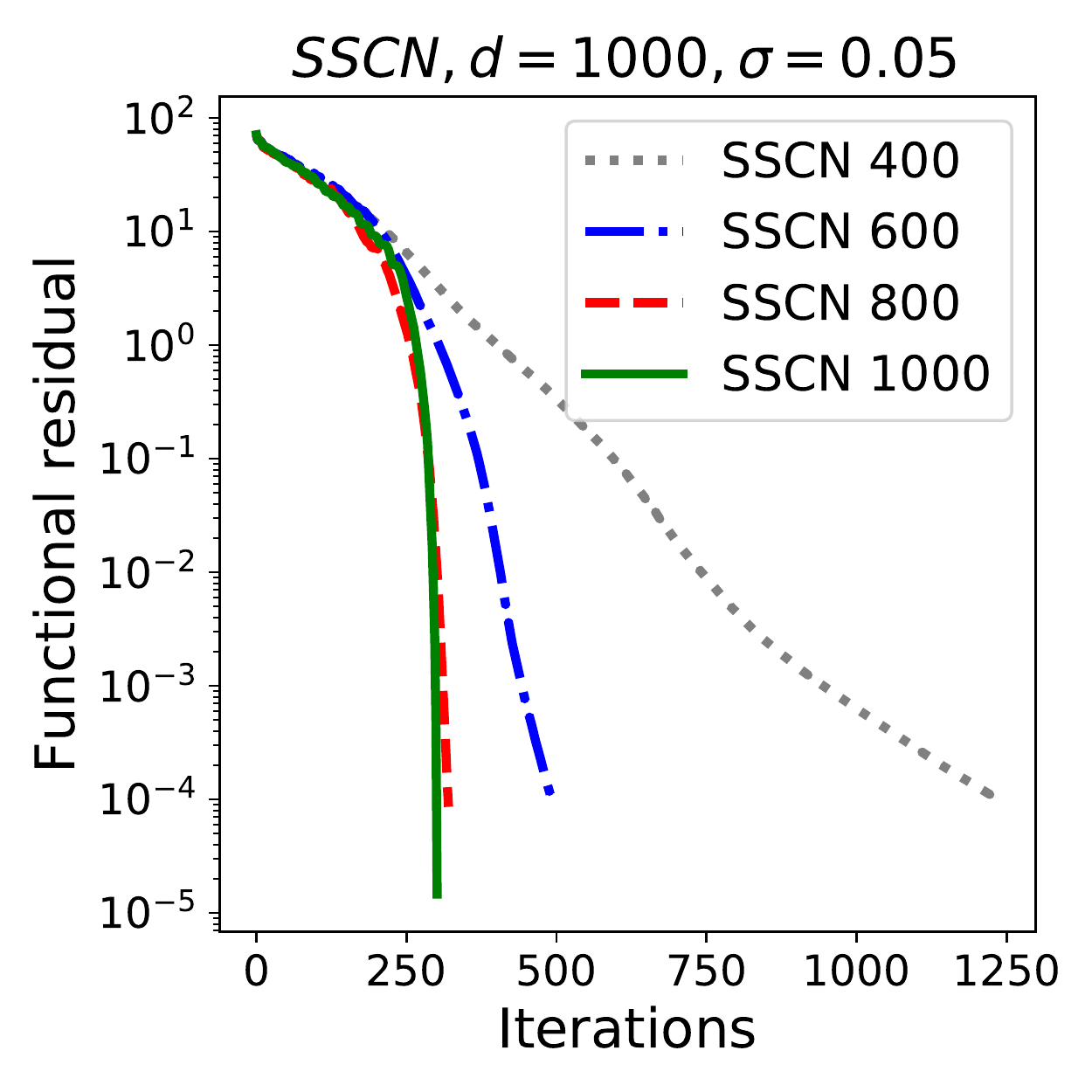}
	\end{minipage}
	\begin{minipage}{0.23\textwidth}
		\centering
		\includegraphics[width =  \textwidth ]{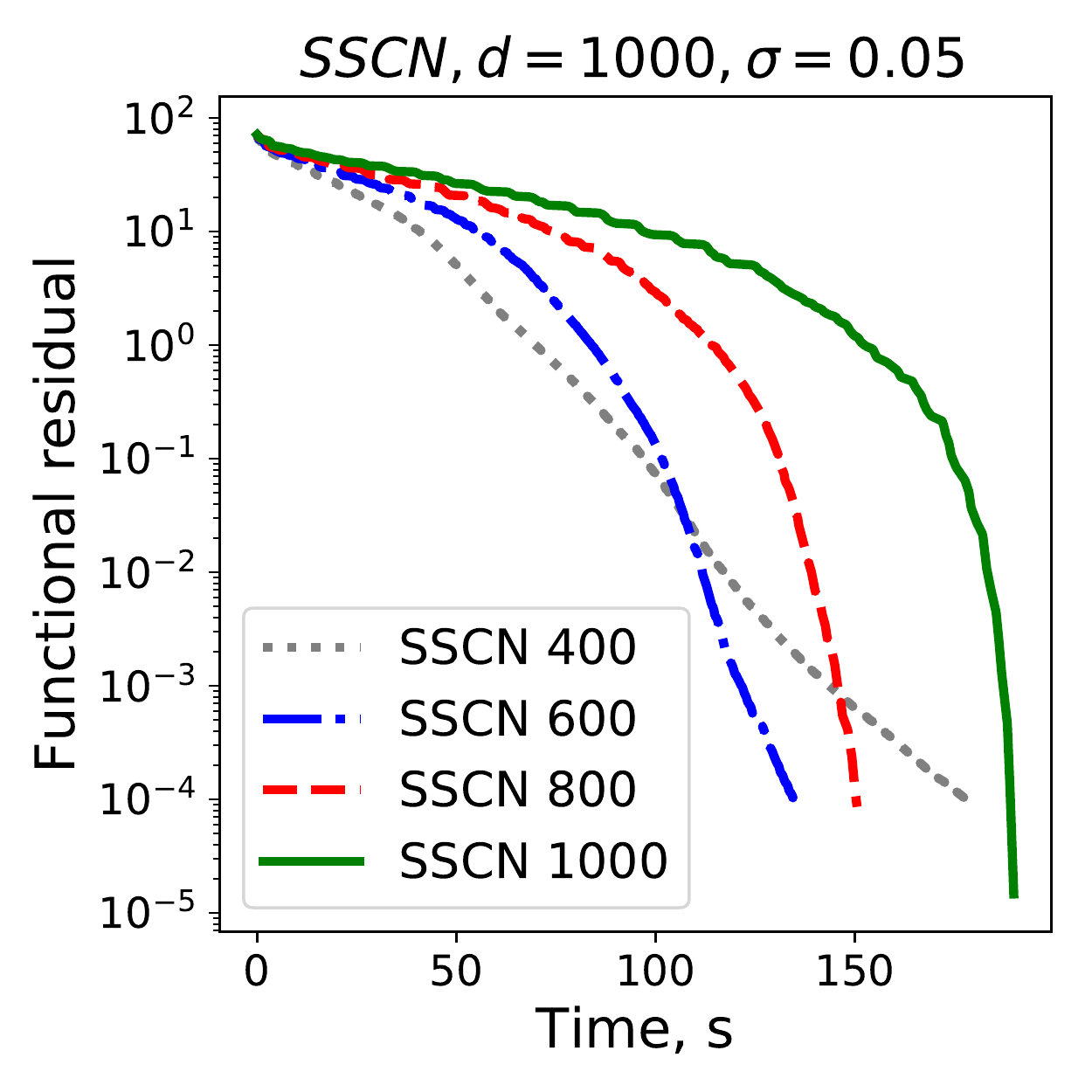}
	\end{minipage}

	\caption{SSCN and Coordinate Descent (CD) methods, minimizing Log-Sum-Exp function, $d = 1000$.} 
	\label{fig:log_sum_exp_1000}
\end{figure}

\section{Future Work}

Lastly, we list several possible extensions of our work.
\paragraph{Acceleration.} We believe it would be valuable to incorporate Nesterov's momentum into Algorithm~\ref{alg:crcd}. Ideally, one would like to get the global rate in between convergence rate of accelerated cubic regularized Newton~\citep{nesterov2008accelerating} and accelerated CD~\citep{allen2016even, nesterov2017efficiency}. On the other hand, the local rate (for strongly convex objectives) should recover accelerated sketch-and-project~\citep{tu2017breaking, gower2018accelerated}. If accelerated sketch-and-project is optimal (this is yet to be established), then accelerated SSCN (again, given that it recovers accelerated sketch-and-project) would be a locally optimal algorithm as well.

\paragraph{Non-separable $\psi$.} As mentioned in Section~\ref{sec:setup}, one should not hope for linear convergence of SSCN if $\psi$ is not separable, as the iterates can ``jump'' away from the optimum in such case. This issue has been resolved for first-order methods using control variates~\citep{hanzely2018sega}, resulting in SEGA. Therefore, the development of second-order SEGA remains an interesting open problem.

\paragraph{Inexact method.} SSCN is applicable in the setup, where function $f$ is accessible via zeroth-order oracle only. In such a case, for any $\mS \in \R^{\tau\times d}$ we can estimate $\nabla_{\mS} f(x)$ and $ \nabla_\mS^2 f(x)$ using $\cO(\tau^2)$ function value evaluations. However, since both $\nabla_{\mS} f(x)$ and $ \nabla_\mS^2 f(x)$ are only evaluated inexactly, a slight modification of our theory is required.

\paragraph{Non-uniform sampling.} Note that our local theory allows for arbitrary non-uniform distribution of $\mS$, which might be potentially exploited. At the same time, in some applications, it might be feasible to use a greedy selection rule for $\mS$ (our theory does not support that).

While developing optimal and implementable importance sampling for the local convergence is beyond the scope of this paper,\footnote{As this is still an open problem even for sketch-and-project~\citep{gower2015randomized}. }
we sketch several possible sampling strategies that might yield faster convergence.\footnote{This only applies to the local results as the global convergence requires some uniformity; see Assumption~\ref{as:uniform}.}

\begin{itemize}
\item Let $\Prob(\mS \in \{e_1, e_2, \dots, e_d \})=1$. If we evaluate the diagonal of the Hessian close to optimum (cost $\cO(nd)$ for linear models) and sample proportionally to it, we obtain local linear rate with leading complexity term $\frac{\Tr{\nabla^2 f(x^*)}}{\lambda_{\min}\nabla^2 f(x^*)}$. 

\item It is unclear how to design an efficient importance sampling for minibatch (i.e., $1<\E{\tau(\mS)}<d$) methods. Determinantal point processes (DPP)~\citep{rodomanov2019randomized, mutny2019convergence} were proposed to speed up SDNA from~\citep{qu2016sdna} (i.e., analogous CD with static matrix upper bound) -- we thus believe they might be applicable on our setting too. However, in such a case, one would need to evaluate the whole Hessian close to optimum, which is infeasible for applications where $d$ is large.

\item It is known that SDNA (see related literature) is faster than minibatch CD under the ESO assumption~\citep{qu2016coordinate1, qu2016coordinate2}. Therefore, we might instead apply minibatch importance sampling for ESO assumption from~\citep{hanzely2018accelerated} (which corresponds to optimizing the upper bound on iteration complexity). Using the mentioned sampling, we only require evaluating the diagonal of Hessian at some point close to optimum, which is of the same cost as computing the full gradient for linear models -- thus is feasible.

\item It is a natural question to ask whether one can speed up the convergence using a greedy rule instead of the random one. For standard CD, greedy rule was shown to have a superior iteration complexity to any randomized rule~\citep{nutini2015coordinate, karimireddy2018efficient}. For simplicity, consider case where $\Prob(\mS \in \{e_1, e_2, \dots, e_d \})=1$. Far from the optimum, (approximate) greedy rule at iteration $k$ chooses index $i = \argmax_{j} | \nabla_j f(x^k)|^{\frac32}M_{e_j}^{-\frac12}$. Close to optimum, if a diagonal of a Hessian was evaluated, (approximate) greedy index would be $\argmax_{j} | \nabla_j f(x^k)|^{2}\nabla_{j,j} f(x)^{-1}$. For linear models, both of the mentioned cases are implementable using the efficient neirest neighbour search~\citep{dhillon2011nearest} with sublinear complexity in terms of $d$.
\end{itemize}

\section*{Acknowledgements}

The work of the second and the fourth author was supported by ERC Advanced Grant 788368.

\newpage
\clearpage

\bibliography{literature}
\bibliographystyle{dinat}

\appendix
\clearpage 

\onecolumn 
\part*{Appendix}

\section{Table of Frequently Used Notation} \label{sec:notation_table}

\begin{table}[!h]
\caption{Summary of frequently used notation.}
\label{tbl:notation}
\begin{center}

\begin{tabular}{|c|l|c|}
\hline
\multicolumn{3}{|c|}{{\bf From main paper} }\\
\hline
 $F: \R^d \rightarrow \R$ & Objective function & \eqref{eq:problem}\\
  $f: \R^d \rightarrow \R$ & Smooth part of the objective & \eqref{eq:problem}\\
 $\psi: \R^d \rightarrow \R \cup \{ +\infty \}$ & Non-smooth part of the objective & \eqref{eq:problem}\\
 $x^*$ & Global optimum of \eqref{eq:problem} &\\
 $F^{*}$ & $ \eqdef F(x^{*})$, the optimum value of the objective &\\
  $\mS \in \R^{d, \tau(\mS)}$ & Random matrix sampled from distribution $\cD$& \eqref{eq:update_general}\\
    $S$ & Random  subset of $\{1,\dots ,d \}$& \eqref{eq:update_general}\\
$\mu$ & The constant of strong convexity& As. \ref{as:sc} \\
$M_{\mS}$ & Lipschitz constant of $\nabla^2 f(x)$ on the range of $\mS$ & \eqref{eq:MS_def} \\
$M$ & Lipschitz constant of $\nabla^2 f(x)$ on $\R^d$; $M = M_{ \mI^d}$ & \\
$L$  & Lipschitz constant of $\nabla f(x)$ on $\R^d$ & \\
$\mA_{\mS}$& $\eqdef \mS^\top \mA \mS \in \R^{\tau(\mS) \times \tau(\mS)}$, for 
a given matrix $\mA \in \R^{d\times d}$ &  \\
$\nabla_{\mS} f(x)$& $ \eqdef \mS^\top \nabla f(x)$ & \\
$\nabla^2_{\mS} f(x)$& $ \eqdef (\nabla^2 f(x))_{\mS} = \mS^\top \nabla^2 f(x) \mS$ & \\
$\mH_{\mS}(x)$&  $\eqdef  \nabla^2_{\mS} f(x) +\sqrt{ \frac{M_{\mS}}{2}} \| \nabla_{\mS} f(x)\|^{\frac12} \mI^{\tau(\mS)} $ & Lem. \ref{lem:decrease} \\
$\zeta$ & $   \eqdef \lambda_{\min} \left( \left(\nabla^2 f(x^*)\right)^{\frac12}  \E{\mS (\nabla^2_{\mS} f(x^*) )^{-1} \mS^\top}  \left(\nabla^2 f(x^*)\right)^{\frac12} \right)$ & \eqref{eq:sc_generalized}\\
$\mP^\mS$& $ \eqdef \mS \left(\mS^\top \mS\right)^{-1} \mS^\top$, the projection onto range of $\mS$ & Sec.~\ref{sec:setup} \\
$R$ & $\eqdef\sup\limits_{x \in \R^d} \Bigl\{  \|x - x^{*} \| \; : \; F(x) \leq F(x^{0})  \Bigr\}$  & \eqref{Rdef} \\ 
\hline
\hline
\multicolumn{3}{|c|}{{\bf Standard} }\\
\hline
$\E{\cdot}$ & Expectation & \\
$\Prob(\cdot)$ & Probability & \\
$\mI^q$ & Identity matrix in $\R^{q\times q}$& \\
$\lambda_{\max} (\cdot), \lambda_{\min}(\cdot)$ & Maximal eigenvalue, minimal eigenvalue  & \\
$\la \cdot, \cdot \ra$ & Scalar product of vectors: $\la x, y \ra \eqdef x^\top y$& \\
$\| \cdot \| $ & Standard Euclidean norm: $\|x\| \eqdef \sqrt{ \la x, x \ra}$ & \\
 $\| \cdot \|_\mB$ & Weighted Euclidean norm: $\| x \|_\mB \eqdef \sqrt{\langle \mB x ,x \rangle} $ &\\
$e_i$ & $i$-th vector from the standard basis in $\R^d$& \\
$e$ & Vector of ones in $\R^d$; i.e., $e \eqdef \sum_{i=1}^d e_i$  & \\
 \hline
\hline
\multicolumn{3}{|c|}{{\bf From Appendix} }\\
\hline
 $\lambda_f(x) $ & $ \eqdef \left(\nabla f(x)^\top \left(\nabla^2 f(x)\right)^{-1}\nabla f(x) \right)^\frac12$, Newton decrement & \eqref{eq:newton_decrement} \\
 $\level $& $ \eqdef \{x ; f(x)\leq f(x^0)\}$, sublevel set & \\
$ \Tr{\cdot}$ & Trace &Sec.~\ref{sec:exp_size_proof} \\
\hline
\end{tabular}

 \end{center}
\end{table}

\clearpage

\section{Missing Proofs and Lemmas From Section~\ref{sec:preliminaries}}

\subsection{Explicit update}
\begin{lemma}\label{lem:explicit_update}
Let $x^+ = \argmin_y \langle g', y-x \rangle + \frac{H'}{2} \|x-y\|^2+ \frac{M'}{6}\|x-y \|^3$, where $H',M'  >0$.
Then we have
\begin{equation}\label{eq:x_update_implicit}
    x^+= x- \frac{2g'}{ H' + \sqrt{{H'}^2 + 2M'\|g'\|}}
\end{equation}
\end{lemma}
\begin{proof}
By first-order optimality conditions we have $g'+ H'(x^+-x) + \frac{M'}{2}\|x^+-x\|(x^+-x) = 0$ which immediately yields
\begin{equation} \label{eq:lpdasadissio}
 x^+ = x - \frac{g'}{H' + \frac{M'}{2}\|x^+-x\|}. 
\end{equation}
Rearranging the terms and taking the norm we have $\frac{M'}{2}\|x^+-x \|^2 + H'\|x^+-x \| + \|g'\| = 0$. Solving the quadratic equation we arrive at
\[
\|x^+-x \|  = \frac{\sqrt{{H'}^2 + 2M'\|g'\|}-H'}{M'}.
\]
Plugging it back to~\eqref{eq:lpdasadissio}, we get~\eqref{eq:x_update_implicit}.
\end{proof}

\subsection{Proof of Lemma~\ref{lem:ub}}

\begin{eqnarray*}
&& D_f(x^+,x) - \frac{1}{2} (x^+-x)^\top \nabla^2 f(x)(x^+-x) \\
&& \qquad \qquad = 
\int_{0}^{1} \langle \nabla f(x+t(x^+-x)) - f(x), x^+-x \rangle \,dt - \frac{1}{2} (x^+-x)^\top \nabla^2 f(x)(x^+-x)
\\
&& \qquad \qquad = 
\int_{0}^{1} \int_{0}^{1} \langle t \nabla^2 f(x+st(x^+-x)),x^+-x, x^+-x \rangle \, ds\,dt- \frac{1}{2} (x^+-x)^\top \nabla^2 f(x)(x^+-x)
\\
&& \qquad \qquad = 
\int_{0}^{1} \int_{0}^{1} \langle t \nabla^2 f (x+st(x^+-x)) - \nabla^2 f(x),x^+-x,x^+-x \rangle \, ds\,dt 
\\ 
&& \qquad \qquad = 
\int_{0}^{1} \int_{0}^{1}  \int_{0}^{1}  \langle t^2 s \nabla^3 f (x+rst(x^+-x)),x^+-x,x^+-x,x^+-x \rangle \, dr\, ds\,dt.
\end{eqnarray*}
Using~\eqref{eq:update_general} we get
\begin{eqnarray*}
|f(x^+) - f(x) + \langle \nabla f(x),\mS h \rangle + \frac12 h^\top   \nabla^2_{\mS}  f(x) h |
& \stackrel{\eqref{eq:update_general} }{=}&   \left| \int_{0}^{1} \int_{0}^{1}  \int_{0}^{1}  \langle t^2 s \nabla^3 f (x+rst\mS h),\mS h,\mS h, \mS h \rangle \, dr\, ds\,dt \right| \\
& \refLE{eq:MS_def} &  
\int_{0}^{1} \int_{0}^{1}  \int_{0}^{1}   t^2 s M_{\mS} \|h_{\mS}\|^3 \, dr\, ds\,dt \\
&= &
 \frac{M_{\mS}}{6} \| h_{\mS}\|^3.
\end{eqnarray*}

\subsection{Proof of Lemma~\ref{lem:sharpness}}

First, $M \geq M_{\mS} $ is trivial. At the same time $M=M_{\mS}$ if $\nabla^3 f(x)$ is identity tensor always, which corresponds to $f(x) = \frac16\sum_{i=1}^d  x_i^3$. Therefore, the inequality is tight. 

To show sharpness of $M_{\mS} \geq \left(\frac{\tau}{d}\right)^{\frac32} M$, consider $f(x) = \frac16 (x^\top e)^3$. In this case, we have\footnote{By $[e] \in \R^{d\times d\times d}$ we mean third order outer product of vector $e$.} $\nabla^3 f(x) = [e]^3$  and $\mS=e_i$. In such case, $M = d^\frac32$ and $M_{\mS} = \tau^\frac32$. 

Note that $f$ is non-convex in both examples. However, it is is convex on a set where $x_i \geq 0$ for all $i$.

\section{Proofs for Section~\ref{sec:global}}

\subsection{Proof of Lemma~\ref{lem:exp_size}\label{sec:exp_size_proof}}
Let $\Tr{\mA}$ be a trace of square matrix $\mA$. We have
\begin{eqnarray*}
\E{\tau(\mS)} &=& \E{ \Tr{\mI^{\tau(\mS)}}} = \E{\Tr{ \mS^\top \mS \left(\mS^\top \mS\right)^{-1} } } = \E{\Tr{ \mS \left(\mS^\top \mS\right)^{-1}  \mS^\top }}\\
&=&
 \Tr{\E{ \mS \left(\mS^\top \mS\right)^{-1}  \mS^\top }} \stackrel{\eqref{eq:uniform_sampling}}{=} \Tr{\frac{\tau}{d} \mI^d} \\
 &=&\tau.
\end{eqnarray*}

\subsection{Proof of Lemma~\ref{lem:keylemma}}

For any $h' \in \R^d$ denote 
$$
\ba{rcl}
\Omega_{\mS}(x, h') & \Def &\displaystyle  \la \nabla f(x), \mP^\mS h'\ra
+ \frac{1}{2}\la \nabla^2 f(x)\mP^\mS h', \mP^\mS h' \ra
+ \frac{H}{6}\|\mP^\mS h' \|^3 + \psi(x +\mP^\mS h').
\ea
$$
Clearly, it holds
$$
\ba{rcl}
\min\limits_{h' \in \R^d} \Omega_{\mS}(x, h') 
& = & 
\min\limits_{h \in \R^{\tau(S)}} T_{\mS}(x, h).
\ea
$$
Therefore, for any fixed $y \in \R^d$ we have
$$
\ba{rcl}
F(x^{k + 1}) & \refLE{eq:coordinate_ub_full} &
f(x^k) +  \min\limits_{h' \in \R^d} \Omega_{\mS}(x^k, h')  
\;\; \leq \;\; 
f(x^k) +  \Omega_{\mS}(x^k; y - x^k).
\ea
$$
Therefore,
$$
\ba{rcl}
\E{ F(x^{k + 1}) \, | \, x^k}& \leq & \displaystyle f(x^k) + \E{ \Omega_{\mS}(x^k; y - x^k) } \\
\\
& = & \displaystyle f(x^k) + \frac{\tau}{d} \la \nabla f(x^k), y - x^k \ra
+ \E{  \frac{1}{2}\la \mP^\mS \nabla^2 f(x^k)\mP^\mS (y - x^k), y - x^k \ra }\\
\\
& \; &\displaystyle  \quad + \quad
\frac{M}{6}\E{ \| \mP^\mS (y - x^k) \|^3 }
+ \E{\psi(x^k +\mP^\mS (y-x^k))} .
\ea.
$$

Let us get rid of the expectations above. Firstly, we have 
\begin{eqnarray*}
\E{\psi(x +\mP^\mS (y-x^k))} 
&=&
 \E{\left \langle \psi'\left( \left( \mI^d - \mP^\mS \right)x^k +\mP^\mS y\right) ,e  \right \rangle }  \\
  & = &
  \E{\left \langle \left( \mI^d - \mP^\mS \right) \psi'\left( x^k\right) ,e  \right \rangle } + \E{\left \langle \mP^\mS  \psi'\left( y\right) ,e  \right \rangle } 
  \\
  &=&
 \left(1 - \frac{\tau}{d}\right)\psi(x^k) + \frac{\tau}{d} \psi(y) .
\end{eqnarray*}

For the cubed norm it can be estimated as follows
$$
\ba{rcl}
\E{ \| \mP^\mS h' \|^3 }& \leq & \displaystyle \|h'\| \cdot \E{\| \mP^\mS h'\|^2 }
\; = \; \frac{\tau}{d}\|h'\|^3, \qquad \forall h' \in \R^d.
\ea
$$
Lastly, note that 
\begin{eqnarray*}
\E{\mP^\mS \nabla^2 f(x^k) \mP^\mS}
& = & \E{\mP^\mS \left(\nabla^2 f(x^k)\right)^{\frac12}} \E{ \left(\nabla^2 f(x^k)\right)^{\frac12}\mP^\mS} 
\\
&& \qquad + \E{ \left(  \mP^\mS \left(\nabla^2 f(x^k)\right)^{\frac12} -   \E{\mP^\mS \left(\nabla^2 f(x^k)\right)^{\frac12}} \right)\left(  \mP^\mS \left(\nabla^2 f(x^k)\right)^{\frac12} -   \E{\mP^\mS \left(\nabla^2 f(x^k)\right)^{\frac12}} \right)^\top  } 
\\
&=& \frac{\tau^2}{d^2} \nabla^2 f(x^k) +\E{  \left( \mP^\mS - \frac{\tau}{d} \mI^d \right)\nabla^2 f(x^k) \left( \mP^\mS - \frac{\tau}{d} \mI^d \right)  } 
\\
&\preceq & \frac{\tau^2}{d^2} \nabla^2 f(x^k) + L \E{  \left( \mP^\mS - \frac{\tau}{d} \mI^d \right)^2  }
 \\
&= & \frac{\tau^2}{d^2} \nabla^2 f(x^k) + \frac{\tau(d-\tau)}{d^2}L\mI^d.
\end{eqnarray*}

Therefore, we conclude
$$
\ba{rcl}
\E{  F(x^{k + 1}) \, | \, x^k }& \leq &\displaystyle 
f(x^k) + \frac{\tau}{d} \la \nabla f(x^k), y - x^k \ra 
+ \frac{\tau(d-\tau)}{d^2}\cdot \frac{L}{2} \| y - x^k \|^2
\\
\\
& \; & \displaystyle \quad + \quad \frac{\tau^2}{d^2}  \cdot \frac{1}{2} \la \nabla^2 f(x^k)(y - x^k), y - x^k \ra
+ \frac{\tau}{d} \cdot \frac{M}{6}\|y - x^k\|^3 \\
\\
& \; & \displaystyle \quad + \quad \frac{\tau}{d} \psi(y) + \left(1 - \frac{\tau}{d}\right) \psi(x^k). \\
\ea
$$

Finally, by convexity and from Lipschitz continuity of the Hessian~\eqref{eq:coordinate_ub},
we have the following upper estimate:
$$
\ba{cl}
& \displaystyle  \la  \nabla f(x^k), y - x^k \ra + \frac{\tau }{d} \cdot \frac{1}{2} \la \nabla^2 f(x^k)(y - x^k), y - x^k \ra \\
\\
&\displaystyle   \qquad \qquad \; = \;
\frac{d - \tau}{d } \la \nabla f(x^k), y - x^k \ra
+ \frac{\tau }{d } \Bigl( 
\la \nabla f(x^k),y - x^k \ra + \frac{1}{2} \la \nabla^2 f(x^k)(y - x^k), y - x^k \ra
\Bigr) \\
\\
&\displaystyle   \qquad \qquad \; \leq \; 
\frac{d - \tau}{d } \Bigl( f(y) - f(x^k) \Bigr)
+ \frac{\tau }{d } \Bigl( 
f(y) - f(x^k) + \frac{M}{6}\|y - x^k\|^3
\Bigr) \\
\\
& \displaystyle \qquad \qquad  \; \leq \; f(y) - f(x^k) + \frac{M}{6}\|y - x^k\|^3.
\ea
$$
which completes the proof. \qed

\subsection{Proof of Theorem~\ref{thm:global_weakly}}

Let us denote the following auxiliary sequences:
$$
\ba{rcl}
a_{k} & \Def & k^2, \qquad A_{k} \; \Def \displaystyle  \; A_{0} + \sum\limits_{i = 1}^k a_i, \qquad k \geq 1,
\ea
$$
and
$$
\ba{rcl}
A_0 & \Def &\displaystyle   \frac{4}{3}\left( \frac{d}{\tau} \right)^3.
\ea
$$
Then, we have an estimate
\beq \label{A_k_grows}
\ba{rcl}
A_{k} \; = \; \displaystyle A_0 + \sum\limits_{i = 1}^k i^2 & \geq &\displaystyle 
A_0 + \int\limits_{0}^k x^2 dx \; = \; A_0 + \frac{k^3}{3}.
\ea
\eeq
Now, let us fix iteration counter $k \geq 0$ and set
$$
\ba{rcl}
\alpha_k & \Def &\displaystyle   \frac{d}{\tau} \frac{a_{k + 1}}{A_{k + 1}}
\quad \Leftrightarrow \quad 1 - \frac{\tau}{d} \alpha_k \; = \;  \frac{A_k}{A_{k + 1}}.
\ea
$$
Note that we have $\alpha_k \leq 1$ by the choice of $A_0$, since it holds
$$
\ba{rcl}
\displaystyle  \max\limits_{\xi \geq 0} \frac{\xi^2}{A_0 + \frac{\xi^3}{3}} & = & \displaystyle \frac{\tau}{d}.
\ea
$$ 
Let us plug $y \equiv \alpha_k x^{*} + (1 - \alpha_k) x^k$ into~\eqref{GlobalUpper}.
By convexity we obtain
$$
\ba{rcl}
\E{ F(x^{k + 1}) \, | \, x^k } & \leq & \displaystyle 
\Bigl(1 - \frac{\tau}{d} \Bigr) F(x^k) + \frac{\tau}{d}\alpha_k F^{*} + \frac{\tau}{d}(1 - \alpha_k) F(x^k) \\
\\
& \; & \displaystyle \quad + \quad 
\frac{\tau}{d}\biggl( \frac{d - \tau}{d } \frac{L \|x^k - x^{*}\|^2}{2}  \alpha_k^2
+ \frac{M \|x^k - x^{*}\|^3}{3} \alpha_k^3 \biggr) \\
\\
& = &\displaystyle 
\frac{A_k}{A_{k + 1}} F(x^k) + \frac{a_{k + 1}}{A_{k + 1}} F^{*}
+ \frac{d}{\tau} \frac{d - \tau}{d }  \frac{L \|x^k - x^{*}\|^2}{2} \left( \frac{a_{k + 1}}{A_{k + 1}} \right)^2
+ \left(\frac{d}{\tau}\right)^2 \frac{M \|x^k - x^{*}\|^3}{3} \left( \frac{a_{k + 1}}{A_{k + 1}}  \right)^3 \\
\\
& \leq & \displaystyle 
\frac{A_k}{A_{k + 1}} F(x^k) + \frac{a_{k + 1}}{A_{k + 1}} F^{*}
+ \frac{d - \tau}{\tau} \frac{LR^2}{2} \left( \frac{a_{k + 1}}{A_{k + 1}} \right)^2
+ \left( \frac{d}{\tau} \right)^2 \frac{M R^3}{3} \left( \frac{a_{k + 1}}{A_{k + 1}}  \right)^3.
\ea
$$

Therefore, for the residual $\delta_k \Def \E{ F(x^k) }- F^{*}$ we have the following bound
$$
\ba{rcl}
A_{k + 1} \delta_{k + 1} & \leq & \displaystyle A_k \delta_k 
+ \frac{d - \tau}{\tau}\frac{LR^2}{2} \frac{a_{k + 1}^2}{A_{k + 1}}
+ \left( \frac{d}{\tau} \right)^2 \frac{M R^3}{3} \frac{a_{k + 1}^3}{A_{k + 1}^2}, \quad k \geq 0.
\ea
$$
Summing up these inequalities for different $k$, we obtain
$$
\ba{rcl}
A_k \delta_k & \leq & \displaystyle A_0 \delta_0 
+ \frac{d - \tau}{\tau} \frac{L R^2}{2} \sum\limits_{i = 1}^k \frac{a_i^2}{A_i}
+ \left( \frac{d}{\tau} \right)^2 \frac{M R^3}{3} \sum\limits_{i = 1}^k \frac{a_i^3}{A_i^2}, \qquad k \geq 1.
\ea
$$
To finish the proof it remains to notice that
$$
\ba{rcl}
\displaystyle \sum\limits_{i = 1}^k \frac{a_i^2}{A_i} & \refLE{A_k_grows} & \displaystyle 
\sum\limits_{i = 1}^k \frac{i^4}{A_0 + \frac{1}{3} i^3 }
\; \leq \;
3 \sum\limits_{i = 1}^k i \; \leq \; 3 k^2,
\ea
$$
and
$$
\ba{rcl}
\displaystyle \sum\limits_{i = 1}^k \frac{a_i^3}{A_i^2} & \displaystyle \refLE{A_k_grows} & \displaystyle
\sum\limits_{i = 1}^k \frac{i^6}{(A_0 + \frac{1}{3}i^3 )^2} 
\; \leq \;
9 k.
\ea
$$

\qed

\subsection{Proof of Theorem~\ref{thm:global_strongly}}

Given that Assumption~\ref{as:sc} (strong convexity) is satisfied, the following inequality holds
\[
\ba{rcl}
\displaystyle \frac{\mu}{2}\|x - x^{*}\|^2 & \leq & \displaystyle F(x) - F^{*}, \qquad  \forall x \in \R^d,
\ea
\]
and thus we have a bound for the radius of level sets~\eqref{Rdef}:
$$
\ba{rcl}
R^2 & \leq & \displaystyle \frac{2}{\mu}(F(x^0) - F^{*}).
\ea
$$
Combining the above with~\eqref{GlobalConv} we obtain the following convergence estimate:
$$
\ba{rcl}
\displaystyle \E{ F(x^k) - F^{*} } & \leq & \displaystyle 
\left(
\frac{d - \tau}{\tau} \cdot \frac{18 L}{\mu k}
+ \left( \frac{d}{\tau} \right)^2 \cdot
\frac{18 M R }{\mu k^2}
+ \frac{1}{1 + \frac{1}{4}\bigl( \frac{\tau}{d} k \bigr)^3}
\right) \cdot 
\bigl( F(x^0) - F^{*} \bigr), \quad k \geq 1.
\ea
$$
Therefore, we get the linear decrease of the expected residual
$$
\ba{rcl}
\displaystyle \E{ F(x^k) - F^{*}  }
& \leq & \displaystyle \frac{1}{2}
\bigl( F(x^0) - F^{*} \bigr),
\ea
$$
as soon as the following three bounds for $k$ are all reached:
\begin{enumerate}
	\item $\frac{d - \tau}{\tau} \cdot \frac{18L}{\mu k} \leq \frac{1}{6}
	\quad \Leftrightarrow \quad
	k \geq 108 \frac{d - \tau}{\tau} \cdot \frac{L}{\mu}$.
	
	\item $\bigl( \frac{d}{\tau} \bigr)^2 \cdot \frac{18 M R}{\mu k^2} \leq \frac{1}{6}
	\quad \Leftrightarrow \quad
	k \geq \frac{d}{\tau} \sqrt{ 108  \frac{M R}{\mu}   }.
	$
	
	\item $\frac{1}{1 + \frac{1}{4}\bigl( \frac{\tau}{d} k^3 \bigr)^3} \leq \frac{1}{6}
	\quad \Leftrightarrow \quad
	k \geq \frac{d}{\tau} 20^{1/3}$.
\end{enumerate}
\qed

\section{Proofs for Section~\ref{sec:local} \label{sec:local_proofs}}
\subsection{Several technical Lemmas}
 It will be convenient to denote the Newton decrement as follows:
 
 \begin{equation} 
 \label{eq:newton_decrement}
 \lambda_f(x) \eqdef \left(\nabla f(x)^\top \left(\nabla^2 f(x)\right)^{-1}\nabla f(x) \right)^\frac12
 \end{equation}
 and a sublevel set of $x^0$ as $\level$; i.e., $\level \eqdef \{x ; f(x)\leq f(x^0)\} $.

\begin{lemma} (Local bounds)
Suppose that $x^0$ is such that $f(x^0)-f(x^*) \leq \varrho^4 \frac{2\left(\min_{x\in \level}\lambda_{\min} \nabla^2_{\mS} f(x)\right)^4}{LM_{\mS}^2 \|S \|^2}$ for some $\varrho>0$. Then, we have 
\begin{equation}\label{eq:ndajbdhahbj}
\sqrt{ \frac{M_{\mS}}{2}} \| \mS^\top \nabla f(x^k)\|^{\frac12} \mI^{\tau(\mS)} \preceq \varrho \nabla^2_{\mS}f(x^k).
\end{equation} Suppose further that $f(x^0) - f(x^*)\leq \varphi^2 \frac{ \mu\left(\lambda_{\min} \nabla^2_{\mS} f(x^*)\right)^2 }{2 M_{\mS}^2}$ for some $\varphi>0$. Then we have 
\begin{equation}\label{eq:dasjnlsdhjkd}
(1+\varphi)^{-1} \nabla^2_{\mS}f(x^*) \preceq \nabla^2_{\mS}f(x^k) \preceq (1+\varphi) \nabla^2_{\mS}f(x^*).
\end{equation}
Lastly, if $f(x^0)-f(x^*)\leq \omega^{-1}\left( \frac{2\mu^{\frac32}}{(1+\gamma^{-1})M} \right)$ where $\omega(y)\eqdef y - \log(1+y)$ and $\gamma>0$, we have
\begin{equation} \label{eq:dojasjodasj}
f(x^k) -f(x^*)\leq  \frac12 (1+\gamma) \lambda_f(x^k)^2.
\end{equation}
\end{lemma}

\begin{proof}
For the sake of simplicity, let $x = x^k$ and $\mS = \mS^k$ throughout this proof.
For the first part, we have
\begin{eqnarray*}
\sqrt{ \frac{M_{\mS}}{2}} \| \mS^\top \nabla f(x)\|^{\frac12} \mI^{\tau(\mS)}&\preceq& 
\sqrt{ \frac{M_{\mS}}{2}} \| \mS\|^\frac12 \| \nabla f(x)\|^{\frac12}  \mI^{\tau(\mS)} \\
& \preceq &
\sqrt{ \frac{M_{\mS}}{2}}  \| \mS\|^\frac12  2^{\frac14}L^{\frac14}\left(f(x^0)-f(x^*) \right)^\frac14  \mI^{\tau(\mS)}
\\
&\preceq&
\varrho\min_{x\in \level}\lambda_{\min} \nabla^2_{\mS} f(x)   \mI^{\tau(\mS)}
 \; \preceq \;
 \varrho \nabla^2_{\mS}f(x).
\end{eqnarray*}

For the second part, we have
\begin{eqnarray*}
\nabla^2_{\mS} f(x^k)-\nabla^2_{\mS}  f(x^*) &\preceq & 
M_{\mS} \|x^k-x^* \| \mI^{\tau(\mS)} \\
& \preceq &  M_{\mS}\sqrt{\frac{2 (f(x^k) - f(x^*))}{\mu}}  \mI^{\tau(\mS)} \\
& \preceq &   M_{\mS}\sqrt{\frac{2 (f(x^0) - f(x^*))}{\mu}}  \mI^{\tau(\mS)}
\\
&\preceq & 
\varphi  \nabla^2_{\mS} f(x^*).
\end{eqnarray*}
Therefore, we can conclude that $\nabla^2_{\mS} f(x) \preceq (1+\varphi) \nabla^2_{\mS} f(x^*)$. Analogously we can show $\nabla^2_{\mS} f(x^*) \preceq (1+\varphi) \nabla^2_{\mS} f(x)$ and thus~\eqref{eq:dasjnlsdhjkd} follows.

Lastly, if $f(x^0)-f(x^*)\leq \omega\left( \frac{2\mu^{\frac32}}{(1+\gamma^{-1})M} \right)$, then due to~\citep{nesterov2018lectures} we have
\[
\omega\left(\lambda_f(x^k)\right)\leq  f(x^k)-  f(x^*) \leq f(x^0)-  f(x^*) \leq  \omega\left( \frac{2\mu^{\frac32}}{(1+\gamma^{-1})M} \right)
\]
and thus $\lambda_f(x^k) \leq \frac{2\mu^{\frac32}}{(1+\gamma^{-1})M}$. Now~\eqref{eq:dojasjodasj} follows from Lemma~\ref{lem:selfconc} and Lemma~\ref{lem:sc2}.
\end{proof}

 \begin{lemma}\label{lem:selfconc}
 Function $f$ is $\frac{M}{\mu^{\frac32}}$ self-concordant.
 \end{lemma}
 \begin{proof}
 \[
 \frac{M}{\mu^{\frac32}} \|u\|^3_{\nabla^2 f(x)}  \geq  M\|u\|^3 \geq  \nabla^3 f(x)[u,u,u]
 \]
 \end{proof}
 
\begin{lemma}\label{lem:sc2}
Consider any $\gamma\in \R^+$ and suppose that $f$ is $\varsigma$ self-concordant. Then if  $\lambda_f(x)  < \frac{2}{(1+\gamma^{-1})\varsigma}$ we have
\begin{equation}\label{eq:sc2}
f(x) - f(x^*)\leq \frac12 \left(1+\gamma\right)  \lambda_f(x)^2   
\end{equation}
\end{lemma}

\begin{proof}
Define $\omega_*(z)\eqdef -z-\ln(1-z)$.
Note first that, $h(x)\eqdef \frac{\varsigma^2}{4} f(x)$ is 2 self concordant~\citep{nesterov2018lectures}. As a consequence, if $\lambda_h(x)<1$ we have~\citep{nesterov2018lectures}
\[
h(x) - h(x^*) \leq  \omega_*( \lambda_h(x)  ).
\] 
If further $\lambda_h(x) \leq \frac{1}{1+\gamma^{-1}}$ due to Lemma~\ref{lem:snjodanj}, we get
\[
\omega_*( \lambda_h(x)  ) \leq \left(1+\gamma\right) \frac {\lambda_h(x) ^2}{2}.
\] 
As $\lambda_h(x) = \frac{\varsigma}{2}\lambda_f(x)$, we get~\eqref{eq:sc2}.
\end{proof}

\begin{lemma}\label{lem:snjodanj}
Let $c\in \R^+$ and $0\leq y \leq \frac{1}{1+c}$. Then we have $\omega_*( y  ) \leq \left(1+\frac{1}{c}\right) \frac {y ^2}{2}.$
\end{lemma}
\begin{proof}
Clearly $\omega_*( y  ) = \sum_{i=2}^\infty \frac{y^i}{i}$ and thus function $\left(1+\frac{1}{c}\right) \frac {y ^2}{2} - \omega_*( y  ) $ is non-increasing for $y\geq 0$. Therefore, it suffices to check verify $\left(1+\frac{1}{c}\right) \frac {1}{2(1+c)^2} - \omega_*( \frac{1}{1+c} ) \geq 0$, which is an easy task for Mathematica, see Figure~\ref{fig:proof}.\begin{figure}[H]
\centering
\includegraphics[scale=0.5]{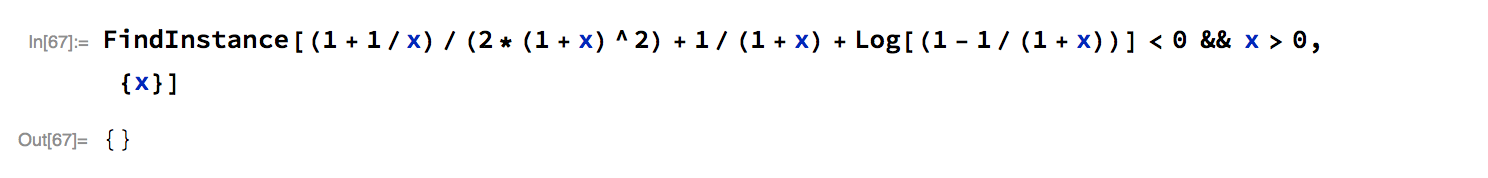}
\caption{Proof of $\left(1+\frac{1}{c}\right) \frac {1}{2(1+c)^2} - \omega_*( \frac{1}{1+c} ) \geq 0$ for all $c>0$.}
\label{fig:proof}
\end{figure}
\end{proof}

\subsection{Proof of Lemma~\ref{lem:decrease}}
Note that the update rule of SSCN yields immediately (using first-order optimality conditions)

\begin{equation} \label{eq:grad_equality}
-\mS^\top \nabla f(x)= \left( \nabla^2_{\mS}f(x) + \frac{1}{2} M_{\mS}\|x^+-x\| \mI^{\tau(\mS)}\right)\left(x^+-x\right)
\end{equation}

and therefore

\begin{eqnarray}
\left \| \mS^\top \nabla f(x) \right \|^{\frac12}
 &=&
 \left(
 \left(x^+-x\right)^\top\left( \nabla^2_{\mS}f(x) + \frac{1}{2} M_{\mS}\|x^+-x\| \mI^{\tau(\mS)} \right)^2\left(x^+-x\right) \right)^{\frac14}\nonumber
  \\
 & \geq & 
  \left(
 \left(x^+-x\right)^\top\left(  \frac{1}{2} M_{\mS}\|x^+-x\| \mI^{\tau(\mS)}\right)^2\left(x^+-x\right) \right)^{\frac14}\nonumber
 \\ 
 &=&
 \sqrt{\frac{M_{\mS}}{2}}\|x^+-x\|.
 \label{eq:stupid_bound}
\end{eqnarray}

Furthermore, taking dot product of~\eqref{eq:grad_equality} with $(x^+-x)$ yields
\[
\left\langle \mS^\top \nabla f(x), x^+-x\right\rangle+\left\langle  \nabla^2_{\mS}f(x)\left(x^+-x\right), x^+-x\right\rangle+\frac{1}{2} M_{\mS} \|x^+-x\|^3=0
\]
and thus 

\[
\begin{aligned} f(x)-f(x^+) 
&\stackrel{\eqref{eq:coordinate_ub}}{\geq}
\left\langle  \mS^\top \nabla f(x), x-x^+\right\rangle-\frac{1}{2}\left\langle   \nabla^2_{\mS}f(x) (x^+-x), x^+-x\right\rangle-\frac{M_{\mS}}{6}  \|x^+-x\|^3
 \\ 
 &=
\frac{1}{2}\left\langle   \nabla^2_{\mS} f(x) (x^+-x), x^+-x\right\rangle+\frac{M_{\mS}}{3}  \|x^+-x\|^3 
\\
 &\stackrel{(*)}{\geq}
 \frac12
\left(x^+-x\right)^\top\left( \nabla^2_{\mS}f(x) + \frac{1}{2} M_{\mS}\|x^+-x\| \mI^{\tau(\mS)} \right)\left(x^+-x\right)
\\
 &\stackrel{\eqref{eq:grad_equality}}{=}
 \frac12 \nabla f(x)^\top \mS \left( \nabla^2_{\mS}f(x) + \frac{1}{2} M_{\mS}\|x^+-x\| \mI^{\tau(\mS)} \right)^{-1}  \mS^\top\nabla f(x)
 \\
 &\stackrel{\eqref{eq:stupid_bound}}{\geq }
 \frac12 \nabla f(x)^\top \mS \left( \nabla^2_{\mS}f(x) +\sqrt{ \frac{M_{\mS}}{2}} \| \mS^\top \nabla f(x)\|^{\frac12} \mI^{\tau(\mS)} \right)^{-1}  \mS^\top \nabla f(x).
  \end{aligned}
\]
Above, in inequality $(*)$ we have used the fact that matrix $  \left( \nabla^2_{\mS}f(x) + \frac{1}{2} M_{\mS}\|x^+-x\| \mI^{\tau(\mS)}\right)$ is invertible since $f$ is strongly convex and thus $\nabla^2_{\mS}f(x) \succ 0$.

\subsection{Proof of Theorem~\ref{thm:local}}
First, suppose that $f(x^0)-f(x^*) \leq \varrho^4 \frac{2\left(\min_{x\in \level}\lambda_{\min} \nabla^2_{\mS} f(x)\right)^4}{LM_{\mS}^2\|\mS\|^{2}}$ for some $\varrho>0$. Using the fact that $\nabla^2_{\mS} f(x)$ is invertible ($\mS$ has full column rank and $\nabla^2 f(x) \succ 0$) we have
\begin{eqnarray}
\E{ \frac12 \| \mS^\top \nabla f(x^k)\|^2_{ \left( \mH(x^k) \right)^{-1} } } 
& \stackrel{\eqref{eq:ndajbdhahbj}}{\geq}&
\E{  \frac12 \nabla f(x)^\top {\mS}\left((1+\varrho) \nabla^2_{\mS}f(x)\right)^{-1} \mS^\top \nabla f(x)}
\nonumber
\\
&=&
 \frac{1}{2(1+\varrho)} \nabla f(x)^\top \E{ {\mS}\left( \nabla^2_{\mS}f(x)\right)^{-1}{\mS}^\top }\nabla f(x).
 \label{eq:dnjasnjaaa}
\end{eqnarray}

If further $f(x^0) - f(x^*)\leq  \varphi^2\frac{ \mu \left(\lambda_{\min}\nabla^2_{\mS} f(x^*)\right)^2 }{2 M_{\mS}^2}$ for some $\varphi>0$ we get

\begin{eqnarray}
\nonumber
\E{ \frac12 \| \mS^\top \nabla f(x^k)\|^2_{ \left( \mH(x^k) \right)^{-1} } } 
&\stackrel{\eqref{eq:dnjasnjaaa}}{ \geq}&
  \frac{ \nabla f(x)^\top \E{ {\mS}\left( \nabla^2_{\mS}f(x)\right)^{-1}{\mS}^\top }\nabla f(x)}{2(1+\varrho)}
  \\
  \nonumber
  &\stackrel{\eqref{eq:dasjnlsdhjkd}}{ \geq}&
    \frac{ \nabla f(x)^\top \E{ {\mS}\left( \nabla^2_{\mS}f(x^*)\right)^{-1}{\mS}^\top }\nabla f(x) }{2(1+\varrho)(1+\varphi)}
  \\
  \nonumber
  &\stackrel{\eqref{eq:sc_generalized}}{ \geq}&
  \frac{\nabla f(x)^\top \left(  \zeta \left(  \nabla^2f(x^*)\right)^{-1}\right) \nabla f(x) }{2(1+\varrho)(1+\varphi)}
    \\
    \label{eq:cjdinusvsibu}
  &\stackrel{\eqref{eq:dasjnlsdhjkd}}{ \geq}&
    \frac{\zeta \lambda_f(x)^2}{2(1+\varrho)(1+\varphi)^2}
\end{eqnarray}

Lastly, if if $f(x^0)-f(x^*)\leq \omega^{-1}\left( \frac{2\mu^{\frac32}}{(1+\gamma^{-1})M} \right)$ where $\omega(y)\eqdef y - \log(1+y)$ and $\gamma>0$, we get

\begin{eqnarray*}
\E{ \frac12 \| \mS^\top \nabla f(x^k)\|^2_{ \left( \mH(x^k) \right)^{-1} } } 
  &\stackrel{\eqref{eq:cjdinusvsibu}}{ \geq}&
      \frac{\zeta \lambda_f(x)^2}{2(1+\varrho)(1+\varphi)^2}
      \\
  &\stackrel{\eqref{eq:dojasjodasj}}{ \geq}&
    \frac{\zeta (f(x)-f(x^*))}{(1+\varrho)(1+\varphi)^2(1+\gamma)}
\end{eqnarray*}

and thus~\eqref{eq:local_rate} follows. In particular for any $\varrho, \varphi, \gamma>0$, we can choose 
\[
\delta = \min \left\{  \varrho^4 \frac{2\left(\min_{x\in \level}\lambda_{\min} \nabla^2_S f(x)\right)^4}{LM_S^2},  \varphi^2\frac{ \mu \left(\lambda_{\min}\nabla^2_S f(x^*)\right)^2 }{2 M_S^2} ,\omega^{-1}\left( \frac{2\mu^{\frac32}}{(1+\gamma^{-1})M} \right) \right\}
\]
and 
\[
\varepsilon = 1 -\frac{1}{(1+\varrho)(1+\varphi)^2(1+\gamma)}.
\]
\qed

\end{document}

%% file: everything.tex
\usepackage{amsfonts}
\usepackage{amsmath}
\usepackage{amsthm}
\usepackage{amssymb} 

\usepackage{mdframed} 
\usepackage{thmtools}
\usepackage{enumitem}

 \usepackage{xcolor}
\usepackage{color}
\usepackage{graphicx}

\newcommand{\R}{\mathbb{R}} 

\newcommand{\cD}{{\cal D}}

\newcommand{\cO}{{\cal O}}

\newcommand{\mA}{{\bf A}}
\newcommand{\mB}{{\bf B}}

\newcommand{\mH}{{\bf H}}
\newcommand{\mI}{{\bf I}}

\newcommand{\mL}{{\bf L}}

\newcommand{\mP}{{\bf P}}

\newcommand{\mS}{{\bf S}}

\newcommand{\mU}{{\bf U}}

\usepackage[colorinlistoftodos,bordercolor=orange,backgroundcolor=orange!20,linecolor=orange,textsize=scriptsize]{todonotes}

\newcommand{\eqdef}{:=} 

\DeclareMathOperator{\dom}{dom}         
\DeclareMathOperator{\argmin}{argmin}        
\DeclareMathOperator{\argmax}{argmax}        

\providecommand{\Range}[1]{{\rm Range}\left( #1\right)}

\newcommand{\Prob}{\mathbb{P}}

\newcommand{\E}[1]{\mathbb{E}\left[#1\right] } 
\newcommand{\EE}[2]{{\bf E}_{#1}\left[#2\right] }

\declaretheorem[within=section]{definition}
\declaretheorem[sibling=definition]{theorem}

\declaretheorem[sibling=definition]{assumption}
\declaretheorem[sibling=definition]{corollary}
\declaretheorem[sibling=definition]{lemma}

\theoremstyle{remark}
\newtheorem{example}{Example} 
\newtheorem{remark}{Remark} 

			\newcommand{\level}{\chi^0 }
			\newcommand{\Tr}[1] {\mathrm{Tr}\left( {#1} \right)}

\def\ba{\begin{array}}
\def\ea{\end{array}}
\def\beann{\begin{eqnarray*}}
\def\eeann{\end{eqnarray*}}
\def\bea{\begin{eqnarray}}
\def\eea{\end{eqnarray}}

\def\BT{\begin{theorem}}
\def\ET{\end{theorem}}
\def\BL{\begin{lemma}}
\def\EL{\end{lemma}}
\def\BC{\begin{corollary}}
\def\EC{\end{corollary}}
\def\BE{\begin{example}}
\def\EE{\end{example}}
\def\BD{\begin{definition}}
\def\ED{\end{definition}}
\def\BR{\begin{remark}}
\def\ER{\end{remark}}
\def\BAS{\begin{assumption}}
\def\EAS{\end{assumption}}
\def\BI{\begin{itemize}}
\def\EI{\end{itemize}}

\def\BMP{\begin{minipage}{9.5cm}}
\def\EMP{\end{minipage}}
\def\MPT{\begin{minipage}{11.5cm}}
\def\EPT{\end{minipage}}

\def\Def{\stackrel{\mathrm{def}}{=}}

\def\dom{{\rm dom \,}}
\def\beq{\begin{equation}}
\def\eeq{\end{equation}}
\def\R{\mathbb{R}}

\newcommand{\refLE}[1]{\ensuremath{\stackrel{(\ref{#1})}{\leq}}}

\def\la{\langle}
\def\ra{\rangle}